\documentclass[sn-mathphys.bst]{sn-jnl}

\usepackage{tikz}%
\usepackage{anyfontsize}
\usepackage{cite}
\usepackage{graphicx}%
\usepackage{multirow}%
\usepackage{amsmath,amssymb,amsfonts}%
\usepackage{amsthm}%
\usepackage{mathrsfs}%
\usepackage[title]{appendix}%
\usepackage{xcolor}%
\usepackage{textcomp}%
\usepackage{manyfoot}%
\usepackage{booktabs}%
\usepackage{algorithm}%
\usepackage{algorithmicx}%
\usepackage{algpseudocode}%
\usepackage{listings}%
\usepackage{slashed}%

\newcommand{\RR}{\mathbb{R}}

\newcommand{\NN}{\mathbb{N}}
\newcommand{\ra}{\rightarrow}
\newcommand{\sph}{\mathbb{S}}
\newcommand{\mc}{\mathcal}

\newcommand{\pa}{\partial}

\newcommand{\lesa}{\lesssim}
\newcommand{\geqa}{\gtrsim}

\newcommand{\bs}{\backslash}

\newcommand{\ol}{\overline}
\newcommand{\LT}{\frac{\Lambda}{3}}

\newcommand{\mr}{\mathring}
\newcommand{\de}{\textnormal{d}}
\newcommand{\rc}{r_{\mathcal{C}}}
\newcommand{\rh}{r_{\mathcal{H}}}
\newcommand{\roc}{\ol{r_{\mathcal{C}}}}
\newcommand{\tg}{\tilde{g}}
\renewcommand{\thefigure}{\thesection.\arabic{figure}}
\numberwithin{figure}{section}
\renewcommand{\thetable}{\thesection.\arabic{table}}
\numberwithin{table}{section}

\numberwithin{equation}{section}
\newtheoremstyle{plainspace}%
{3pt}
{3pt}
{\it}
{\parindent}
{\bfseries}
{.}
{.5em}
{}
\theoremstyle{plainspace}%
\newtheorem{theorem}{Theorem}[section]
\newtheorem{proposition}[theorem]{Proposition~}%
\newtheorem{corollary}[theorem]{Corollary}%
\newtheorem{lemma}[theorem]{Lemma}
\theoremstyle{remark}%
\newtheorem{remark}[theorem]{Remark}%
\newtheorem{definition}[theorem]{Definition}%

\begin{document}

\title[Article Title]{Linear waves on the expanding region of Schwarzschild-de Sitter spacetimes: forward asymptotics and scattering from infinity}


\author{\fnm{Louie} \sur{Bernhardt}}

\affil{\orgdiv{School of Mathematics and Statistics}, \orgname{University of Melbourne}, \orgaddress{\city{Parkville}, \postcode{3010}, \state{VIC}, \country{Australia}}\email{lbernhardt@student.unimelb.edu.au}}


\abstract{We study solutions to the linear wave equation on the cosmological region of Schwarzschild-de Sitter spacetimes. We show that all sufficiently regular finite-energy solutions to the linear equation possess a particular finite-order asymptotic expansion near the future boundary. Specifically, we prove that several terms in this asymptotic expansion are identically zero. This is accomplished with new weighted higher-order energy estimates that capture the global expansion of the cosmological region. Furthermore we prove existence and uniqueness of scattering solutions to the linear wave equation on the expanding region. Given two pieces of scattering data at infinity, we construct solutions that have the same asymptotics as forward solutions. The proof involves constructing asymptotic solutions to the wave equation, as well as a new weighted energy estimate that is suitable for the backward problem. This scattering result extends to a large class of expanding spacetimes, including the Kerr de Sitter family.}

\maketitle
\section{Introduction}
In this paper we prove results relating to the asymptotic behaviour of solutions to the linear wave equation
\begin{equation}\label{eq:waveequation}
\square_g \psi = 0
\end{equation}
on $3+1$-dimensional Schwarzschild-de Sitter spacetimes. For an introduction to the global geometry of Schwarzschild-de Sitter, see \cite{LR77}. We focus on the \textit{expanding region} which we denote by $\mc{R}^+$; see Fig. \ref{fig:globalgeometrySdS}. The expanding region is bounded in the past by two cosmological horizons $\mc{C}^+,\ol{\mc{C}}^+$, and to the future by the future boundary $\Sigma^+$.  We introduce on $\mc{R}^+$ the standard coordinate chart $(r,t,\theta,\varphi) \in (r_\mc{C},\infty)\times\RR\times(0,\pi)\times(0,2\pi)$ and metric
$$g_{\Lambda,m} = -\frac{1}{\LT r^2 - 1 + \frac{2m}{r}}\de r^2 + \Big(\LT r^2 -1 +\frac{2m}{r}\Big) \de t^2 + r^2\>\de \theta^2 + r^2\sin^2\theta \>\de \varphi^2,$$
where $r_{\mc{C}}$ denotes the largest positive root of $\LT r^2 - 1 + \frac{2m}{r}$. For more properties of the expanding region, see \cite{Sch15}.

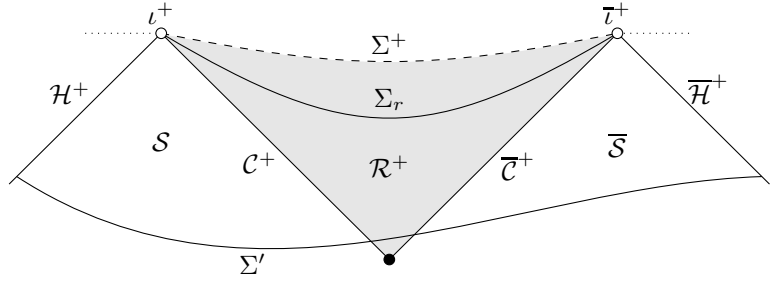
\begin{figure}[t]
\centering
\begin{tikzpicture}
\filldraw[fill=gray!20,draw=white] (0,3) -- (3,0) -- (6,3) .. controls (3.5,2.5) and (2.5,2.5) .. cycle;
\draw (0,3) -- (3,0) -- (6,3);
\draw[dashed] (0,3) .. controls (2.5,2.5) and (3.5,2.5) .. (6,3);
\draw[dotted] (-1,3) -- (0,3);
\draw[dotted] (6,3) -- (7,3);
\draw (-2,1) -- (0,3);
\draw (6,3) -- (8,1);
\filldraw  (3,0) circle (0.07cm);
\draw (-1.2,1.95) node[above]{$\mc{H^+}$};
\draw (7.2,1.95) node[above]{$\ol{\mc{H}}^+$};
\draw (1.3,1.55) node[below]{$\mc{C}^+$};
\draw (4.7,1.55) node[below]{$\ol{\mc{C}}^+$};
\draw (3,2.6) node[above]{$\Sigma^+$};
\draw (3,1.5) node[below]{$\mc{R}^+$};
\draw (0,1.8) node[below]{$\mc{S}$};
\draw (6,1.8) node[below]{$\ol{\mc{S}}$};
\draw (0,3) .. controls (2.5,1.5) and (3.5,1.5) .. (6,3);
\draw (3,1.85) node[above]{$\Sigma_r$};
\draw (0,3) node[above]{$\iota^+$};
\draw (6,3) node[above]{$\ol{\iota}^+$};
\filldraw[fill=white] (0,3) circle (0.07cm);
\filldraw[fill=white] (6,3) circle (0.07cm);
\draw (-1.9,1.1) .. controls (1.3,-1) and (4.7,1) .. (7.9,1.1);
\draw (1.2,0.2) node[below]{$\Sigma'$};

\end{tikzpicture}
\caption{Geometry of the expanding region $\mc{R}^+$ of Schwarzschild-de Sitter, shaded in gray here, which is bounded in the past by cosmological horizons $\mc{C}^+,\ol{\mc{C}}^+$, and to the future by $\Sigma^+$. One can also consider the Cauchy problem for the linear wave equation on the region $\mc{S}\cup\mc{R}^+\cup\ol{\mc{S}}$, where one sets initial data on a Cauchy surface like $\Sigma'$.}
\label{fig:globalgeometrySdS}
\end{figure}
A natural foliation of $\mc{R}^+$ is by the level sets of constant radius which we denote by $\Sigma_r$. These are spacelike on the interior of the expanding region. We note that the hypersurfaces $\Sigma_r$ and future boundary $\Sigma^+$ are all diffeomorphic to $\RR\times\sph^2$, terminating at $\iota^+$ on one end, and at $\ol{\iota}^+$ on the other. As such we will often identify functions on $\Sigma_r$ or $\Sigma^+$ with functions on the cylinder in this paper. Define the following energy:\footnote{We will discuss the geometric interpretation of this energy in Section~\ref{sec:notation}.}
\begin{equation}\label{eq:naturalenergy}
E[\psi](r) = \frac{1}{2}\int_{\Sigma_r} \phi^{-1}[(\pa_r\psi)^2 + |\ol{\nabla}\psi|^2]\de \mu_{\ol{g}_r},
\end{equation}
where
$$\phi = \frac{1}{\sqrt{\LT r^2 -1 + \frac{2m}{r}}}$$
is the \emph{lapse} for the $\Sigma_r$-foliation, $\ol{\nabla}\psi$ is the covariant derivative restricted to $\Sigma_r$, and $\de \mu_{\ol{g}_r}$ is the induced volume form on $\Sigma_r$.

We now present the two main theorems of this paper. The first theorem states that there is a \textit{finite-order asymptotic expansion} in powers of $r^{-1}$ for forward solutions to the linear wave equation on Schwarzschild-de Sitter.
\begin{theorem}[Forward asymptotics of linear waves on \texorpdfstring{$\mc{R}^+$}{the expanding region}]\label{thm:forwardasymptotics}
Let $\psi$ be a finite energy solution to \eqref{eq:waveequation} on $\mc{R}^+$, in the sense that on a level set $\Sigma_{r_0} \subset \mc{R}^+$ we have
\begin{equation}\label{eq:fafiniteenergy}
\sum_{k=0}^2 E[\pa_r^k \psi](r_0) < \infty.
\end{equation}
Then the following limits exist, where we emphasise that the limits are taken with respect to the norm of the homogeneous Sobolev space $\dot H^1(\RR\times\sph^2)$:
\begin{equation}\label{eq:asymptoticlimits}
\psi_0 = \lim_{r\ra\infty}\psi,\quad \psi_2 = \lim_{r\ra\infty}\Big[r^2(\psi-\psi_0)\Big],\quad \psi_3 = \lim_{r \ra \infty}\Big[r^3\Big(\psi-\psi_0-\frac{\psi_2}{r^2}\Big)\Big].
\end{equation}
\end{theorem}
It immediately follows from Theorem \ref{thm:forwardasymptotics} that such solutions $\psi$ have asymptotics of the form
\begin{equation}\label{eq:asymptoticsintro}
\psi \sim \psi_0 + \frac{\psi_2}{r^2} + \frac{\psi_3}{r^3}
\end{equation}
for large $r$.
\begin{remark}This extends a result from \cite{Sch15}, where it was shown that $\psi$ has a nonzero limit on $\Sigma^+$ with respect to the $\dot{H}^1$ norm, proving the existence of just the $\psi_0$ term in the asymptotic expansion \eqref{eq:asymptoticsintro}. Similar asymptotic expansions to \eqref{eq:asymptoticsintro} were given in \cite{Vas10} for smooth solutions to the linear wave equation on a class of de Sitter like spacetimes that include Schwarzschild-de Sitter. In particular it was shown that on this class of spacetimes, smooth solutions $\psi$ to the linear wave equation have an asymptotic expansion of the form
\begin{equation}\label{eq:asymptoticsvasy}
\psi \sim \psi_0 + \frac{\psi_1}{r} + \frac{\psi_2}{r^2} + \frac{\log r\> \psi_{3,1}}{r^3} + \frac{\psi_3}{r^3}.
\end{equation}
However in \cite{Vas10} it is not specified whether all terms in this expansion are generically nonzero.\footnote{It is shown in \cite{Vas10} that if the metric of a de Sitter-like spacetime has a Taylor expansion near $\Sigma^+$ that contains only even powers of $r^{-1}$, then the $r^{-3}\log r\>\psi_{3,1}$ term is identically zero. However this does not apply to Schwarzschild-de Sitter due to the black hole mass, as one can see for instance that 
$$g_{tt} = \frac{\Lambda}{3}r^2-1+\frac{2m}{r} +O(r^{-2}).$$} Importantly, Theorem~\ref{thm:forwardasymptotics} implies that the $r^{-1}\psi_1$ and $r^{-3}\log r \>\psi_{3,1}$ terms in \eqref{eq:asymptoticsvasy} are absent for linear waves on Schwarzschild-de Sitter.\end{remark}
\begin{remark}
In the context of the Cauchy problem, where one poses data on a Cauchy surface such as $\Sigma'$ in Figure \ref{fig:globalgeometrySdS}, the bound \eqref{eq:fafiniteenergy} holds for sufficiently regular finite-energy solutions to the linear wave equation. This follows in particular from the works \cite{BH08,DR07,Dy11} which studied solutions to the linear wave equation on the \emph{static region} $\mc{S}$ of Schwarzschild-de Sitter (and the related stationary region of Kerr-de Sitter). For more recent work, see for example \cite{Ma23}. In fact, it is known that the static region $\mc{S}$ is nonlinearly stable as a solution to Einstein's vacuum equations with positive cosmological constant. This was proved in the seminal work \cite{HV18}, see also \cite{Fang23}.
\end{remark}

Moreover we consider the scattering problem for the linear wave equation on Schwarzschild-de Sitter. We show that $\psi_0,\psi_3$ in the asymptotic expansion \eqref{eq:asymptoticsintro} constitute the correct notion of scattering data, so that given the functions $\psi_0,\psi_3$ there exists a unique solution to \eqref{eq:waveequation} on $\mc{R}^+$.
\begin{theorem}[Scattering of linear waves on \texorpdfstring{$\mc{R}^+$}{the expanding region}]\label{thm:scatteringwithouthorizon}
Let $\psi_0\in H^4(\RR\times\sph^2),\psi_3\in H^2(\RR\times\sph^2)$. Then there exists a unique solution $\psi$ to the linear wave equation on the expanding region of Schwarzschild-de Sitter such that for all $r \geq r_{\mc{C}}+\epsilon$,
\begin{equation}\label{eq:scatteringenergybound}
E[\psi](r) \lesa \|\psi_0\|_{H^4(\RR\times\sph^2)}^2 + \|\psi_3\|_{H^2(\RR\times\sph^2)}^2.
\end{equation}
Moreover we have the following $L^\infty$ decay estimate for large $r$: 
\begin{equation}\label{eq:scatteringpointwisedecay}
\sup_{\Sigma_r}\Big|\psi - \psi_0 - \frac{\psi_2}{r^2}-\frac{\psi_3}{r^3}\Big| \lesa \frac{1}{r^4}\|\psi_0\|_{H^6(\RR\times\sph^2)} + \frac{1}{r^5}\|\psi_3\|_{H^4(\RR\times\sph^2)},
\end{equation}
provided the right hand side is finite. Here $\psi_2$ is a function determined entirely by $\psi_0$.
\end{theorem}
\begin{remark}
The fact that there are two pieces of scattering data are due to the spacelike nature of $\Sigma^+$. For the corresponding Cauchy problem, one prescribes two pieces of data on a spacelike hypersurface which are the solution and its normal derivative restricted to $\Sigma^+$. In Theorem~\ref{thm:scatteringhorizon}, $\psi_0$, $\psi_3$ are the two pieces of data, with $\psi_3$ playing the role of the normal derivative.\footnote{However $\psi_3$ is not strictly speaking the normal derivative of the solution $\psi$, as any radial derivative $\psi$ would contain also the term $\psi_2$, which is determined by $\psi_0$. See the discussion in Remark~\ref{rmk:normalderiv} below.}
\end{remark}
\begin{remark}For related scattering results, we mention again \cite{Vas10}, where it was shown that given smooth functions $\psi_0,\psi_3$, there exists a unique smooth solution to the linear wave equation on a class of de Sitter-like spacetimes. See in particular the discussion below Theorem~1.2 in \cite{Vas10}. In \cite{Cic23}, a scattering theory was established for the linear wave equation on de Sitter spacetimes with even number of spatial dimensions.\end{remark}
\begin{remark}
Scattering problems have also been considered for linear wave equations on asymptotically flat black hole spacetimes, see \cite{BW14,DRSR18,KSR19}. The scattering problem for Einstein's vacuum equation with positive cosmological constant has also been treated, see for example \cite{Hintz23}, while for scattering of Einstein's vacuum equations in the asymptotically flat setting, see \cite{DHR13}.
\end{remark}
\subsection{Existence of solutions to the linear wave equation with prescribed scattering data on \texorpdfstring{$\Sigma^+$}{Sigma plus}}\label{sec:scatteringintro}
We give a complete treatment of scattering of wave equations on Schwarzschild-de Sitter in \textbf{Section~\ref{sec:scattering}}. We prove Theorem~\ref{thm:scatteringwithouthorizon} by first constructing an asymptotic solution to the wave equation of the form
$$\psi_{\text{asymp}}(t,\omega) = \psi_0(t,\omega) + \frac{\psi_2(t,\omega)}{r^2} + \frac{\psi_3(t,\omega)}{r^3},$$
where $\psi_0$, $\psi_3$ are freely chosen and $\psi_2$ is determined by $\psi_0$.\footnote{We are not constructing smooth scattering solutions to the wave equation by performing an infinite power series expansion in $r^{-1}$, although if one does this it turns out that $\psi_0$, $\psi_3$ are necessarily the leading order terms in such an expansion that can be freely chosen.} Clearly $\psi_{\text{asymp}}$ captures the asmyptotics of forward solutions proved in Theorem~\ref{thm:forwardasymptotics}, but moreover the asymptotic solution solves the wave equation up to some error that decays quickly in $r$, specifically we have $\square_g \psi_{\text{asymp}} = O(r^{-4})$.

Then we showing that there exists an actual solution to $\square_g \psi = 0$ that converges to that asymptotic solution. The key tool we use to prove this is the following \text{weighted energy estimate} suitable for the backward problem:
\begin{proposition}[Weighted energy estimate for the backward problem]\label{prop:backwardenergyestimate}
Suppose $\psi$ is a solution to the inhomogeneous wave equation 
\begin{equation}\label{inhomwav}
\square_g \psi = F
\end{equation} 
in the expanding region $\mc{R^+}$ of Schwarzschild-de Sitter spacetime. Then for all $r_2>r_1>r_\mc{C}$,
\begin{equation}\label{eq:backwardenergyestimate}
r_1^2(E[\psi](r_1))^{1/2} \leq r_2^2(E[\psi](r_2))^{1/2}  + \frac{1}{\sqrt{2}}\int_{r_1}^{r_2} \Big(\int_{\Sigma_r}r^2\phi F^2 \de \mu_{\ol{g}_r}\Big) \de r,
\end{equation}
provided the right hand side is finite.
\end{proposition}
We prove Proposition~\ref{prop:backwardenergyestimate} by identifying a vectorfield multliplier $M = r^2\phi^{-2}\pa_r$ that produces a \emph{nonpositive} bulk integral in the energy estimate; see Section~\ref{sec:backwardestimate} for relevant definitions. We prove this implies that the weighted energy $r^2E[\psi]$ is monotone decreasing \emph{backwards-in-time}\footnote{This is in contrast to the redshift estimate \eqref{eq:redshift0} from \cite{Sch15}, where a vectorfield was identified that produces a \emph{nonnegative} bulk integral, yielding an energy $E[\psi]$ that is monotonically decreasing forwards-in-time. These two energy estimates can also be compared to the standard energy estimate for the wave equation on Minkowski spacetime. There one can apply the timelike Killing vectorfield multiplier $T = \frac{\pa}{\pa t}$, which yields an energy that is conserved, rather than just monotone.}.

We will apply the weighted energy estimate from Proposition~\ref{prop:backwardenergyestimate} not to the solution itself, but to a fast-decaying remainder $\psi_{\text{rem}}$ that is obtained by subtracting from the true solution an asymptotic solution $\psi_{\text{asymp}}$ constructed in Section~\ref{sec:asympsol}. This asymptotic solution captures the leading order asymptotics of the solution.
\begin{remark}
This approach to constructing scattering solutions is inspired by \cite{LS23}, where scattering problems for nonlinear wave equations on Minkowski are treated. In that paper asymptotic solutions are constructed from a radiation field prescribed at null infinity, and then a fast-decaying remainder is estimated backwards-in-time via a \emph{fractional Morawetz estimate}, which plays the same role as Proposition \ref{prop:backwardenergyestimate} in the present paper. For other papers where this method is used for constructing scattering solutions to wave equations and other dispersive equations, see for example \cite{He21,Yu22}.
\end{remark}
\paragraph{Extension to the cosmological horizon} While Theorem~\ref{thm:scatteringwithouthorizon} guarantees existence of solutions to $\square_g \psi = 0$ on the interior of the expanding region, the energy estimate from Proposition~\ref{prop:backwardenergyestimate} degenerates at the cosmological horizons, and does not guarantee that solutions remain uniformly bounded on the cosmological horizons. 

We define a nondegenerate energy on the cosmological horizons. We denote this energy by $E^N[\psi]$. The full definition of $E^N[\psi]$ can be found in Section~\ref{sec:nondegcurrent}, but for now we mention that on the cosmological horizons:
$$E^N[\psi] = \frac{1}{2}\int_{\mc{C}^+\cup \ol{\mc{C}}^+} (T\psi)^2 + |\slashed\nabla\psi|^2,$$
where $\slashed\nabla$ denotes the covariant derivative on the sphere of radius $r$.

In order to solve all the way to the cosmological horizons $\mc{C}^+\cup \ol{\mc{C}}^+$ in a nondegenerate energy space, we additionally assume exponential decay of the data $\psi_0$, $\psi_3$ along the future boundary $\Sigma^+$ towards $\iota^+$ and $\ol{\iota}^+$.
\begin{theorem}[Scattering up to the cosmological horizon]\label{thm:scatteringhorizon}
Let $\psi_0,\psi_3$ be scattering data defined on the future boundary $\Sigma^+$ such that $\psi_0 \in H^4(\RR\times\sph^2),\psi_3\in H^2(\RR\times\sph^2)$. Moreover, assume that $\psi_0$, $\psi_3$ decay exponentially, in that there exists some sufficiently large constant $\beta = \beta(\Lambda,m) > 0$ such that as $\tau \ra \infty$ we have
\begin{equation}\label{eq:expdecayassumption}\|\psi_0\|_{H^4((-\tau,\tau)^c\times \sph^2)},\|\psi_3\|_{H^2((-\tau,\tau)^c\times \sph^2)} \lesa_{\Lambda,m} e^{-\beta \tau}.
\end{equation}
Then there exists a unique solution $\psi$ to \eqref{eq:waveequation} on $\mc{R}^+\cup\mc{C}^+ \cup \ol{\mc{C}}^+$ with scattering data $\psi_0,\psi_3$, and at the cosmological horizon, we have the energy estimate
\begin{equation}\label{eq:scatteringhorizonbound}
E^N[\psi] \lesa_{\Lambda,m} \|\psi_0\|_{H^4(\RR\times\sph^2)}^2 + \|\psi_3\|_{H^2(\RR\times\sph^2)}^2.
\end{equation}
\end{theorem}
\begin{remark}
The infimum of all values of the constant $\beta$ for which Theorem~\ref{thm:scatteringhorizon} holds is $\frac{1}{2}\kappa_{\mc{C}}$, where $\kappa_{\mc{C}}$ is the \emph{surface gravity} of the cosmological horizons. This requirement of a sufficient degree of exponential decay is tied to the \emph{blueshift effect}, a property of Killing horizons such as cosmological horizons. The previously-discussed redshift effect leads to boundedness of waves forwards-in-time along Killing horizons. However for the backwards problem this is seen as a blueshift, which leads to exponential growth of waves backwards-in-time. While it is primarily discussed in the setting of black holes in asymptotically flat spacetimes (see \cite{DHR13,DRSR18,Sb15}), the blueshift effect is present for all nondegenerate Killing horizons.
\end{remark}
By applying the pointwise boundedness estimate found in Corollary 2.12 in \cite{Sch15}, it follows that the solutions constructed in Theorem~\ref{thm:scatteringhorizon} are uniformly bounded on $\ol{\mc{R}}^+$. 
\begin{corollary}\label{crl:pwbound}
Let $\psi_0,\psi_3,\psi$ be as in Theorem~\ref{thm:scatteringwithouthorizon}, with $\psi_0 \in H^6(\RR\times\sph^2)$,$\psi_3\in H^4(\RR\times\sph^2)$.Then we have have the pointwise bound away from the cosmological horizons:
\begin{equation}
\sup_{\{r \geq \rc + \epsilon\}}|\psi| \lesa_{\Lambda,m,\epsilon} \|\psi_0\|_{H^6(\RR\times\sph^2)} + \|\psi_3\|_{H^4(\RR\times\sph^2)}.
\end{equation}
If additionally $\psi_0$, $\psi_3$ decay exponentially along $\Sigma^+$ like in Theorem~\ref{thm:scatteringhorizon}, then the solution $\psi$ is uniformly bounded up to the horizon:
\begin{equation}\label{eq:pointwiseboundeverywhere}
\sup_{\ol{\mc{R}}^+}|\psi| \lesa_{\Lambda,m} \|\psi_0\|_{H^6(\RR\times\sph^2)} + \|\psi_3\|_{H^4(\RR\times\sph^2)}.
\end{equation}
\end{corollary}
\begin{remark}
Corollary~\ref{crl:pwbound} is proved by commuting energy estimates with the Killing vectorfields that span the hypersurfaces $\Sigma_r$, $\mc{C}^+$ and $\ol{\mc{C}}^+$. These are the generators of the spherical isometries $\Omega_i:i=1,2,3$, and the coordinate vectorfield $T =\frac{\pa}{\pa t}$. A Sobolev embedding on these hypersurfaces gives pointwise bounds on $\psi$, see~\cite{Sch15} for details.
\end{remark}
\subsection{Forward asymptotics for linear waves on Schwarzschild-de Sitter}
We now turn our attention to the proof of Theorem~\ref{thm:forwardasymptotics}, which we set up and prove in \textbf{Section~\ref{sec:forwardasymptotics}}. Theorem~\ref{thm:forwardasymptotics} relies on several new higher-order energy estimates. These estimates capture the global expansion of $\mc{R}^+$. We have:
\begin{theorem}[Higher-order redshift estimate]\label{thm:forwardenergyintro}
Let $\psi$ be a solution to the wave equation $\square_g \psi = 0$ on the expanding region of Schwarzschild-de Sitter, and let $r_0$ be a radius that is slightly larger than $\rc$. Let $X,Y$ denote the vectorfields
$$X = r\phi^{-2}\pa_r,\quad Y = r\phi^{-1}\pa_r.$$
Then for all $r_2 > r_1 \geq r_\mc{C}$ we have the energy estimate
\begin{multline}
E_{\Sigma_r}^{\pa_r}[Y(X\psi)](r_2) + E_{\Sigma_r}^{\pa_r}[X\psi](r_2) + E_{\Sigma_r}^{\pa_r}[\psi](r_2)\\* 
\lesa_{\Lambda,m,r_0} E_{\Sigma_r}^{\pa_r}[Y(X\psi)](r_1) + E_{\Sigma_r}^{\pa_r}[X\psi](r_1) + E_{\Sigma_r}^{\pa_r}[\psi](r_1).\label{eq:forwardestintro}
\end{multline}
\end{theorem}
\begin{remark}\label{rmk:normalderiv}
The weights in $r$ for the vectorfield commutators $X,Y$ are, to leading order:
$$X \sim r^3\pa_r,\quad Y \sim r^2\pa_r,$$
and so applying these vectorfields to a solution $\psi$ to the wave equation essentially identifies higher-order terms in the asymptotic expansion \eqref{eq:asymptoticsintro}. To see this, if we assume the asymptotic expansion \eqref{eq:asymptoticsintro} holds, then
$$\psi \sim \psi_0, \quad r^3\pa_r \psi \sim \psi_2,\quad r^2\pa_r(r^3\pa_r\psi) \sim \psi_3.$$
\end{remark}
\begin{remark}In \cite{Sch15}, it was shown that the zeroth-order energy $E[\psi]$ is monotone in $r$, so that if $\square_g \psi = 0$, then
\begin{equation}\label{eq:redshift0}
E[\psi](r_2) \leq E[\psi](r_1)
\end{equation}
for all $r_2 \geq r_1 > \rc$. In that paper a \emph{vectorfield multiplier} is identified that captures the global expansion of $\mc{R}^+$, in the sense that it produces a nonnegative bulk term in the resulting energy identity.
\end{remark}
We derive the higher-order redshift estimate of Theorem \ref{thm:forwardenergyintro} by showing that as \emph{vectorfield commutators}, $X$ and $Y$ satisfy good commutation identities, once again leading to a nonnegative bulk term. These commutators are proportional to $\pa_r$, and so are not Killing. This means that $X$ and $Y$ have commutation identities containing lower-order terms that also need to be controlled. 
\begin{remark}
A similar higher-order global redshift estimate has been proved in \cite{Sch22} for \emph{Weyl fields} on expanding spacetimes. Weyl fields are trace-free solutions to the Bianchi equations
$$\nabla^\mu W_{\mu\nu\lambda\delta} = 0.$$
In that paper, the Bianchi equations are the hyperbolic differential equations for which a higher-order redshift estimate is proved. In the context of the local redshift effect on black hole horizons, these vectorfield commutator estimates stem from Dafermos and Rodnianski in \cite{DR13}. We also mention \cite{FS20}, in which similar ideas are used to proved higher-order energy estimates for linear waves on the black hole region of Schwarzschild spacetimes.
\end{remark}
\begin{remark}
See \cite{CNO19,NR23} for results on forward asymptotics for the wave equation on related cosmological spacetimes. Forward asymptotics have been studied for solutions to Einstein's vacuum equations. In particular, asymptotic expansions for perturbations of Minkoswki are obtained in \cite{Lind17,HV20}.
\end{remark}
Through a density argument presented in Section~\ref{sec:forwardlimits}, Theorem~\ref{thm:forwardenergyintro} implies that the quantities $\psi, X\psi,Y(X\psi)$ all have a nonzero limit on $\Sigma^+$ with respect to the energy norm. A simple application of the fundamental theorem of calculus then implies the asymptotic expansion \eqref{eq:asymptoticsintro}, due to the weights in $r$ possessed by the vectorfield commutators $X,Y$.
\begin{remark}
We expect that one can derive successive higher-order weighted estimates by repeatedly commuting the existing estimate \eqref{eq:forwardestintro} with the commutator $Y$. This would allow use to derive for solutions an asymptotic expansion to arbitrary order in powers of $r^{-1}$. However, the present estimate is sufficient for us to produce an asymptotic expansion that identifies the two pieces of data $\psi_0$ ,$\psi_3$ for the corresponding scattering problem.
\end{remark}

\subsection{Scattering on perturbations of Schwarzschild-de Sitter}
The results we prove in this paper Schwarzschild-de Sitter can be suitably generalised to a broad class of expanding spacetimes. This is because they do not rely on fundamental solution techniques. In Section~\ref{sec:perturbations} we highlight this by extending the scattering result from Theorem~\ref{thm:scatteringwithouthorizon} to a large class of expanding spacetimes. In particular, we show that the construction of scattering solutions extends to Kerr-de Sitter, as well as spacetimes that do not necessarily converge to Schwarzschild-de Sitter. The class of metrics we consider, broadly speaking, are those with metrics that have the same leading-order asymptotic behaviour as Schwarzschild-de Sitter in terms of growth or decay in $r$.\footnote{This class of spacetimes is similar to those considered in \cite{Vas10}, with the difference being that the metric components here can be assumed to be of finite-order regularity rather than smooth.} We do not assume closeness of the metric to Schwarzschild-de Sitter, and we make no symmetry assumption. An explicit description of the metrics we consider can be found in Definition~\ref{def:perturbations}

On this class of spacetimes we have the following scattering result, stated informally:
\begin{theorem}[Scattering of linear waves on perturbations of Schwarzschild-de Sitter]\label{thm:pertscatteringinformal}
An analogue of the scattering result of Theorem~\ref{thm:scatteringwithouthorizon} holds for solutions to the linear wave equation on spacetimes with the same leading-order asymptotic behaviour as Schwarzschild-de Sitter.
\end{theorem}
For a formal statement of this result, see \textbf{Theorem~\ref{thm:pertscatteringformal}} in \textbf{Section~\ref{sec:scatteringperturbed}}. We show that the two components of the proof, which are the construction of asymptotic solutions and a weighted energy estimate suitable for the backward problem, generalise from Schwarzschild-de Sitter to the perturbed setting. 

Depending on the form of the metrics, the asymptotic solutions to the wave equation on a given spacetime we construct are of the form
$$\psi_{\text{asymp}} \sim \psi_0 + \frac{\psi_2}{r^2} + \frac{\log r\>\psi_{3,1}}{r^3} + \frac{\psi_3}{r^3}.$$
Thus a $r^{3}\log r\> \psi_{3,1}$ term can be present in the expansion. We specify sufficient conditions on the metric for which this log term is not present. In particular these conditions are satisfied by the expanding region of the Kerr-de Sitter family of spacetimes. This implies the following:
\begin{corollary}[Scattering on Kerr-de Sitter]
Let $\psi_0,\psi_3$ be sufficiently regular functions at the future boundary of Kerr-de Sitter. Then there exists a unique, finite-energy solution to the wave equation $\square_g \psi = 0$ on the expanding region of Kerr-de Sitter such that for large $r$ we have
$$\psi \sim \psi_0 + \frac{\psi_2}{r^2} + \frac{\psi_3}{r^3},$$
where $\psi_2$ is determined by $\psi_0$.
\end{corollary}

\subsection{Acknowledgements}
I would like to express my gratitude to my adviser Volker Schlue for suggesting this problem to me, and for his support and guidance throughout the process of writing this paper. I would also like to thank the University of Melbourne for its financial support.


\section{Geometry of the Schwarzschild-de Sitter family of spacetimes}\label{sec:notation}
In this section we recall some of the relevant properties of Schwarzschild-de Sitter spacetimes, and the energies, norms, and other notation which we use throughout this paper. For a more complete introduction of the geometry of Schwarzschild-de Sitter, see \cite{LR77}, as well as \cite{Sch15}.

\subsection{Schwarzschild-de Sitter spacetime}
The Schwarzschild-de Sitter family of spacetimes, discovered independently in \cite{Kot18} and \cite{Wey19}, constitute a one-parameter family of solutions to
Einsteins vacuum equations with positive cosmological constant
\begin{equation}\label{eq:EVE}
\text{Ric}(g) - \Lambda g = 0, \quad \Lambda > 0.
\end{equation}
These spacetimes are parametrised by the black hole mass, denoted by $m$.
We assume that $0 <3m < 1/\sqrt{\Lambda}$, which restricts the class of spacetimes we analyse to \textit{subextremal} Schwarzschild-de Sitter.

Recall the metric
\begin{equation}\label{eq:metricstandardcoords}
g_{\Lambda,m} = -\frac{1}{\LT r^2 - 1 + \frac{2m}{r}}\de r^2 + \Big(\LT r^2 -1 +\frac{2m}{r}\Big) \de t^2 + r^2\>\de \theta^2 + r^2\sin^2\theta \>\de \varphi^2,
\end{equation}
defined on the expanding region. From this, we see that the coordinate vector field $\frac{\pa}{\pa r}$ is timelike on $\mc{R}^+$, and we can view $r$ as a \textit{time function} for the expanding region. Recall also the lapse of the $\Sigma_r$-foliation
$$\phi = \frac{1}{\sqrt{\LT r^2 - 1 + \frac{2m}{r}}}.$$
As $\phi$ frequently appears in weighted energies in this paper, we emphasise that for large $r$, $\phi \sim 1/r$. One can decompose the metric \eqref{eq:metricstandardcoords} so that
\begin{equation}\label{eq:metricstandarddecompose}
g = -\phi^2\de r^2 + \ol{g}_r,
\end{equation}
where $\ol{g}_r$ is the induced metric on the hypersurfaces $\Sigma_r$.

We also recall from \cite{Sch15} the \emph{Kruskal coordinates} $(u,v)$ on Schwarzschild-de Sitter
The coordinates $(r,t,\theta,\varphi)$ cover only the expanding region $\mc{R}^+$, and the metric \eqref{eq:metricstandardcoords} becomes singular at $r = \rc$. In contrast the Kruskal coordinates $(u,v,\theta,\varphi)$ cover the cosmological horizons $\mc{C}^+,\ol{\mc{C}}^+$, and are related to $(r,t,\theta,\varphi)$ coordinates implicitly by
\begin{align*}
uv = \frac{r - \rc}{(r-\rh)^{\alpha_{\mc{H}}}(r+|\roc|)^{\ol{\alpha_{\mc{C}}}}},\\*
\log\Big|\frac{u}{v}\Big| = -\LT\frac{(\rc-\rh)(\rc+|\roc|)}{\rc}t,
\end{align*}
where 
$$\alpha_{\mc{H}} = \frac{\rh}{\rc}\frac{\rc+\roc}{\rh+\roc},\quad \ol{\alpha_{\mc{C}}} = \frac{\roc}{\rc}\frac{\rc-\rh}{\rc+\roc}$$
are positive constants such that $\alpha_{\mc{H}}+\ol{\alpha_{\mc{C}}} = 1$.
We will also write
\begin{equation}\label{eq:surfacegrav}
\kappa_{\mc{C}}\doteq\frac{1}{2}\LT\frac{(\rc-\rh)(\rc+|\roc|)}{\rc},
\end{equation}
and mention briefly that $\kappa_{\mc{C}}$ is the so-called surface gravity of the cosmological horizons. With respect to these coordinates, the metric is
\begin{equation}\label{eq:metrickruskalcoords}
g = -\Omega^2 \de u \de v + r^2 \gamma_{\sph^2},
\end{equation}
where $\Omega^2$ is the function
\begin{equation}\label{eq:omegasquareddef}
\Omega^2 = \LT \frac{1}{\kappa_\mc{C}^2 r}(r-\rh)^{1+\alpha_{\mc{H}}}(r+|\roc|)^{1+\ol{\alpha_{\mc{C}}}}.
\end{equation}
We emphasise that $\Omega^2$ does not vanish on $\mc{C}^+$. The Kruskal coordinates $(u,v) \in \RR^2$ cover the region $\rh < r(u,v) < \infty$, and the expanding region $\mc{R}^+$ is equivalent to the set $\{(u,v): 0 < uv < 1\}$. Moreover, for the cosmological horizons $\mc{C}^+= \{(u,v): v = 0, u \geq 0\}$, $\ol{\mc{C}}^+ = \{(u,v): u=0,v\geq 0\}$. Finally, the future boundary $\Sigma^+$ corresponds with the level set $\{(u,v): uv = 1, u,v > 0\}$.

We recall the Killing vector field $T$, which in Kruskal coordinates is given by
\begin{equation}\label{eq:killingvectorfielddef}
T = \kappa_{\mc{C}}\Big(u\frac{\pa}{\pa u} - v \frac{\pa}{\pa v}\Big).
\end{equation}
On $\mc{R}^+$, $T$ is precisely the coordinate vectorfield $\frac{\pa}{\pa t}$, and is spacelike. On the cosmological horizons $T$ is null. We also define the vectorfield
\begin{equation}\label{eq:yvectorfielddef}
Y|_{\mc{C}^+} = \frac{1}{\iota_{\mc{C}}}\frac{1}{u}\frac{\pa}{\pa v},
\end{equation}
where 
$$\iota_{\mc{C}} = \frac{4}{\kappa_{\mc{C}}}\frac{1}{\Omega^2}\Big|_{\mc{C}^+}.$$
One can see that $Y$ is conjugate to $T$ on the horizon, specifically we have
$$g(T,Y)|_{\mc{C}^+} = -\frac{1}{2}\kappa_{\mc{C}}\iota_{\mc{C}}\Omega^2|_{\mc{C}^+} = -2.$$
\subsection{Wave equations and the energy method}\label{sec:vectorfieldmethod}
In this paper we apply the classical vectorfield method to prove estimates on finite-energy solutions to linear wave equations, see \cite{DR13} for an introduction to these methods in a general relativity context. Recall the standard energy-momentum tensor $T$ for the linear wave equation, given by
$$T_{\mu\nu}[\psi] = \pa_\mu \psi \pa_\nu \psi - \frac{1}{2} g_{\mu\nu} (\pa^\alpha \psi \pa_\alpha \psi),$$
as well as the energy current $J^X$ with respect to a given vectorfield $X$, a $1$-form defined by
$$J^X[\psi] \cdot Y = T[\psi](X,Y).$$
Given timelike vector fields $X,Y$, $J^{X}[\psi]\cdot Y$ is a positive-definite quadratic form of $\pa \psi$ on $\mc{R}^+$, in the sense that the following pointwise inequality holds:
$$J^{X}[\psi]\cdot Y \geqa \sum_{|\alpha|=1}|\pa^\alpha \psi|^2.$$ We will use the following notation for energy fluxes and norms induced by these energy currents. 
\begin{definition}
Given a causal vectorfield $X$, and spacelike or null hypersurface $\Sigma$, Let $E_\Sigma^X[\psi]$ denote the energy flux through $\Sigma$ with respect to $X$, so that
\begin{equation}\label{eq:energynotation}
E_\Sigma^X[\psi] \doteq \int_{\Sigma} \>^*J^X[\psi],
\end{equation}
where $\>^*J^X[\psi]$ denotes the Hodge dual of the corresponding vectorfield $(J^X)^\sharp$. 
We also define the corresponding induced energy norm
\begin{equation}\label{eq:energynormnotation}\|\psi\|_{X,\Sigma} \doteq \Big(E_\Sigma^X[\psi]\Big)^{1/2}.
\end{equation}
\end{definition}
If $\Sigma$ is a spacelike hypersurface, we may write 
$$\>^*J^X[\psi]|_\Sigma = J^X[\psi]\cdot n_{\Sigma} \>\de \mu_{\ol{g}_\Sigma},$$
where $n_{\Sigma}$ is the unit normal of $\Sigma$, and $\ol{g}_\Sigma$ is the induced metric on $\Sigma$. 
\begin{remark}
We point out that the energy $E[\psi]$ defined in \eqref{eq:naturalenergy} is equivalent to the energy flux $E_{\Sigma_r}^{\pa_r}$:
\begin{equation*}
E[\psi](r) = E_{\Sigma_r}^{\pa r}[\psi] = \int_{\Sigma_r} J^{\pa_r}[\psi] \cdot n_{\Sigma_r} \de \mu_{\ol{g}_r}
\end{equation*}
To understand the weights in $r$ associated with the energy $E[\psi]$, we note that 
\begin{gather*}
 n_{\Sigma_r} = \phi^{-1}\frac{\pa}{\pa r} \sim r\frac{\pa}{\pa r}, \quad \de \mu_{\ol{g}_r} = r^2\phi^{-1}\>\de t \wedge \de\mu_{\sph^2} \sim r^3\>\de t \wedge \de \mu_{\sph^2}.
\end{gather*}
Thus for large $r$ we have
\begin{align*}
E[\psi]  &= \frac{1}{2}\int_{\RR \times \sph^2} r^2\phi^{-2}(\pa_r\psi)^2 + r^2\phi^2 (\pa_t\psi)^2 + r^2|\slashed\nabla \psi|^2 \de t \de \mu_{\mr{\gamma}}\\*
&\sim r^4\|\pa_r\psi(r)\|_{L^2(\RR\times\sph^2)}^2 + \|\tilde{\nabla}\psi(r)\|_{L^2(\RR\times\sph^2)}^2.
\end{align*}
where  $\slashed\nabla\psi$ is the standard covariant derivative on a sphere of radius $r$, $\tilde{\nabla}\psi$ is the standard covariant derivative on the cylinder $\RR\times\sph^2$, and $\de \mu_{\mr\gamma}$ is the standard volume form on $\sph^2$.

\end{remark}
We will also define the following $L^2$-norm on the $\Sigma_r$-hypersurfaces. Let
$$\|\psi\|_{L^2(\Sigma_r)}^2 = \int_{\Sigma_r}\psi^2 \de \mu_{\ol{g}_r} \sim r^3\|\psi(r)\|_{L^2(\RR\times\sph^2)}^2.$$

In this paper, we adopt the following notation conventions to distinguish differential operators on relevant manifolds or submanifolds. We write $\slashed\nabla$ to denote the covariant derivative restricted to a sphere of radius $r$, $\ol\nabla$ for the covariant derivative restricted to the level sets $\Sigma_r$, and $\tilde\nabla$ for the covariant derivative with respect to the conformal metric at $\Sigma^+$:
$$\tg = \lim_{r \ra \infty}\Big(\frac{1}{r^2} \ol{g}_r\Big) = \LT dt^2 + \gamma_{\sph^2}.$$
We will also use this convention for other differential operators which appear in this paper, such as the Laplacian.

\section{Scattering of linear waves on Schwarzschild-de Sitter}\label{sec:scattering}
In this section we will prove existence and uniqueness of scattering solutions to the linear wave equation on the expanding region. The main results of this section are \textbf{Theorem~\ref{thm:scatteringwithouthorizon}} which we prove in Section~\ref{sec:scatteringinterior}, and \textbf{Theorem~\ref{thm:scatteringhorizon}} which we prove in Section~\ref{sec:scatteringhorizon}.

First we will prove the existence and uniqueness of scattering solutions on the interior of the expanding region, which corresponds to Theorem~\ref{thm:scatteringwithouthorizon}. Specifically, we will show that given scattering data $\psi_0$, $\psi_3$, there exists a unique solution to $\square_g \psi = 0$ that is uniformly bounded up to a $\Sigma_{r_0}$ hypersurface for some $r_0 > \rc$. We will consider an asymptotic solution of the form
$$\psi_{\text{asymp}} = \psi_0(t,\omega) + \frac{\psi_2(t,\omega)}{r^2} + \frac{\psi_3(t,\omega)}{r^3},$$
where $\psi_0,\psi_3$ are freely chosen, and $\psi_2$ is determined entirely by $\psi_0$. We will show that that $\psi_{\text{asymp}}$ approximates a true solution, in the sense that
$$|\square_g \psi_{\text{asymp}}|  = O\big(\frac{1}{r^4}\big).$$

\begin{figure}[t]
\centering
\begin{tikzpicture}[scale=1.5]
\filldraw[fill=gray!30](0,3) .. controls (2.75,0.5) and (3.25,0.5) .. (6,3) .. controls (3.5,1.75) and (2.5,1.75)  .. cycle;
\draw (0,3) -- (3,0);
\draw (3,0) -- (6,3);
\draw[dashed] (0,3) .. controls (2.5,2.5) and (3.5,2.5) .. (6,3);
\filldraw  (3,0) circle (0.035cm);

\draw (3,1.1) node[below]{$\Sigma_{r_0}$};
\draw (0,3) .. controls (2.5,1.75) and (3.5,1.75) .. (6,3);
\draw[->] (3,2.125) -- (3,2.55);
\draw[->] (2.2,2.25) -- (2.2,2.6);
\draw[->] (3.75,2.25) -- (3.75,2.6);
\draw (2.5,1.9) node{$\Sigma_R$};
\draw (3,2.8) node{$\Sigma^+$};
\filldraw[fill=white] (0,3) circle (0.035cm);
\filldraw[fill=white] (6,3) circle (0.035cm);
\end{tikzpicture}
\caption{Spacetime region on which we solve the wave equation on the interior of the expanding region. We solve a sequence of finite problems, i.e. we consider a solution $\psi_R$ with prescribed data on $\Sigma_R$, and then show that the limit $\lim_{R \ra\infty}\psi_R$ exists in a suitable energy space. In this step we solve up to a $\Sigma_{r_0}$ hypersurface, where $r_0 > \rc$.}\label{fig:scatteringexpanding}
\end{figure}
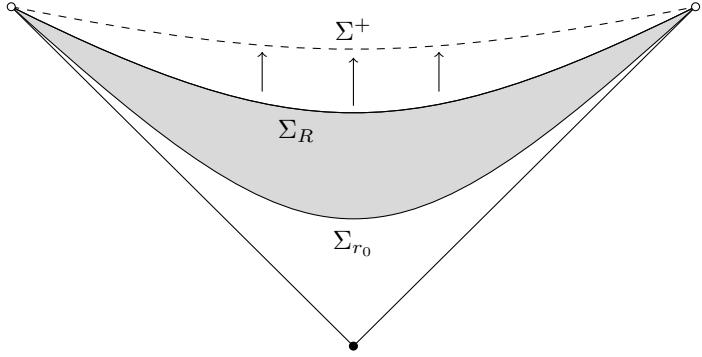

We will prove existence of scattering solutions with scattering data $\psi_0,\psi_3$ on the interior of the expanding region $\mc{R}^+$. We do this by taking a sequence of solutions of finite problems, and showing that the limiting function exists and solves the scattering problem. Specifically, we take sequence of solutions $\psi_R$ to \eqref{eq:waveequation} with prescribed data on the hypersurface $\Sigma_R$; see Figure~\ref{fig:scatteringexpanding}. We set $\psi_R = \psi_{\text{asymp}} + \psi_{\text{rem}}^{R}$, where $\psi_{\text{rem}}^{R}$ satisfies
$$\square_g \psi_{\text{rem}}^{R} = -\square_g \psi_{\text{asymp}},$$
and has trivial data on $\Sigma_R$, so that
$$\psi_{\text{rem}}^{R}|_{\Sigma_R} = 0,\quad n_{\Sigma_R}\cdot \psi_{\text{rem}}^{R}|_{\Sigma_R}=0.$$
The key technical tool we use to control the sequence of remainders $\psi_{\text{rem}}^{R}$ is the weighted energy estimate \textbf{Proposition~\ref{prop:backwardenergyestimate}}, which we prove in Section~\ref{sec:backwardestimate}. Using this weighted energy estimate, we show that the limit $\psi = \lim_{R \ra \infty}\psi_R$ exists with respect to the energy norms \eqref{eq:naturalenergy}, and $\psi$ remains uniformly bounded on $\Sigma_{r_0}$, proving Theorem~\ref{thm:scatteringwithouthorizon}.

We then extend the solutions constructed in Theorem~\ref{thm:scatteringwithouthorizon} to the cosmological horizons by solving the wave equation backwards-in-time from $\Sigma_{r_0}$ to the $\mc{C}^+\cup\ol{\mc{C}}^+$, proving Theorem~\ref{thm:scatteringhorizon}. For this we will assume exponential decay of the scattering data $\psi_0,\psi_3$ along $\Sigma^+$ toward $\iota,\ol{\iota}$. We will consider a sequence of solutions $\psi^{T}$ to the wave equation $\square_g \psi =0$. This time $\psi_T$ has data that is prescribed on a truncation of the level set $\Sigma_{r_0}$, and on null hypersurfaces $C_T^c$ ,$\ol{C}_T^c$ that are transversal to each cosmological horizon; see Figure \ref{fig:scatteringhorizon}. The prescribed data on $\Sigma_{r_0}$ matches the solution constructed in Section~\ref{sec:scatteringinterior}, while the prescribed data on $C_T^c$, $\ol{C}_T^c$ is trivial. We will show that the limit $\lim_{T \ra \infty}\psi^T = \psi$
exists with respect to nondegenerate energy norms, and that $\psi$ has finite nondegenerate energy on $\mc{C}^+ \cup \ol{\mc{C}}^+$.
\begin{figure}[b]
\centering
\begin{tikzpicture}[scale=1.5]
\filldraw[fill=gray!30](4.325,1.325) -- (3,0) -- (1.675,1.325) -- (2,1.65).. controls (2.75,1.25) and (3.25,1.25) .. (4,1.65) -- cycle;
\draw (0,3) -- (1.675,1.325);
\draw (4.325,1.325) -- (6,3);
\draw[dashed] (0,3) .. controls (2.5,2.5) and (3.5,2.5) .. (6,3);

\draw (3,1.8) node[below]{$\Sigma_{r_0}^T$};
\draw (2,1.65) .. controls (1.4,2) .. (0,3);
\draw (4,1.65) .. controls (4.6,2) .. (6,3);
\filldraw  (2,1.65) circle (0.035cm);
\filldraw  (1.675,1.325) circle (0.035cm);
\filldraw  (4,1.65) circle (0.035cm);
\filldraw  (4.325,1.325) circle (0.035cm);
\filldraw  (3,0) circle (0.035cm);
\draw (2,1.55) node[below]{$C_T^c$};
\draw (3.95,1.55) node[below]{$\ol{C_T^c}$};
\draw[->] (1.75,1.55) -- (0.95,2.2);
\draw[->] (4.25,1.55) -- (5.05,2.2);
\draw (3,2.8) node{$\Sigma^+$};
\draw (1.95,0.9) node[below]{$\mc{C}^+$};
\draw (4,0.9) node[below]{$\ol{\mc{C}}^+$};
\filldraw[fill=white] (0,3) circle (0.035cm);
\filldraw[fill=white] (6,3) circle (0.035cm);
\draw (0,3) node[above]{$\iota^+$};
\draw (6,3) node[above]{$\ol{\iota}^+$};
\end{tikzpicture}
\caption{Spacetime region on which we solve the wave equation up to the cosmological horizons. We solve a sequence of finite problems, i.e. we consider a solution $\psi_T$ with prescribed data on $\Sigma_{r_0}^T\cup C_T^c \cup \ol{C}_T^c$, and then show that the limit $\lim_{T \ra\infty}\psi_T$ exists in a suitable energy space.}\label{fig:scatteringhorizon}
\end{figure}
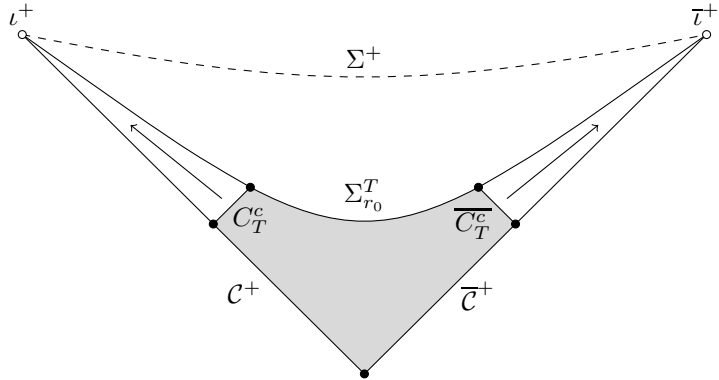


\subsection{The construction of asymptotic solutions}\label{sec:asympsol}
In this section we construct asymptotic solutions that satisfy the wave equation $\square_g \psi_{\text{asymp}} = 0$ up to some small error that decays toward $\Sigma^+$. As previously stated, we will eventually set
$$\psi_{\text{asymp}} = \psi_0 + \frac{\psi_2}{r^2} +\frac{\psi_3}{r^3},$$
where $\psi_0$ and $\psi_3$ are freely chosen, and $\psi_2$ is determined by $\psi_0$. We motivate this choice by first considering a more general finite-order power series that includes logarithmic terms, so that
\begin{equation}\label{eq:asympsolhighorder}
\psi_{\text{asymp}}^{(N)}(r,t,\omega) = \sum_{n=0}^N \frac{\psi_n(t,\omega)}{r^n} + \sum_{n=1}^N \frac{\log r \> \psi_{n,1}(t,\omega)}{r^n},
\end{equation}
for some positive integer $N$. Here the $\psi_n$, $\psi_{n,1}$ are functions on the future boundary $\Sigma^+$. Since $\Sigma^+$ is diffeomorphic to $\RR\times\sph^2$, the $\psi_n$ can also be thought of as functions on the standard cylinder. 
\begin{remark}
We emphasise that we will not be constructing solutions via an infinite power series expansion.
The notion of constructing solutions to the linear wave equation on Minkowski spacetime as a power series in $1/r$ was first explored in detail by Friedlander, see \cite{Fr1}. Due to the different geometry of Minkowski spacetime, the particular asymptotic expansion is manifestly different. In particular, solutions to $\square \psi = 0$ behave to leading order near $\mc{I}^+$ like
$$\psi(t,x) \sim \frac{F(r-t,\omega)}{r},\quad r = |x|,\quad\omega=\frac{x}{|x|},$$
where $F(r-t,\omega)$ is the so-called \emph{radiation field}, whose existence was established by Friedlander in the aforementioned papers. See also \cite{FS20} for an example of asymtotic expansions of linear waves near a black hole singularity.
\end{remark}

The wave equation on the expanding region of Schwarzschild-de Sitter spacetime is expressed as
\begin{equation}\label{eq:waveeqcoords}
\square_g \psi = \phi^2 \pa_t^2 \psi - \frac{1}{r^2}\pa_r\Big(\frac{r^2}{\phi^2}\pa_r\psi\Big)+\frac{1}{r^2}\Delta_{\sph^2} \psi.
\end{equation}

We begin with a preliminary computation that we will use to compute $\square_g \psi_{\text{asymp}}^{(N)}$.
\begin{lemma}\label{lem:waveeqmonomial}
Fix $r_0 > \rc$. Let $n$ be a nonnegative integer, and let $\psi_n \in C^2(\RR\times \sph^2)$. Then for all $r \geq r_0$, we have the following two identities:
\begin{multline}\label{eq:waveeqmonomial}
\square_g \Big(\frac{\psi_n(t,\omega)}{r^n}\Big) = -\LT n(n-3)\frac{1}{r^n}\psi_n + \frac{1}{r^{n+2}}\Big(n(n-1)\psi_n + \tilde{\Delta}\psi_n\Big)\\*
- n^2\frac{2m}{r^{n+3}}\psi_n + O\big(\frac{1}{r^{n+4}})(\pa_t^2\psi_0).
\end{multline}
\begin{equation}\label{eq:waveeqmonomiallog}
\square_g \Big(\frac{\log r \>\psi_{n,1}}{r^n}\Big) = -\LT n(n-3)\frac{\log r}{r^n}\psi_{n,1} + \LT (2n-3)\frac{1}{r^n}\psi_{n,1} + O\big(\frac{\log r}{r^{n+2}})\big.
\end{equation}
\end{lemma}
\begin{proof}
By \eqref{eq:waveeqcoords}, we have
\begin{equation}\label{eq:waveeqmonomial2}
\square_g \Big(\frac{\psi_n}{r^n}\Big) = \frac{\phi^2}{r^n}\pa_t^2 \psi_n - \frac{1}{r^2}\pa_r\Big(\frac{r^2}{\phi^2}\pa_r\Big(\frac{1}{r^n}\Big)\Big)\psi_n + \frac{1}{r^{n+2}}\Delta_{\sph^2}\psi_n.
\end{equation}
The leading-order behaviour of $\phi^2$ is
\begin{equation}\label{eq:lapseexpansion}\phi^2 = \frac{1}{\LT r^2 - (1-\frac{2m}{r})} = \frac{3}{\Lambda} \frac{1}{r^2} + \frac{1-\frac{2m}{r}}{\LT r^2\Big(\LT r^2 - 1 + \frac{2m}{r}\Big)}= \frac{3}{\Lambda}\frac{1}{r^2} + O\big(\frac{1}{r^4}\big),
\end{equation}
and so
$$\frac{\phi^2}{r^n}\pa_t^2\psi_n + \frac{1}{r^{n+2}}\Delta_{\sph^2}\psi_n = \frac{1}{r^{n+2}}\tilde{\Delta}\psi_n + O\big(\frac{1}{r^{n+4}}\big).$$
Meanwhile, we compute for the radial derivatives:
\begin{align*}
-\frac{1}{r^2}\pa_r\Big(\frac{r^2}{\phi^2}\pa_r\Big(\frac{1}{r^n}\Big)\Big)\psi_n &= \frac{n}{r^2}\pa_r\Big(\frac{1}{\phi^2r^{n-1}}\Big)\psi_n\\*
&= \frac{n}{r^2}\pa_r\Big(\LT \frac{1}{r^{n-3}} - \frac{1}{r^{n-1}} + \frac{2m}{r^n}\Big)\psi_n\\*
&= -\LT\frac{n(n-3)}{r^n}\psi_n +\frac{n(n-1)}{r^{n+2}}\psi_n - 2m\frac{n^2}{r^{n+3}}\psi_n.
\end{align*}
Combining this with \eqref{eq:waveeqmonomial2} yields \eqref{eq:waveeqmonomial}. A similar computation implies \eqref{eq:waveeqmonomiallog}.
\end{proof}
\begin{remark}\label{rmk:waveeqmonomial}
We see from Lemma \ref{lem:waveeqmonomial} that for all but finitely many $n$, we have
\begin{gather*}
\square_g (r^{-n}\psi_n) = O(r^{-n}),\quad\square_g(r^{-n}\log r \>\psi_{n,1}) = O(r^{-n}\log r).
\end{gather*}
The three exceptions are $\psi_0,\psi_3,\psi_{3,1}$, where the leading order term vanishes, giving
\begin{gather*}
\square_g \psi_0 = O(r^{-2}),\quad \square_g (r^{-3}\psi_3) = O(r^{-5}),\quad \square_g (r^{-3}\log r\>\psi_{3,1}) = O(r^{-3}). 
\end{gather*}
\end{remark}
In the next Lemma we consider a general asymptotic solution $\psi_{\text{asymp}}^{(N)}$ of the form \eqref{eq:asympsolhighorder}. We will show that $\psi_{\text{asymp}}^{(N)}$ satisfies the linear wave equation up to a particular order in $r$ if and only if the terms $\psi_{n}$,$\psi_{n,1}$ satisfy recurrence relations. Moreover, we will show that $\psi_0,\psi_3$ capture the correct notion of scattering data for the asymptotic solution.
\begin{lemma}\label{lem:recurrencerelations}
Let $\psi_0,\psi_{1,1},\psi_1,\dots,\psi_{N,1},\psi_N$ be sufficiently regular functions on $\Sigma^+$, and let $\psi_{\text{asymp}}^{(N)}$ be of the form \eqref{eq:asympsolhighorder}. Then
\begin{equation}\label{eq:asympdecaygeneral}
\square_g \psi_{\text{asymp}} = O\big(\frac{\log r}{r^{N+1}}\big)
\end{equation}
if and only if $\psi_0,\psi_1,\psi_{1,1},\dots,\psi_N,\psi_{N,1}$ satisfy $2N - 1$ recurrence relations. These recurrence relations have two degrees of freedom, with $\psi_0,\psi_3$ being the leading-order terms that can be freely chosen. The recurrence relations for the logarithmic terms reduce to
\begin{equation}\label{eq:logrecurrence}
\psi_{n,1} = 0,\quad n = 1,\dots,N,
\end{equation}
while the recurrence relations for the other terms reduce to
\begin{gather*}
\psi_{1} = 0, \quad \psi_{2} = -\frac{3}{\Lambda}\tilde{\Delta}\psi_0,\quad \psi_{n} = R_n(\psi_0,\psi_3),\quad n \geq 4,
\end{gather*}
where the $R_n$ are linear functions of $\psi_0,\psi_3$ and their derivatives.
\end{lemma}
\begin{proof}
This is a straightforward application of Lemma~\ref{lem:waveeqmonomial}. In general, we have
$$\square_g \psi_{\text{asymp}} = \sum_{n=0}^{N} \frac{F_n}{r^n} + \sum_{n=1}^N \frac{\log r\>F_{n,1}}{r^n} + O\big(\frac{\log r}{r^{N+1}}\big),$$ where the coefficients $F_n$ are functions of $\psi_0,\dots,\psi_n$, $\psi_{1,1}\dots,\psi_{n,1}$, while the coefficients $F_{n,1}$ are functions of $\psi_{1,1}\dots,\psi_{n,1}$ only. From this we see that \eqref{eq:asympdecaygeneral} holds if and only if $F_0 = \dots = F_n = 0$, and $F_{1,1} = \dots = F_{n,1} = 0$, and these generate $2N+1$ recurrence relations between the $\psi_n,\psi_{n,1}$. We compute these recurrence relations in order of decay, simplifying as needed. As $\square_g \psi_0 = O(r^{-2})$, the  recurrence relation $F_0 = 0$ is trivial, and the first two nontrivial recurrence relations are 
$$F_{1,1} = \frac{2\Lambda}{3}\psi_{1,1} = 0,\quad F_{1} =\frac{2\Lambda}{3} \psi_1 - \LT \psi_{1,1} = 0,$$
which reduce to $\psi_{1,1} = \psi_1 = 0$. Assuming this holds, the next two recurrence relations are 
$$F_{2,1} = \frac{2\Lambda}{3}\psi_{2,1} = 0,\quad F_2 = \tilde{\Delta}\psi_0 + \frac{2\Lambda}{3}\psi_2 + \LT \psi_{2,1} = 0.$$
These reduce to $\psi_{2,1} = 0$ and $\psi_2 = -\frac{3}{2\Lambda}\tilde{\Delta}\psi_0$. This implies that $\psi_0$ can be freely chosen. By Lemma \ref{lem:waveeqmonomial}, we see that
$$\square_g \psi_0 = \frac{1}{r^2}\tilde{\Delta}\psi_0 + O\big(\frac{1}{r^4}\big)\quad \square_g \Big(\frac{\psi_2}{r^2}\Big) = \frac{2\Lambda}{3} \frac{\psi_2}{r^2} + O\big(\frac{1}{r^4}\big).$$
Moreover, we have $\square_g (r^{-3}\psi_3) = O(r^{-5})$, $\square_g (r^{-3}\log r\>\psi_{3,1}) = O(r^{-3})$, and so the recurrence relation $F_{3,1} = 0$ is trivial, while the recurrence relation for $F_{3}$ is 
$$F_{3,1} = \frac{\Lambda}{3} \psi_{3,1} = 0.$$
For $n \geq 4$, the coefficients $F_{n,1}$ are functions of $\psi_{1,1},\dots,\psi_{n,1}$ only. Since $\square_g (r^{-n}\log r\>\psi_{n,1}) = O(r^{-n}\log r)$ for $n \geq 4$, the recurrence relations for the $F_{n,1}$ reduce to $\psi_{n,1} = 0$ for all $4 \leq n \leq N$. Thus the coefficients of all log terms in $\psi_{\text{asymp}}^{(N)}$ are identically zero only if \eqref{eq:asympdecaygeneral} holds. While we do not compute the remaining recurrence relations explicitly, we note that the relation $F_4 = 0$ implies that $\psi_4$ is determined entirely by $\psi_0$, and for $F_5$ we have
$$F_5 = F_5(\psi_0,\psi_2,\psi_3,\psi_4,\psi_5).$$
As $\psi_2,\psi_4$ are determined by $\psi_0$, this recurrence relation also has a degree of freedom, with $\psi_3$ being the next leading-order term that can be freely chosen. Hence the remaining recurrence relations imply that the remaining terms are all determined by $\psi_0,\psi_3$.
\end{proof}
\begin{remark}\label{rmk:logterms}
Lemma~\ref{lem:recurrencerelations} depends heavily on the asymptotic expansion of the components of the metric \eqref{eq:metricstandardcoords} for Schwarzschild-de Sitter. In particular, the fact that $\psi_{3,1}$ vanishes identically relies on the fact that $\phi^{-2} = \LT r^2 + O(1)$, i.e. $\phi^{-2}$ contains no $O(r)$ term in its expansion. One can easily construct perturbations where this is not the case, leading to nonzero logarithmic terms in $\psi_{\text{asymp}}^{(N)}$. We will discuss this again when we construct scattering solutions in the perturbed setting in Section~\ref{sec:perturbations}, see in particular Lemma~\ref{lem:asympperturbed}. For the moment we emphasise that for Schwarzschild-de Sitter log terms are not present in any asymptotic solution satisfying \eqref{eq:asympdecaygeneral}.
\end{remark}
To prove existence and uniqueness of scattering solutions, we will consider the approximate solution $\psi_{\text{asymp}}^{(N)}$ with $N=3$,\footnote{With Lemma~\ref{lem:recurrencerelations}, one may construct an asymptotic solution that satisfies the wave equation up to arbitrary order in powers of $r^{-1}$ by including more terms in the asymptotic solution. By the argument given in this paper one could use this asymptotic solution to construct a true scattering solution to the wave equation with more detailed asymptotics. This would require additional regularity of the scattering data $\psi_0$, $\psi_3$ though. For the purpose of constructing exact scattering solutions to $\square_g \psi = 0$, we only require decay of our asymptotic solution like $$\square_g \psi_{\text{asymp}} = O\big(\frac{1}{r^{3+\epsilon}}\big)$$
for some $\epsilon > 0$; \eqref{eq:asympsolboundl2} in Lemma~\ref{lem:asympsolboundl2} and \eqref{eq:vest3}.} and fix $\psi_{\text{asymp}} = \psi_{\text{asymp}}(\psi_0,\psi_3)$ to be the asymptotic solution
$$\psi_{\text{asymp}} = \psi_0 - \frac{3}{2\Lambda}\frac{\tilde{\Delta}\psi_0}{r^2} + \frac{\psi_3}{r^3}.$$
It follows from Lemma~\ref{lem:recurrencerelations} that $\square_g \psi_{\text{asymp}} = O(r^{-4}).$ More specifically we have the pointwise bound 
\begin{equation}\label{eq:boxpsiasymp}
|\square_g \psi_{\text{asymp}}| \lesa_{\Lambda,m,r_0} \frac{1}{r^4}\sum_{|\alpha|\leq 4}|\tilde{\nabla}^\alpha\psi_0| +\frac{1}{r^5}\sum_{|\alpha|\leq 2}|\tilde{\nabla}^\alpha\psi_3|.
\end{equation}
In the following Lemma we show that the pointwise decay \eqref{eq:boxpsiasymp} implies certain bounds on the $L^2$-and energy norms of the asymptotic solution. Specifically  we estimate the energy of the asymptotic solution $\psi_{\text{asymp}}$ with respect to the natural energy norm \eqref{eq:naturalenergy}. We will also show that the quantity
$$\|r\phi^{1/2}\square_g \psi_{\text{asymp}}\|_{L^2(\Sigma_r)}$$
is finite and integrable in $r$, given appropriate assumptions on the regularity and decay of the scattering data $\psi_0,\psi_3$. The importance of this property, as well as the particular weights in $r$ are motivated by its inclusion in the weighted energy estimate that is proved in Proposition~\ref{prop:backwardenergyestimate}, see \eqref{eq:backwardenergyestimate} in particular.

\begin{lemma}\label{lem:asympsolboundl2}
Fix $r_0 > r_\mc{C}$, and let $\psi_0$, $\psi_3$ be functions on $\Sigma^+$ such that
$$\|\psi_0\|_{H^4(\RR\times\sph^2)},\>\>\|\psi_3\|_{H^2(\RR\times\sph^2)} < \infty.$$
Then for all $r > \rc$, $\|\psi_{\text{asymp}}\|_{\pa_r, \Sigma_r} < \infty$, and
\begin{equation}\label{eq:asympsolenergybound}
\|\psi_{\text{asymp}}\|_{\pa_r, \Sigma_r} \lesa_{\Lambda,m,r_0} \|\tilde{\nabla} \psi_0\|_{H^2(\RR\times\sph^2)} + \frac{1}{r^2}\|\psi_3\|_{H^1(\RR\times\sph^2)}.
\end{equation}
Moreover, we have $\square_g \psi_{\text{asymp}} \in L^2(\Sigma_r)$, and
\begin{equation}\label{eq:asympsolboundl2}
\int_r^{\infty}\|r\phi^{1/2}\square_g \psi_{\text{asymp}}\|_{L^2(\Sigma_r)} \>\de r \lesa_{\Lambda,m,r_0} \frac{1}{r}\|\psi_0\|_{H^4(\RR\times\sph^2)}+ \frac{1}{r^2}\|\psi_3\|_{H^2(\RR\times\sph^2)}
\end{equation}
for all $r \geq r_0$.
\end{lemma}
\begin{proof}
An application of the triangle inequality and the definition of $\psi_{\text{asymp}}$ implies
$$\|\psi_{\text{asymp}}\|_{\pa_r,\Sigma_r} \lesa \|\psi_0\|_{\pa_r,\Sigma_r} + \|r^{-2}\tilde{\Delta}\psi_0\|_{\pa_r,\Sigma_r} + \|r^{-3}\psi_3\|_{\pa_r,\Sigma_r}.$$
From the definitions of the $\|\cdot\|_{\pa_r,\Sigma_r}$ and $\|\cdot\|_{L^2(\Sigma_r)}$ norms, we have
$$\|\psi\|_{\pa_r, \Sigma_r} \lesa_{\Lambda,m,r_0} r^2\|\pa_r \psi_{\text{asymp}}\|_{L^2(\RR\times\sph^2)} + \|\tilde{\nabla}\psi_{\text{asymp}}\|_{L^2(\RR\times\sph^2)},$$
and so
\begin{align*}
\|\psi_{\text{asymp}}\|_{\pa_r, \Sigma_r} \lesa_{\Lambda,m,r_0}\> \|\tilde{\nabla}\psi_0\|_{L^2(\RR\times\sph^2)} +\frac{1}{r}\|\tilde{\Delta}\psi_0\|_{L^2(\RR\times\sph^2)} + \frac{1}{r^2}\|\tilde{\nabla}(\tilde{\Delta}\psi_0)\|_{L^2(\RR\times\sph^2)}\\*+\frac{1}{r^2}\|\psi_3\|_{L^2(\RR\times\sph^2)}+\frac{1}{r^3}\|\tilde{\nabla}\psi_3\|_{L^2(\RR\times\sph^2)},
\end{align*}
which implies \eqref{eq:asympsolenergybound}.

To prove \eqref{eq:asympsolboundl2}, observe that by \eqref{eq:boxpsiasymp} and
the fact that $\de \mu_{\ol{g}_r} = r^2\phi^{-1}\de \mu_{\tilde{g}}$, we have
$$\|r\phi^{1/2}\square_g \psi_{\text{asymp}}\|_{L^2(\Sigma_r)} \lesa_{\Lambda,m,r_0} \frac{1}{r^2}\|\psi_0\|_{H^4(\RR\times\sph^2)} + \frac{1}{r^3}\|\psi_3\|_{H^2(\RR\times\sph^2)}.$$
Therefore $\|r \phi^{1/2}\square_g \psi_{\text{asymp}} (r)\|_{L^2(\Sigma_r)}$ is integrable in $r$, and integrating in $r$ yields \eqref{eq:asympsolboundl2}.
\end{proof}


\subsection{Weighted energy estimate for the backward problem}\label{sec:backwardestimate}
We now prove the weighted energy estimate of \textbf{Proposition~\ref{prop:backwardenergyestimate}} for the wave equation on the expanding region of Schwarzschild-de Sitter. We identify a vectorfield multiplier whose associated energy is decreasing backwards in time, relating to a ``global redshift'' on the expanding region that was first explored in \cite{Sch15}. In that paper however, energy estimates are constructed for forward solutions to \eqref{eq:waveequation}, and so the vectorfield multiplier in the present paper differs from that in \cite{Sch15}. This energy estimate will be applied not to solutions of \eqref{eq:waveequation}, but to a small remainder, obtained by subtracting an asymptotic solution which we constructed in Section~\ref{sec:asympsol}. This remainder will satisfy an inhomogeneous wave equation.

Following notation from \cite{Sch22}, let $\mc{D}_{r_1,r_2}^{(\ol{u},\ol{v})}\subset \mc{R}^+$ denote the spacetime region
\begin{equation}\label{eq:backwardestimatedomain}\mc{D}_{r_1,r_2}^{(\ol{u},\ol{v})} = \bigcup_{r_1 \leq r \leq r_2}\Sigma_r \bigcap \{u \geq \ol{u}\}\bigcap \{v \geq \ol{v}\},\end{equation}
as depicted in Figure \ref{fig:backwardestimatedomain}. Here, $u,v$ are the null Kruskal coordinates defined in Section~\ref{sec:notation}. It is apparent that $\mc{D}_{r_1,r_2}^{(\ol{u},\ol{v})}$ is bounded by the two spacelike hypersurfaces $\Sigma_{r_1}^c$, $\Sigma_{r_2}^c$, and two null hypersurfaces, which we denote by $\mc{F}_{\ol{u}}^c,\mc{G}_{\ol{v}}^c$. The superscript $c$ denotes that these hypersurfaces are `capped'.
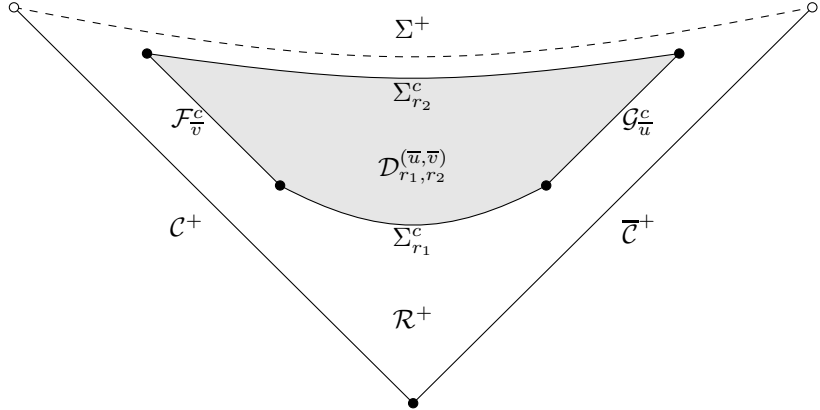
\begin{figure}[t]

\centering
\begin{tikzpicture}[scale=1.75]
\draw (0,3) -- (3,0);
\draw (3,0) -- (6,3);
\draw[dashed] (0,3) .. controls (2.5,2.5) and (3.5,2.5) .. (6,3);
\filldraw  (3,0) circle (0.035cm);
\filldraw[fill=white] (0,3) circle (0.035cm);
\filldraw[fill=white] (6,3) circle (0.035cm);
\draw (1.3,1.5) node[below]{$\mc{C}^+$};
\draw (4.7,1.5) node[below]{$\ol{\mc{C}}^+$};
\draw (3,2.7) node[above]{$\Sigma^+$};
\draw (3,0.5) node[above]{$\mc{R}^+$};
\filldraw[fill=gray!20] (2,1.65) -- (1,2.65) .. controls (2.75,2.4) and (3.25,2.4) .. (5,2.65) -- (4,1.65) .. controls (3.25,1.25) and (2.75,1.25) .. (2,1.65);
\filldraw  (2,1.65) circle (0.035cm);
\filldraw  (1,2.65) circle (0.035cm);
\filldraw  (4,1.65) circle (0.035cm);
\filldraw  (5,2.65) circle (0.035cm);
\draw (3,2.5) node[below]{$\Sigma_{r_2}^c$};
\draw (3,1.4) node[below]{$\Sigma_{r_1}^c$};
\draw (3,2) node[below]{$\mc{D}_{r_1,r_2}^{(\ol{u},\ol{v})}$};
\draw (1.5,2.3) node[below left]{$\mc{F}_{\ol{v}}^c$};
\draw (4.5,2.3) node[below right]{$\mc{G}_{\ol{u}}^c$};
\end{tikzpicture}
\caption{Spacetime domain $\mc{D}_{r_1,r_2}^{(\ol{u},\ol{v})}$ of the backwards weighted energy estimate.}\label{fig:backwardestimatedomain}
\end{figure}
Let $X$ be a timelike vectorfield. Given some scalar function $\psi$, we apply the divergence theorem to the energy current $J^{X}[\psi]$ on $\mc{D}_{r_1,r_2}^{(\ol{u},\ol{v})}$, which gives
\begin{multline}\label{eq:divergencetheorem}
\int_{\Sigma_{r_1}^c} J^{X}[\psi] \cdot n_{\Sigma_r} +\int_{\mc{F}_{\ol{v}}^c} \>^*J^{X}[\psi] + \int_{\mc{G}_{\ol{u}}^c} \>^*J^{X}[\psi]\\*
= \int_{\Sigma_{r_2}^c} J^{X}[\psi] \cdot n_{\Sigma_r} \de \mu_{\ol{g}_r} + \int_{\mc{D}_{r_1,r_2}^{(\ol{u},\ol{v})}}\nabla \cdot J^{X}[\psi] \de \mu_{g}
,
\end{multline}
where $ \>^*J^{M}[\psi]$ denotes the Hodge dual of the energy current $J^M[\psi]$. 
Note that the boundary integrals in \eqref{eq:divergencetheorem} are over spacelike or null hypersurfaces. For a timelike future pointing vectorfield $X$, these boundary integrals will all be nonnegative. The key to producing energy estimates for wave equations lies in identifying vectorfield multipliers $X$ for which the divergence of the energy current $\nabla \cdot J^X[\psi]$ is coercive in some sense, so that we may control the bulk integral in \eqref{eq:divergencetheorem}. By the Leibniz product rule we have
\begin{equation}\label{eq:currentdivergence}
\nabla \cdot J^X [\psi] = K^X[\psi] + (X\psi) (\square_g \psi),
\end{equation}
where
\begin{equation}\label{eq:seconcurrentdef}
K^X[\psi] = \>^{(X)}\pi^{\mu\nu}T_{\mu\nu}[\psi].
\end{equation}
We define the vectorfield
\begin{equation}\label{eq:vfbackwardenergy}
M \doteq r^2\phi^{-2}\frac{\pa}{\pa r} = r^2\Big(\LT r^2 - 1 + \frac{2m}{r}\Big)\frac{\pa}{\pa r}
\end{equation}
We emphasize that $M$ is timelike, and for large $r$, $M \sim r^4\pa_r$. In the following proposition, we will show that the associated energy current \eqref{eq:seconcurrentdef} is nonpositive.
\begin{proposition}\label{prop:vfbackwardenergy}
Let $M$ denote the vectorfield \eqref{eq:vfbackwardenergy}. Then everywhere on $\mc{R}^+$, we have
$$K^M[\psi] \leq 0.$$
\end{proposition}
\begin{proof}
The deformation tensor of $M$ is given by
\begin{align*}^{(M)}\pi^{\mu\nu} &= \frac{1}{2}(\nabla^{\mu}M^{\nu}+\nabla^{\nu}M^{\mu})\\*
&= \frac{1}{2}\Big(g^{\mu\lambda}\pa_\lambda M^\nu + g^{\nu\lambda}\pa_\lambda M^\mu+g^{\mu\lambda}\Gamma_{\lambda\eta}^\nu M^\eta + g^{\nu\lambda}\Gamma_{\lambda\eta}^\mu M^\eta\Big).\\*
\end{align*}
From \eqref{eq:metricstandardcoords}, we find that the nonzero Christoffel symbols of Schwarzschild-de Sitter on $\mc{R}^+$ are 
\begin{gather*}
\Gamma_{rr}^r = -\Big(\LT r - \frac{m}{r^2}\Big)\phi^2, \quad \Gamma_{rt}^t = \Big(\LT r - \frac{m}{r^2}\Big)\phi^2, \quad \Gamma_{tt}^r = \Big(\LT r - \frac{m}{r^2}\Big)\phi^{-2}.\\*
\Gamma_{AB}^r = \frac{\phi^{-2}}{r}g_{AB}, \quad \Gamma_{B r}^A = \frac{1}{r}\delta_B^A .
\end{gather*}
Another computation reveals that the nonzero components of the deformation tensor of $M$ are
\begin{gather*}
\>^{(M)}\pi^{tt} = \big(\LT r^3-m\big)\phi^2,\quad \>^{(M)}\pi^{rr} = \frac{\LT r^3 - m}{\phi^2} - \frac{r}{\phi^2}\Big(\frac{4\Lambda}{3} r^2 - 2 + \frac{2m}{r}\Big).\\*
\>^{(M)}\pi^{\theta\theta} = \frac{1}{r\phi^2},\quad \>^{(M)}\pi^{\varphi\varphi} = \frac{1}{r\phi^2\sin^2\theta}.
\end{gather*}

Then the corresponding components of the stress-energy-momentum tensor are

\begin{gather*}
T_{tt}[\psi] = \frac{1}{2}\Big((\pa_t\psi)^2 + \frac{1}{\phi^4}(\pa_r\psi)^2 - \frac{1}{\phi^2}|\slashed\nabla\psi|^2 \Big).\\*
T_{rr}[\psi] = \frac{1}{2} \Big(\phi^4(\pa_t\psi)^2 + (\pa_r\psi)^2 + \phi^2|\slashed\nabla \psi|^2 \Big).\\*
T_{\theta\theta}[\psi] = \frac{1}{2} \Big(-r^2\phi^2(\pa_t \psi)^2 + \frac{r^2}{\phi^2}(\pa_r \psi)^2+(\pa_\theta\psi)^2  - \frac{1}{\sin^2\theta}(\pa_{\varphi}\psi)^2 \Big).\\*
T_{\varphi\varphi}[\psi] = \frac{\sin^2\theta}{2} \Big(-r^2\phi^2(\pa_t \psi)^2 + \frac{r^2}{\phi^2}(\pa_r \psi)^2 - (\pa_\theta\psi)^2  + \frac{1}{\sin^2\theta} (\pa_{\varphi}\psi)^2 \Big).
\end{gather*}
We then compute the sum $\>^{(M)}\pi^{\mu\nu} T_{\mu\nu}[\psi]$. The sum of the angular components are
\begin{equation*}\>^{(M)}\pi^{\theta\theta}T_{\theta\theta}[\psi] + \>^{(M)}\pi^{\varphi\varphi}T_{\varphi\varphi}[\psi] = -r(\pa_t \psi)^2 + \frac{r}{\phi^4}(\pa_r \psi)^2,
\end{equation*}
while the time and radial components sum to
\begin{align*}\>^{(M)}\pi^{tt}T_{tt}[\psi] + \>^{(M)}\pi^{rr}T_{rr}[\psi] &= \big(\LT r^3-m\big)\phi^2(\pa_t\psi)^2 + \frac{\LT r^3-m}{\phi^2} (\pa_r \psi)^2 \\*
&- \frac{r}{\phi^2}\Big(\frac{2\Lambda}{3} r^2 -1 + \frac{m}{r}\Big)\Big(\phi^4(\pa_t\psi)^2 + (\pa_r\psi)^2+\phi^2|\slashed\nabla\psi|^2\Big).
\end{align*}
Adding all terms together gives
\begin{align*}
\>^{(M)}\pi^{\mu\nu} T_{\mu\nu}[\psi] &= (r-3m)\phi^2(\pa_t \psi)^2 + \frac{r}{\phi^2}\Big(\frac{2\Lambda}{3}r^2-1+\frac{m}{r}\Big)(\pa_r \psi)^2 \\*
&- \frac{r}{\phi^2}\Big(\frac{2\Lambda}{3} r^2 - 1+\frac{m}{r}\Big)\Big(\phi^4(\pa_t\psi)^2 + (\pa_r\psi)^2+\phi^2|\slashed\nabla\psi|^2\Big)\\*
&=\phi^2\Big(r-3m-r\Big(\frac{2\Lambda}{3}r^2-1+\frac{m}{r}\Big)\Big)(\pa_t\psi)^2-r\Big(\frac{2\Lambda}{3}r^2-1+\frac{m}{r}\Big)|\slashed\nabla\psi|^2.\\*
&=-2r(\pa_t\psi)^2-r\Big(\frac{2\Lambda}{3}r^2-1+\frac{m}{r}\Big)|\slashed\nabla\psi|^2
\end{align*}
The function $\frac{2\Lambda}{3} r^2 - 1+\frac{m}{r}$
is positive on $\mc{R}^+$, since we have
\begin{equation*}
\Big(\frac{2\Lambda}{3} r^2 - 1+\frac{m}{r}\Big) = \frac{2}{\phi^2} + \frac{1}{r}(r-3m),
\end{equation*}
and $r > \rc > 3m$. Hence $K^M[\psi] \leq 0$.
\end{proof}

\begin{remark}
Due to the identity \eqref{eq:currentdivergence}, an immediate consequence of Proposition~\ref{prop:vfbackwardenergy} is that the divergence of the standard energy current $J^M[\psi]$ obeys the bound\begin{equation}\label{eq:currentdivergencebound}
\nabla \cdot J^M[\psi] \leq (M\psi)(\square_g \psi).
\end{equation}
\end{remark}
\begin{remark}
The multiplier $M$ and its associated bound \eqref{eq:currentdivergencebound} is sharp in the following sense. Given an exponent $\alpha > 0$, the vectorfield $r^\alpha M$ is also timelike on $\mc{R}^+$, and possesses the property that the associated energy current $K^{r^\alpha M}$ obeys the bound
\begin{equation}\label{eq:currentalpha}
\phi K^{r^\alpha M} \leq - \frac{\alpha}{r}J^{r^\alpha M}[\psi]\cdot n_{\Sigma_r}.
\end{equation}
To see this, observe that by the product rule we have
\begin{align*}
\>^{(r^\alpha M)}\pi^{\mu\nu} &= r^{\alpha}\>^{(M)}\pi^{\mu\nu} + \frac{1}{2} \Big(g^{\mu\lambda}\pa_\lambda(r^\alpha) M^\nu +g^{\nu\lambda}\pa_\lambda(r^\alpha) M^\mu\Big)\\*
&=r^{\alpha}\>^{(M)}\pi^{\mu\nu} - \frac{1}{2}\alpha r^{\alpha-1}\phi^{-2}\Big(\delta_{r}^\mu M^\nu + \delta_{r}^\nu M^\nu\Big),
\end{align*}
and so 
\begin{equation*}
\phi K^{r^\alpha M} = r^\alpha \phi K^M - \alpha r^{\alpha-1} J^M[\psi]\cdot n_{\Sigma_r}.
\end{equation*}
As $K^M \leq 0$ and  $r^\alpha J^M[\psi] = J^{r^\alpha M}[\psi]$, we have \eqref{eq:currentalpha}. Proposition~\ref{prop:vfbackwardenergy} corresponds to the endpoint case $\alpha = 0$. Through a Gr\"onwall-type inequality, \eqref{eq:currentalpha} implies an energy estimate similar to \eqref{eq:backwardenergyestimate}. If $\square_g \psi = 0$, the resulting energy estimate is in fact the same. Howeve we will be applying this energy estimate to solutions of the inhomogeneous wave equation $\square_g \psi = F$, and the term containing $F$ in \eqref{eq:backwardenergyestimate} is multiplied by an extra factor of $r^\alpha$, so the estimate is no longer sharp for $\alpha > 0$.
\end{remark}

\begin{proof}[Proof of Proposition~\ref{prop:backwardenergyestimate}]
We will apply the divergence theorem to the energy current $J^M[\psi]$ on the spacetime region $\mc{D}_{r_1,r_2}^{(\ol{u},\ol{v})}$ defined in \eqref{eq:backwardestimatedomain}, which implies by \eqref{eq:divergencetheorem} that
\begin{multline*}
\int_{\Sigma_{r_1}^c} J^{M}[\psi] \cdot n_{\Sigma_r} +\int_{\mc{F}_{\ol{v}}^c} \>^*J^{M}[\psi] + \int_{\mc{G}_{\ol{u}}^c} \>^*J^{M}[\psi]\\*
= \int_{\Sigma_{r_2}^c} J^{M}[\psi] \cdot n_{\Sigma_r} \de \mu_{\ol{g}_r} + \int_{\mc{D}_{r_1,r_2}^{(\ol{u},\ol{v})}}\nabla \cdot J^{M}[\psi] \de \mu_{g},
\end{multline*}
We write the boundary integrals on $\Sigma_{r_1}^c,\Sigma_{r_2}^c$ in terms of the energy norms $\|\cdot\|_{M,\Sigma_{r}^c}$ and apply the coarea formula to the bulk term, which implies
\begin{equation*}
\|\psi\|_{M,\Sigma_{r_1}^c}^2 +  \int_{\mc{F}_{\ol{v}}^c} \>^*J^{M}[\psi] + \int_{\mc{G}_{\ol{u}}^c} \>^*J^{M}[\psi] = \|\psi\|_{M,\Sigma_{r_2}^c}^2 + \int_{r_1}^{r_2} \Big(\int_{\Sigma_r}\phi \nabla \cdot J^M[\psi]\> \de\mu_{\ol{g}_r}\Big)\de r.
\end{equation*}
We then differentiate in $r_2$, keeping $r_1,\ol{u},\ol{v}$ fixed so that
\begin{equation*}
\pa_{r_2} \Big(\int_{\mc{F}_{\ol{v}}^c} \>^*J^{M}[\psi] + \int_{\mc{G}_{\ol{u}}^c} \>^*J^{M}[\psi]\Big) = \pa_{r_2} (\|\psi\|_{M,\Sigma_{r_2})}^2) + \int_{\Sigma_{r_2}} \phi \nabla \cdot J^M[\psi] \de\mu_{\ol{g}_r}.
\end{equation*}
The left hand side of the above expression is nonnegative, as $\>^*J^M[\psi]$ restricted to a null hypersurface is nonnegative, and the hypersurfaces $\mc{F}_{\ol{v}}^c, \mc{G}_{\ol{u}}^c$ are increasing in $r_2$. Hence the right hand side is also nonnegative. 

From Proposition~\ref{prop:vfbackwardenergy} the current $K^M[\psi] \leq 0$, and so we bound
\begin{equation*}
\int_{\Sigma_{r_2}^c} \phi \nabla \cdot J^M[\psi] \de\mu_{\ol{g}_r} \leq \int_{\Sigma_{r_2}^c} \phi(M\psi)(\square_g \psi) \de\mu_{\ol{g}_r}.
\end{equation*}
We then apply the Cauchy-Schwartz inequality to the right hand side above which implies
\begin{equation*}
\int_{\Sigma_{r_2}^c} \phi(M\psi)(\square_g \psi) \de\mu_{\ol{g}_r} \leq 2\Big(\frac{1}{2}\int_{\Sigma_{r_2}^c}\frac{\phi}{r^2}(M\psi)^2\de\mu_{\ol{g}_r}\Big)^{1/2}\Big(\frac{1}{2}\int_{\Sigma_{r_2}^c}r^2\phi F^2\de\mu_{\ol{g}_r}\Big)^{1/2}.
\end{equation*}
By the definition of $M$, we have the bound
\begin{equation*}
\frac{1}{2}\frac{\phi}{r^2}(M\psi)^2 == \frac{1}{2}(n_{\Sigma_{r}}\psi)(M\psi) 
\leq J^M[\psi]\cdot n_{\Sigma_r},
\end{equation*}
therefore one can bound the first integral on the right hand side by $\|\psi\|_{M,\Sigma_{r_2}^c}$, while the second integral is equivalent to $\frac{1}{\sqrt{2}}\|r\phi^{1/2}F\|_{L^2(\Sigma_{r_2}^c)}$. Meanwhile,
$$\pa_{r_2}(\|\psi\|_{M,\Sigma_{r_2}^c}^2) = 2\|\psi\|_{M,\Sigma_{r_2}^c}\pa_{r_2}\|\psi\|_{M,\Sigma_{r_2}^c},$$
and so
$$0 \leq 2\|\psi\|_{M,\Sigma_{r_2}^c} \pa_r \|\psi\|_{M,\Sigma_{r_2}^c} + \frac{2}{\sqrt{2}}\|\psi\|_{M,\Sigma_{r_2}^c}\|r\phi^{1/2}F\|_{L^2(\Sigma_{r_2}^c)}.$$
We then divide through by $2\|\psi\|_{M,\Sigma_{r_2}^c}$ and integrate in the radial coordinate on the interval $[r_1,r_2]$, yielding a localised energy estimate,
\begin{equation}\label{eq:backwardestimatelocal}
\|\psi\|_{M,\Sigma_{r_1}^c} \leq \|\psi\|_{M,\Sigma_{r_2}^c} + \int_{r_1}^{r_2} \frac{1}{\sqrt{2}}\|r\phi^{1/2}F\|_{L^2(\Sigma_r^c)}\> \de r.
\end{equation}
The capped hypersurfaces $\Sigma_r^c$ are described precisely as
$$\Sigma_r^c = \Sigma_r \cap \{u \geq \ol{u}\}\cap \{v \geq \ol{v}\},$$
where $\ol{u},\ol{v} > 0$ to ensure that the $\Sigma_r^c$ are compact. We retrieve the global estimate \eqref{eq:backwardenergyestimate} by taking limits $\ol{u},\ol{v} \ra 0$.
\end{proof}
\begin{remark}\label{KillingCommute}
Theorem~\ref{prop:backwardenergyestimate} can be used to create an array of higher order estimates. The key observation is that the tangent spaces of the hypersurfaces $\Sigma_r$ are spanned by Killing vectorfields. These consist of the Killing vectorfield $T$ that coincides in the expanding region with the coordinate vectorfield $\pa/\pa_t$, and the generators of the spherical isometries $\Omega_i :i=1,2,3$. The wave equation on a given Lorentzian manifold commutes with the Killing vectorfields of that manifold, and so one can commute the estimate Theorem~\eqref{prop:backwardenergyestimate} with these Killing vectorfields, yielding estimates of the form
\begin{multline}\label{Redshift1Com}
\|T^k\Omega_{i_1}\cdots\Omega_{i_l}\psi\|_{M, \Sigma_{r_1}} \leq \|T^k\Omega_{i_1}\cdots\Omega_{i_l}\psi\|_{M, \Sigma_{r_2}}\\* 
+ \frac{1}{\sqrt{2}}\int_{r_1}^{r_2}\|r\phi^{1/2}T^k\Omega_{i_1}\cdots\Omega_{i_l}F\|_{L^2(\Sigma_r)}\>\de r,
\end{multline}
for $k \in \NN, i_1,\dots,i_l = 1,2,3$. We will use such estimates in Section~\ref{sec:pointwisescattering} to obtain pointwise estimates for solutions to the backward scattering problem.
\end{remark}

\subsection{Existence of scattering solutions to the wave equation on the expanding region \texorpdfstring{$\mc{R}^+$}{R plus}: proof of Theorem~\ref{thm:scatteringwithouthorizon}}\label{sec:scatteringinterior}
We now prove that solutions to $\square_g \psi = 0$ exist with prescribed scattering data on $\Sigma^+$. Fix $r_0 > \rc$. For $R \geq r_0$ let $\psi_R = \psi_{\text{asymp}} + \psi_{\text{rem}}^{R}$, where $\psi_{\text{rem}}^{R}$ solves the inhomogeneous wave equation
\begin{equation}\label{RemEq}
\square_g \psi_{\text{rem}}^{R} = -\square_g \psi_{\text{asymp}},
\end{equation}
with trivial initial data on the hypersurface $\Sigma_{R}$, so that
$$\psi_{\text{rem}}^{R}|_{\Sigma_R} = 0,\quad n_{\Sigma_R}\cdot \psi_{\text{rem}}^{R}|_{\Sigma_R} = 0.$$
Let $R_1,R_2 \geq r_0$ and assume without loss of generality that $R_1 < R_2$. Then consider the difference 
\begin{equation*}
v = \psi_{R_2} - \psi_{R_1} = \psi_{\text{rem}}^{R_2}-\psi_{\text{rem}}^{R_1}.
\end{equation*}
We will prove that $v \ra 0$ with respect to the weighted energy norm $\|\cdot\|_{M, \Sigma_r}$ as $R_1,R_2 \ra \infty$. Clearly $v$ satisfies the homogeneous wave equation $\square_g v = 0$, so we may apply the weighted energy estimate from Proposition~\ref{prop:backwardenergyestimate} to $v$ on the spacetime region 
$$\mc{D}_{(r,R_1)} = \bigcup_{r \leq \ol{r} \leq R_1} \Sigma_{\ol{r}},$$
which gives
\begin{equation}\label{eq:vest1}
\|v\|_{M,\Sigma_{r}} \leq \|v\|_{M,\Sigma_{R_1}}.
\end{equation}
As $\psi_{\text{rem}}^{R_1}$ has trivial data on $\Sigma_{R_1}$, we have
$$\|v\|_{M,\Sigma_{R_1}} =\|\psi_{\text{rem}}^{R_2}\|_{M,\Sigma_{R_1}}$$
Then we apply the weighted energy estimate from Proposition~\ref{prop:backwardenergyestimate} again, now to $\psi_{\text{rem}}^{R_2}$ on the region $\mc{D}_{(R_1,R_2)}$. This gives
\begin{equation}\label{eq:vest2}
\|\psi_{\text{rem}}^{R_2}\|_{M,\Sigma_{R_1}} \leq \|\psi_{\text{rem}}^{R_2}\|_{M,\Sigma_{R_2}} + \frac{1}{\sqrt{2}}\int_{R_1}^{R_2}\|r\phi^{1/2}\square_g \psi_{\text{asymp}}\|_{L^2(\Sigma_r)}\de r.
\end{equation}
Since $\psi_{\text{rem}}^{R_2}$ has vanishing data on $\Sigma_{R_2}$ we have $\|\psi_{\text{rem}}^{R_2}\|_{M,\Sigma_{R_2}} = 0$, and so it follows from \eqref{eq:vest1} and \eqref{eq:vest2} that 
\begin{equation}\label{eq:vest3}
\|v\|_{M,\Sigma_{r}} \leq \frac{1}{\sqrt{2}}\int_{R_1}^{R_2}\|r\phi^{1/2}\square_g \psi_{\text{asymp}}\|_{L^2(\Sigma_r)}\de r
\end{equation}
for all $r\in [r_0,R_1]$. By Lemma~\ref{lem:asympsolboundl2}, the integral on the right hand side remains bounded as $R_2 \ra \infty$, and we have the bound
$$\|v\|_{M,\Sigma_{r}} \lesa_{\Lambda,m,r_0} \frac{1}{R_1}\|\tilde{\nabla}\psi_0\|_{H^3(\RR\times\sph^2)} + \frac{1}{R_1^2}\|\psi_3\|_{H^2(\RR\times\sph^2)}.$$
Thus for each $r \geq r_0$, $\|v\|_{M,\Sigma_{r}} \ra 0$ as $R_1$, $R_2 \ra 0$. Hence there exists a function $\psi$ such that 
$$\lim_{R \ra \infty}\|\psi - \psi_R\|_{M, \Sigma_r} = 0$$
for each $r \geq r_0$. 

We now derive the estimates on $\psi$, in particular we prove \eqref{eq:scatteringenergybound}. By the triangle inequality we have
$$\|\psi_{\text{rem}}\|_{M, \Sigma_r} \leq \|\psi_{\text{rem}} - \psi_{\text{rem}}^{R}\|_{M, \Sigma_r} + \|\psi_{\text{rem}}^{R}\|_{M, \Sigma_r}.$$
We apply the weighted energy estimate Proposition~\ref{prop:backwardenergyestimate} to $\psi_{\text{rem}}^{R}$ on $\mc{D}_{(r,R)}$, so that
$$\|\psi_{\text{rem}}\|_{M, \Sigma_r} \leq \|\psi_{\text{rem}} - \psi_{\text{rem}}^{R}\|_{M, \Sigma_r} +  + \frac{1}{\sqrt{2}}\int_{r}^{R}\|s\phi^{1/2}(s)\square_g \psi_{\text{asymp}}\|_{L^2(\Sigma_r)}\de s,$$
where we used the fact that $\psi_{\text{rem}}^{R}$ has vanishing data on $\Sigma_R$, so that $\|\psi_{\text{rem}}^{R}\|_{M, \Sigma_r}=0$. Again from Lemma~\ref{lem:asympsolboundl2}, the final integral is bounded as $R \ra \infty$, and we have the bound
\begin{equation}
\|\psi_{\text{rem}}\|_{M,\Sigma_{r}} \lesa_{\Lambda,m,r_0} \|\psi_{\text{rem}} - \psi_{\text{rem}}^{R}\|_{M, \Sigma_r} + \frac{1}{r}\|\tilde{\nabla} \psi_0\|_{H^3(\RR\times\sph^2)}+ \frac{1}{r^2}\|\psi_3\|_{H^2(\RR\times\sph^2)}.\\*
\end{equation}
Letting $R \ra \infty$, $\|\psi_{\text{rem}} - \psi_{\text{rem}}^{R}\|_{M, \Sigma_r} \ra 0$, and so it follows that for $\psi - \psi_{\text{asymp}} = \psi_{\text{rem}}$ we have.
\begin{align}
\|\psi - \psi_{\text{asymp}}\|_{\pa_r, \Sigma_r} &= \|\psi_{\text{rem}}\|_{\pa_r, \Sigma_r}\nonumber\\*
&\lesa_{\Lambda,m,r_0} \frac{1}{r^2}\|\psi_{\text{rem}}\|_{M, \Sigma_r}\nonumber\\*
&\lesa_{\Lambda,m,r_0} \frac{1}{r^3}\|\tilde{\nabla}\psi_0\|_{H^3(\RR\times\sph^2)}+ \frac{1}{r^4}\|\psi_3\|_{H^2(\RR\times\sph^2)}.\label{eq:scatteringenergydecay}
\end{align}
From this and \eqref{eq:asympsolenergybound} we also have for all $r \geq r_0$,
\begin{align*}\|\psi\|_{\pa_r, \Sigma_r} &\lesa_{\Lambda,m,r_0} r^2\|\psi_{\text{rem}}\|_{M, \Sigma_r} + \|\psi_{\text{asymp}}\|_{\pa_r, \Sigma_r}\\*
&\lesa_{\Lambda,m,r_0} \|\tilde\nabla \psi_0\|_{H^3(\RR\times\sph^2)} + \|\psi_3\|_{H^2(\RR\times\sph^2)},
\end{align*}
which is \eqref{eq:scatteringenergybound}. Uniqueness follows immediately, since if we set $\psi_0 = \psi_3 = 0$, then $\psi = 0$ by the above inequality.\qed
\begin{remark}
When defining the sequence of remainder functions $\psi_{\text{rem}}^R$, we point out that the inhomogeneous wave equation \eqref{RemEq} satisfied by $\psi_{\text{rem}}^R$ does not contain a cutoff function in $r$, such as 
$$\square_g \psi_{\text{rem}}^R = \chi(r/R) \square_g \psi_{\text{asymp}},$$
where $\chi$ is a smooth, compactly supported function such that $\chi(s) = 1$ in a neighbourhood of $0$, and $\chi(s) = 0$ in particular for all $s \geq 1$. Such a cutoff function is typically employed when constructing scattering solutions to wave equations on Minkowski in this manner, see \cite{LS23} for example. However we will use cutoffs (in a slightly different manner) when extending solutions to the cosmological horizon, c.f. \eqref{eq:initialdatav}.
\end{remark}

\subsection{Pointwise estimates}\label{sec:pointwisescattering}
Pointwise estimates can be obtained by repeatedly applying the Killing vectorfields $T,\Omega_i:i=1,2,3$ to the solution $\psi$ constructed in the previous section, and then using a Sobolev estimate adapted to the hypersurfaces $\Sigma_r$. The proof for this Sobolev estimate can be found in Appendix~\ref{sec:sobolev}. In particular we prove the pointwise decay estimate~
\eqref{eq:scatteringpointwisedecay} in Theorem~\ref{thm:scatteringwithouthorizon}. 

First we establish the decay of the $L^2$-norms of the remainder $\psi_{\text{rem}}$.
\begin{lemma}\label{lem:l2est}
Let $\psi$ be the scattering solution constructed in Theorem~\ref{thm:scatteringwithouthorizon} from scattering data $\psi_0$, $\psi_3$. Then the $L^2$-norm of the remainder $\psi_{\text{rem}}$ obeys the bound 
\begin{equation}\label{L2Decay}
\|\psi_{\text{rem}}\|_{L^2(\Sigma_r)} \lesa_{\Lambda,m} \frac{1}{r^{5/2}}\|\psi_0\|_{H^4(\RR\times\sph^2)}+\frac{1}{r^{7/2}}\|\psi_3\|_{H^2(\RR\times\sph^2)}
\end{equation}
for all large $r$. Moreover, $\psi \in L^2(\Sigma_r)$ for all $r \geq r_0$ for some $r_0$ slightly larger than $\rc$, and we have
\begin{equation}\label{L2Bound}
\|\psi\|_{L^2(\Sigma_r)} \lesa_{\Lambda,m,r_0} r^{3/2}\Big(\|\psi_0\|_{H^4(\RR\times\sph^2)}+\|\psi_3\|_{H^2(\RR\times\sph^2)}\Big).
\end{equation}
\end{lemma}
\begin{proof}
We will give a sketch of the proof here as this is a simple modification of Theorem~\ref{thm:scatteringwithouthorizon}. First, we observe that \eqref{L2Bound} follows from the decay estimate \eqref{L2Decay}, and the fact that
$$\|\psi\|_{L^2(\Sigma_r)} \sim_{r_0} r^{3/2}\|\psi(r)\|_{L^2(\RR\times\sph^2)}$$
for all $r \geq r_0$. To prove \eqref{L2Decay}, we start with the straightforward bound
$$\|\varphi\|_{L^2(\Sigma_r)}\lesa r^{3/2}\Big(\frac{1}{R^{3/2}}\|\varphi\|_{L^2(\Sigma_R)} + \int_{r}^R \frac{1}{s^{3/2}}\|\pa_r \varphi\|_{L^2(\Sigma_s)}\de s\Big)$$
which is obtained via the fundamental theorem of calculus. We note that
$$\|\pa_r \varphi\|_{L^2(\Sigma_s)} \lesa \frac{1}{s^{5/2}}\|\varphi\|_{M,\Sigma_s}.$$
Through the same method of proof as that of Theorem~\ref{thm:scatteringwithouthorizon}, one can show that the sequence of functions $\psi_{\text{rem}}^R$ that solve $\square_g\psi_{\text{rem}}^R = -\square_g \psi_{\text{asymp}}$ with trivial data on $\Sigma_R$ have a limit as $R \ra \infty$. Moreover, the uniqueness statement from Theorem~\ref{thm:scatteringwithouthorizon} implies that this limiting function agrees with the limiting function found in Theorem~\ref{thm:scatteringwithouthorizon}. The bound \eqref{L2Decay} is obtained by estimating
\begin{align}
\|\psi_{\text{rem}}\|_{L^2(\Sigma_r)} &\leq \|\psi_{\text{rem}}^R\|_{L^2(\Sigma_r)} + \|\psi_{\text{rem}}-\psi_{\text{rem}}^R\|_{L^2(\Sigma_r)}\nonumber\\*
&\lesa \int_r^R\frac{1}{s^4}\|\psi_{\text{rem}}^R\|_{M,\Sigma_s}\>\de s + \|\psi_{\text{rem}}-\psi_{\text{rem}}^R\|_{L^2(\Sigma_r)}\label{L2Decayproof}
\end{align}
then employing the weighted energy estimate from Proposition~\ref{prop:backwardenergyestimate} to bound the integrand
$$\frac{1}{r^4}\|\psi_{\text{rem}}^R\|_{M,\Sigma_r}\lesa \frac{1}{r^5}\|\tilde{\nabla}\psi_0\|_{H^3(\RR\times\sph^2)}+\frac{1}{r^6}\|\psi_3\|_{H^2(\RR\times\sph^2)}.$$ 
Inserting this bound into \eqref{L2Decayproof} and taking the limit $R\ra \infty$ implies \eqref{L2Decay}.
\end{proof}

\begin{proof}[Proof of pointwise decay \texorpdfstring{\eqref{eq:scatteringpointwisedecay}}{equation 1.3}]
We will apply the Sobolev estimate from Appendix~\ref{sec:sobolev} to $\psi_{\text{rem}}$. Let $I=(I_1,I_2,I_3,I_4)$ be a multiindex, and let $Z^I$ denote the differential operator
$$Z^I = \Omega_1^{I_1}\Omega_2^{I_2}\Omega_3^{I_3}T^{I_4}.$$ By Proposition~\ref{SobEst}, we have
\begin{align}
\sup_{\Sigma_r}| \psi_{\text{rem}}|^2 \lesa_{\Lambda,m,r_0}& \frac{1}{r^3}\Big\{\| \psi_{\text{rem}}\|_{L^2(\Sigma_r)}^2 + \|(T \psi_{\text{rem}})\|_{L^2(\Sigma_r)}^2 + \sum_{i=1}^3\|\Omega_i \psi_{\text{rem}}\|_{L^2(\Sigma_r)}^2\nonumber\\*
&\quad+\sum_{i=1}^3\|\Omega_i T \psi_{\text{rem}}\|_{L^2(\Sigma_r)}^2 + \sum_{i,j=1}^3\|\Omega_i\Omega_j \psi_{\text{rem}}\|_{L^2(\Sigma_r)}^2\Big\}\nonumber\\*
\lesa_{\Lambda,m,r_0}& \frac{1}{r^3}\sum_{|I| \leq 2} \|Z^I \psi_{\text{rem}}\|_{L^2(\Sigma_r)}^2.\label{SobEstRem}
\end{align}
As the operators $Z^I$ are composed of Killing vector fields, they commute with the wave equation, and so $Z^I\psi_{\text{rem}}$ solves the equation
$$\square_g (Z^I\psi_{\text{rem}}) = -\square_g (Z^I\psi_{\text{asymp}}).$$
Moreover, the function $Z^I\psi_{\text{asymp}}$ remains a good asymptotic solution, since
$$Z^I \psi_{\text{asymp}} = Z^I \psi_0 + \frac{1}{r^2} Z^I\psi_2 + \frac{1}{r^3}Z^I\psi_3,$$
where
$$Z^I \psi_2 = -\frac{3}{2\Lambda}\tilde{\Delta}(Z^I\psi_0).$$
By Theorem~\ref{thm:scatteringwithouthorizon}, the solution $\varphi$ to the wave equation with scattering data $Z^I \psi_0$, $Z^I \psi_3$ is unique, and so the function $\psi_{\text{rem}} = \varphi - Z^I \psi_{\text{asymp}}$ obeys the $L^2$-based decay estimate \eqref{L2Decay} from Propostion \ref{lem:l2est}, i.e.
$$\|Z^I\psi_{\text{rem}}\|_{L^2(\Sigma_r)} \lesa \frac{1}{r^{5/2}}\|\tilde{\nabla}(Z^I\psi_0)\|_{H^3(\RR\times\sph^2)}+\frac{1}{r^{7/2}}\|Z^I\psi_3\|_{H^2(\RR\times\sph^2)}.$$
Substituting this into \eqref{SobEstRem}, the result now follows.
\end{proof}

\subsection{Nondegenerate energy estimates on the cosmological horizon}\label{sec:nondegcurrent}
We now prove that the scattering solutions from Theorem~\ref{thm:scatteringwithouthorizon} extend to the cosmological horizon, provided the scattering data $\psi_0,\psi_3$ decay exponentially quickly along $\Sigma^+$. 

We first prove a nondegenerate energy estimate in a neighbourhood of the cosmological horizons. Recall the redshift vectorfield $N$ from \cite{Sch15} that captures the local redshift effect on the cosmological horizons $\mc{C}^+$.
Here we slightly modify the construction of $N$ so that it is suitable for energy estimates backward-in-time. We begin by introducing on $\mc{C^+}$ the null frame $\{T,Y,E_A:A=1,2\}$ comprised of the Killing vector field $T$ introduced in Section~1.2, $Y$ which is null and conjugate to $T$, and an orthonormal frame $E_A:A=1,2$ that is Lie transported by $T$ along the cosmological horizons. See \eqref{eq:killingvectorfielddef}, \eqref{eq:yvectorfielddef} for formulae for $T,Y$ on $\mc{C}^+$. As in \cite{Sch15}, we extend $Y$ off $\mc{C}^+$ into the interior of the expanding region $\mc{R}^+$ like
\begin{equation}\label{eq:defyext}
\nabla_Y Y = -\sigma(Y+T).
\end{equation}
This construction, capturing the local redshift effect of Killing horizons, was initially carried out in \cite{DR09,DR13}, and was adapted to a neighbourhood to the future of the cosmological horizons in \cite{Sch15}. In our construction we choose the constant $\sigma$ to be negative.
\begin{proposition}\label{prop:ncurrent}
For $r_0 > \rc$, let $\mc{D}_{(\rc,r_0)}$ denote the spacetime region $\{(u,v): \rc \leq r(u,v) \leq r_0\}$.
Let $\sigma < 0$ be some sufficiently negative constant, and let $Y$ denoted the vectorfield define in \eqref{eq:yvectorfielddef} and \eqref{eq:defyext}. The vectorfield multiplier $N = Y+T$ is strictly timelike on $\mc{D}_{(\rc,r_0)}$, and has the property that there exists some constant $c = c(\Lambda,m) > 0$ such that on $\mc{D}_{(\rc,r_0)}$ we have
\begin{equation}\label{eq:ncurrentbound}
K^N[\psi] \leq \frac{c}{u}J^N\cdot \frac{\pa}{\pa v}.
\end{equation}
\begin{proof}
From \cite{DR13} and \cite{Sch15}, we have
\begin{equation}\label{eq:ycurrenthorizon}
K^Y[\psi]\Big|_{\mc{C}^+} = -\frac{1}{2} |\sigma| (T\psi)^2 + \frac{1}{2}\kappa_{\mc{C}}(Y\psi)^2 - \frac{1}{2}|\sigma||\slashed\nabla \psi|^2 + \frac{2}{\rc}(T\psi)(Y\psi).
\end{equation}
By the Cauchy-Schwartz inequality, we may bound
$$\frac{2}{r}(T\psi)(Y\psi) \leq \frac{1}{2}|\sigma|(T\psi)^2 + \frac{2}{\rc^2|\sigma|}(Y\psi)^2,$$
and so by \eqref{eq:ycurrenthorizon}, we have
\begin{equation}\label{eq:ycurrenthorizon2}
K^Y[\psi]\Big|_{\mc{C}^+} \leq \Big(\frac{1}{2}\kappa_{\mc{C}}+\frac{2}{\rc^2|\sigma|}\Big)(Y\psi)^2.
\end{equation}
A short computation reveals that the energy current $J^N\cdot \pa_v$ restricted to the cosmological horizons is
$$J^N[\psi]\cdot \frac{\pa}{\pa v}\Big|_{\mc{C}^+} = \frac{u}{\iota_{\mc{C}}}\Big((Y\psi)^2+\kappa_{\mc{C}}^2|\slashed\nabla\psi|^2\Big).$$
Since $T$ is Killing everywhere, $K^N = K^Y + K^T = K^Y$, and so by \eqref{eq:ycurrenthorizon2},
$$K^N[\psi]\Big|_{\mc{C}^+} \leq \Big(\frac{1}{2}\kappa_{\mc{C}} + \frac{2}{\rc^2\sigma}\Big)\frac{\iota_\mc{C}}{u}J^N \cdot \frac{\pa}{\pa v}\Big|_{\mc{C}^+}.$$
By continuity this inequality can be extended to a neighbourhood of $\mc{C}^+$, and so we may choose some $r_0$ sufficiently close to $\rc$, and some $\sigma$ sufficiently negative such that \eqref{eq:ncurrentbound} holds on $\mc{D}_{(\rc,r_0)}$.
\end{proof}
\end{proposition}
We now introduce some notation that will be important for the following proposition.
Let $\tau$ be a function on $\mc{C}^+$ that is constant on the spheres of symmetry $S \subset \mc{C}^+$, and
$$T \cdot \tau\Big|_{\mc{C}^+} = 1,$$
and set $\tau = 0$ on the sphere for which $u = 1$. It follows from \eqref{eq:killingvectorfielddef} that \begin{equation}\label{eq:dtaudu}
\frac{\de \tau}{\de u}\Big|_{\mc{C}^+} = \frac{1}{\kappa_{\mc{C}}u}.
\end{equation}
For given $\tau_1,\tau_2 \in \RR$, let $\mc{C}_{\tau_1}^+$ be the segment of $\mc{C}^+$ where $\tau \geq \tau_1$, and let $C_{\tau_1}$ denote the null hypersurface which lies transverse to the cosmological horizons and intersects $\mc{C}^+$ at the sphere on which $\tau = \tau_1$. For $r_0 > \rc$, let $\Sigma_{r_0}^{(\tau_1)} = \Sigma_{r_0} \cap J^+(\mc{C}_{\tau_1}^+)$. Moreover, define the compact hypersurfaces
\begin{align*}
&\mc{C}_{\tau_1,\tau_2}^+ = \mc{C}_{\tau_1}^+ \bs \mc{C}_{\tau_2}^+.\\*
&C_{\tau_1}^c = C_{\tau_1} \cap J^-(\Sigma_{r_0}).\\*
&\Sigma_{r_0}^{(\tau_1,\tau_2)} = \Sigma_{r_0}^{(\tau_1)} \bs \Sigma_{r_0}^{(\tau_2)}.
\end{align*}
We note here that for any $\tau_1,\tau_2 \in \RR$ with $\tau_1<\tau_2$, and $r_0 > \rc$, the sets $\mc{C}_{\tau_1,\tau_2}^+$, $C_{\tau_1}^c$, $C_{\tau_2}^c$, and $\Sigma_{r_0}^{(\tau_1,\tau_2)}$ enclose a compact region of $\mc{R}^+$, which we denote by $\mc{R}_{\tau_1,\tau_2}^{(r_0)}$; see Figure \ref{fig:energydomainhorizon}.

We will now prove a nondegenerate energy estimate at the cosmological horizons. We refer the reader to Proposition 4.4 in \cite{Sch15} for the corresponding energy estimate on the cosmological horizons proved for the forward problem.
\begin{figure}[t]
\centering
\begin{tikzpicture}[scale=2.5]
\draw (0,3) -- (2.5,0.5);
\draw[dashed] (0,3) .. controls (2.5,2.5)  .. (3,2.5);
\draw (1.3,1.5) node[below]{$\mc{C}_{\tau_1,\tau_2}^+$};
\draw (0,3) .. controls (1.5,2) .. (3,1.5);
\draw (2.6,1.2) node[below]{$C_{\tau_1}^c$};
\draw (1.15,1.95) node[above]{$C_{\tau_2}^c$};
\filldraw[fill=gray!20] (1.41,2.11) .. controls (1.8,1.9) and (2.4,1.7) .. (2.85,1.55) -- (2.15,0.85) -- (1.15,1.85) -- cycle;
\filldraw  (2.15,0.85) circle (0.025cm);
\filldraw  (2.85,1.55) circle (0.025cm);
\filldraw  (1.15,1.85) circle (0.025cm);
\filldraw  (1.42,2.12) circle (0.025cm);
\draw (1.9,1.7) node[below]{$\mc{R}_{\tau_1,\tau_2}^{(r_0)}$};
\draw (2.1,1.9) node[above]{$\Sigma_{r_0}^c$};
\filldraw[fill=white] (0,3) circle (0.035cm);
\end{tikzpicture}
\caption{Spacetime domain $\mc{R}_{\tau_1,\tau_2}^{(r_0)}$ of the nondegenerate energy estimate along $\mc{C}^+$.}\label{fig:energydomainhorizon}
\end{figure}
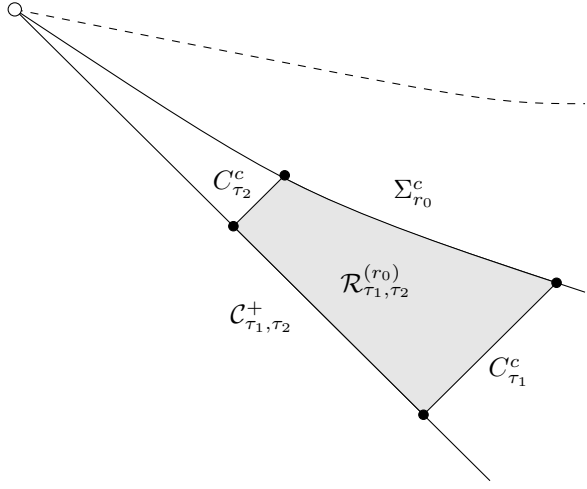
\begin{proposition}\label{prop:backwardesthorizon}
Let $\psi$ be a solution to \eqref{eq:waveequation} with finite energy on $\Sigma_{r_0}$ for some $r_0$ sufficiently close to $\rc$. Set
$$f(\tau) = \int_{\mc{C}_\tau'}\>^*J^N[\psi].$$
The following energy estimate holds for all solutions to $\square_g = 0$ on $\mc{R}^+\cup\mc{C}^+$:
\begin{align*}
f(\tau_1) +\int_{\mc{C}_{\tau_1,\tau_2}^+} \>^*J^N[\psi] \leq \>&e^{c(\tau_2 - \tau_1)}f(\tau_2) + \int_{\Sigma_{r_0}^{(\tau_1,\tau_2)}}J^N[\psi]\cdot n_{\Sigma_r}\de \mu_{\ol{g}_r}\\*
&+ c\int_{\tau_1}^{\tau_2}e^{c(\tau-\tau_1)}\Big(\int_{\Sigma_{r_0}^{(\tau,\tau_2)}}J^N[\psi]\cdot n_{\Sigma_r}\de \mu_{\ol{g}_r}\Big)\de \tau
\end{align*}
for all $\tau_1 < \tau_2$, where $c >0$ is any constant such that $c > \kappa_{\mc{C}}$.
\end{proposition}
\begin{proof}
Applying the divergence theorem to the energy current $J^N[\psi]$ on $\mc{R}_{\tau_1,\tau_2}^{(r_0)}$, we get
\begin{multline}\label{eq:divergencehorizon}
\int_{C_{\tau_2}^c} \>^*J^N[\psi] + \int_{\Sigma_{r_0}^{(\tau_1,\tau_2)}}J^N[\psi] \cdot n_{\Sigma_r} \de \mu_{\ol{g}_r} + \int_{\mc{R}_{\tau_1,\tau_2}^{(r_0)}}K^N[\psi] \>\de \mu_g\\*
=\int_{C_{\tau_1}^c} \>^*J^N[\psi] + \int_{\mc{C}_{\tau_1,\tau_2}^+} \>^*J^N[\psi].
\end{multline}
In the Kruskal coordinates defined in Section~1.2, we have
$$\de \mu_{g} = \frac{1}{2}\Omega^2 r^2 \de u \wedge \de v \wedge \de \mu_{\mr{\gamma}},$$
From this and Proposition~\ref{prop:ncurrent}, it follows that there exists some constant $c > 0$ that depends on $\kappa_{\mc{C}}$ such that
$$\int_{\mc{R}_{\tau_1,\tau_2}^{(r_0)}}K^N[\psi] \>\de \mu_g \leq c\int_{\tau_1}^{\tau_2}\Big(\int\Big( J^N\cdot \frac{\pa}{\pa v}\Big) \de v \de \mu_{\mr{\gamma}}\Big)\de \tau,$$
where we used the relation \eqref{eq:dtaudu} between $u$ and $\tau$. On the other hand, we can write the integral of $\>^*J^N$ on $C_\tau^c$ as
$$\int_{C_{\tau}^c}\>^*J^N = \int\Big(J^N\cdot \frac{\pa}{\pa v}\Big) r^2 \de v \de \mu_{\mr{\gamma}},$$
Therefore we may write
$$\int_{\mc{R}_{\tau_1,\tau_2}^{(r_0)}} K^N[\psi] \>\de \mu_{g} \leq c \int_{\tau_1}^{\tau_2} \Big( \int_{C_\tau^c}\>^*J^N[\psi]\Big)\de \tau.$$
It follows from \eqref{eq:divergencehorizon} that
\begin{equation}\label{eq:divergencehorizon2}
f(\tau_1) + \int_{\mc{C}_{\tau_1,\tau_2}^+} \>^*J^N[\psi] \leq f(\tau_2) + c\int_{\tau_1}^{\tau_2}f(\tau)\de\tau + \int_{\Sigma_{r_0}^{(\tau_1,\tau_2)}}J^N[\psi] \cdot n_{\Sigma_r} \de \mu_{\ol{g}_r}.
\end{equation}
This is a standard Gr\"onwall-type inequality, from which the result follows by integrating over $[\tau_1,\tau_2]$ the inequality
$$-\pa_\tau\Big[e^{c\tau}F(\tau,\tau_2)\Big] = e^{c\tau}\Big[f(\tau) - cF(\tau,\tau_2)\Big] \leq e^{c\tau}\Big[f(\tau_2) + \int_{\Sigma_{r_0}^{(\tau,\tau_2)}}J^N[\psi]\cdot n_{\Sigma_r}\de\mu_{\ol{g}_r}\Big],$$
and applying the inequality \eqref{eq:divergencehorizon2} again.
\end{proof}
\begin{remark}
While Proposition~\ref{prop:backwardesthorizon} was only proven on one segment of the cosmological horizons $\mc{C}^+$, one can easily derive the analogous estimate for $\ol{\mc{C}}^+$. The expanding region and horizons $\mc{R}^+\cup\mc{C}^+\cup\ol{\mc{C}}^+$ are symmetric with respect to the null coordinates $(u,v)$, and so precisely the same argument holds by applying the isometry $(u,v) \mapsto (v,u)$ to $\mc{R}^+\cup\mc{C}^+\cup\ol{\mc{C}}^+$.
\end{remark}

\subsection{Existence of scattering solutions to the wave equation up to the cosmological horizons: proof of Theorem~\ref{thm:scatteringhorizon}}\label{sec:scatteringhorizon} 
We first prove that for the scattering solutions constructed in Theorem~\ref{thm:scatteringwithouthorizon}, exponential decay of the scattering data $\psi_0,\psi_3$ is propagated to exponential decay of the solution $\psi$ along each $\Sigma_r$ hypersurface. Let $\Sigma_r^\tau$ denote the subset of $\Sigma_r$ that lies in the future domain of dependence of $\mc{C}_\tau^+$, and let $\ol{\Sigma}_r^\tau$ denote the analagous subset of $\Sigma_r$ that lies in the future domain of dependence of $\ol{\mc{C}}_\tau^+$. These should be thought of as the ``ends'' of the $\Sigma_r$, which we recall are diffeomorphic to the standard cylinder $\RR\times\sph^2$.
\begin{proposition}\label{prop:localexpdecay}
Let $\psi_0 \in H^4(\RR\times\sph^2)$ ,$\psi_3\in H^2(\RR\times\sph^2)$ be scattering data, and let $\psi$ be the corresponding finite-energy solution on the expanding region $\mc{R}^+$. In addition, assume that $\psi_0$, $\psi_3$ decay exponentially, in the sense that as $\tau \ra \infty$ we have
$$\|\tilde{\nabla} \psi_0\|_{H^3((-\tau,\tau)^c\times \sph^2)},\|\psi_3\|_{H^2((-\tau,\tau)^c\times \sph^2)} \lesa_{\Lambda,m} e^{-\beta \tau}$$
for some constant $\beta > 0$. Then for each $r > \rc$, the energy flux of $\psi$ through $\Sigma_r$ decays exponentially toward $\iota$ and $\ol{\iota}$, so that
\begin{equation}\label{eq:localexpdecayenergy} \|\psi\|_{M,\Sigma_r^\tau},\|\psi\|_{M,\ol{\Sigma}_r^\tau} \lesa_{\Lambda,m,r} e^{-\beta \tau}.
\end{equation}
The $L^2$-norm of $\psi$ also decays exponentially,
\begin{equation}\label{eq:localexpdecayL2}
\|\psi\|_{L^2(\Sigma_r^\tau)},\|\psi\|_{L^2(\ol{\Sigma}_r^\tau)} \lesa_{\Lambda,m,r} e^{-\beta \tau}
\end{equation}
\end{proposition}
\begin{proof}
This follows directly from the localised version of the weighted energy estimate from Proposition~\ref{prop:backwardenergyestimate}, namely \eqref{eq:backwardestimatelocal}. We will only prove the estimate for the energy flux of $\psi$ through $\Sigma_r^\tau$, the proof for $\ol{\Sigma}_r^\tau$ is the same. We estimate
\begin{align}
\|\psi\|_{M,\Sigma_r^\tau} &\leq \|\psi_{\text{asymp}}\|_{M,\Sigma_r^\tau} + \|\psi_{\text{rem}}\|_{M,\Sigma_r^\tau}\nonumber\\*
&\leq \|\psi_{\text{asymp}}\|_{M,\Sigma_r^\tau} +
\|\psi_{\text{rem}} - \psi_{\text{rem}}^R\|_{M,\Sigma_r^\tau}+\|\psi_{\text{rem}}^R\|_{M,\Sigma_r^\tau}\label{eq:localexpdecay1}
\end{align}
As a consequence of the localised energy estimate \eqref{eq:backwardestimatelocal}, we can bound the last term like
$$\|\psi_{\text{rem}}^R\|_{M,\Sigma_r^\tau} \leq \int_r^R \|r\phi^{1/2}\square_g \psi_{\text{asymp}}\|_{L^2(\Sigma_r^\tau)}\de r.$$
This follows from taking the limit $\ol{v} \ra 0$, while still keeping $\ol{u}$ fixed in \eqref{eq:backwardestimatelocal}. From Lemma~\ref{lem:asympsolboundl2} we bound the energy and $L^2$-norms of the asymptotic solution
\begin{gather*}\|\psi_{\text{asymp}}\|_{M,\Sigma_r^\tau} \lesa_{\Lambda,m,r} \|\tilde{\nabla} \psi_0\|_{H^2((\tau,\infty)\times\sph^2)} + \|\psi_3\|_{H^1((\tau,\infty)\times\sph^2)} ,\\*
\int_r^{\infty}\|r\phi^{1/2}\square_g \psi_{\text{asymp}}\|_{L^2(\Sigma_r^\tau)} \>\de r \lesa_{\Lambda,m,r} \|\tilde{\nabla}\psi_0\|_{H^3((\tau,\infty)\times\sph^2)}+ \frac{1}{r^2}\|\psi_3\|_{H^2((\tau,\infty)\times\sph^2)}.
\end{gather*}
It follows from \eqref{eq:localexpdecay1} and the assumption of exponential decay of $\psi_0,\psi_3$ that
\begin{align*}\|\psi\|_{M,\Sigma_r^\tau} &\lesa_{\Lambda,m,r} \|\tilde{\nabla}\psi_0\|_{H^3((\tau,\infty)\times\sph^2)}+ \frac{1}{r^2}\|\psi_3\|_{H^2((\tau,\infty)\times\sph^2)} + \|\psi_{\text{rem}} - \psi_{\text{rem}}^R\|_{M,\Sigma_r^\tau} \\*
&\lesa e^{-\beta\tau} + \|\psi_{\text{rem}} - \psi_{\text{rem}}^R\|_{M,\Sigma_r^\tau}.
\end{align*}
The estimate \eqref{eq:localexpdecayenergy} now follows after taking the limit $R \ra \infty$. The proof of \eqref{eq:localexpdecayL2} follows analogously through a localised version of the $L^2$ estimate \eqref{L2Bound} from Proposition~\ref{lem:l2est}.
\end{proof}
\begin{remark}
Proposition~\ref{prop:localexpdecay} holds for any exponential decay rate, i.e. one can choose any $\beta > 0$. However for the following proof of Theorem~\ref{thm:scatteringhorizon}, one must assume a sufficiently high degree of exponential decay, see in particular \eqref{eq:betarestriction} below. This is a manifestation of the blueshift effect discussed in Section~\ref{sec:scatteringintro}. One can show that Theorem~\ref{thm:scatteringhorizon} follows as long as $\beta > \kappa_{\mc{C}}/2$. This is accomplished by carefully tracking constants throughout the proof and choosing $r_0$ sufficiently close to $\rc$.
\end{remark}

\begin{proof}[Proof of Theorem~\ref{thm:scatteringhorizon}]
Let $T > 0$ be some constant, consider the hypersurfaces $C_T^c,\ol{C_T^c}$, where $C_T^c$ is the outgoing null hypersurface that intersects $\mc{C}^+$ at the point $\tau = T$, and $\ol{C_T^c}$ is the corresponding hypersurface that intersects $\ol{\mc{C}}^+$. Let $\psi'$ be the solution on $\mc{R}^+$ with scattering data $\psi_0,\psi_3$. By Theorem~\ref{thm:scatteringwithouthorizon}, $\psi'$ has finite energy on $\Sigma_{r_0}$ for all $r_0 > \rc$. Let $T' > 0$ denote the value of the coordinate $t$ on the sphere of intersection of the hypersurfaces $\Sigma_{r_0}$ and $C_T^c$. We will consider a solution $\psi_T$ that satisfies the homogeneous wave equation $\square_g \psi^T = 0$, with initial data on $\Sigma_{r_0}$ given as
\begin{equation*}
\psi_T|_{\Sigma_{r_0}} = \chi(t/T')\psi'|_{\Sigma_{r_0}},\quad n_{\Sigma_r}\cdot \psi_T|_{\Sigma_{r_0}} = \chi(t/T')n_{\Sigma_r}\cdot\psi'|_{\Sigma_{r_0}},
\end{equation*}
where $\chi \in C_0^\infty(\RR)$ is a cutoff function such that $0 \leq \chi \leq 1$, $\chi(s) = 1$ for $|s| \leq 1/2$, and $\chi(s) = 0$ for $|s| \geq 1$. By the localised weighted energy estimate \eqref{eq:backwardestimatelocal} from Proposition~\ref{prop:backwardenergyestimate}, it follows that $\psi = 0$ on the entire future domain of dependence of $\mc{C}_T^{+}\cup \ol{\mc{C}_T^{+}}$ up to $\Sigma_{r_0}$. In particular, we have
$$\psi_T|_{C_T^c} = \psi_T|_{\ol{C_T^c}} = 0, \quad \psi_T|_{\mc{C}_T^+} = \psi_T|_{\ol{\mc{C}}_T^+} = 0.$$
Now consider the difference $v = \psi_{T_2} - \psi_{T_1}$, for $T_2 > T_1$. Clearly $v$ satisfies the wave equation, and on $\Sigma_{r_0}$, we have
\begin{gather}
v|_{\Sigma_{r_0}} = (\chi(t/T_2')-\chi(t/T_1'))\psi'|_{\Sigma_{r_0}}\nonumber\\*
n_{\Sigma_r}\cdot v|_{\Sigma_{r_0}} = (\chi(t/T_2')-\chi(t/T_1'))n_{\Sigma_r} \cdot\psi'|_{\Sigma_{r_0}}.\label{eq:initialdatav}
\end{gather}
Note that the initial data $v|_{\Sigma_r}$ and $n_{\Sigma_r}\cdot v|_{\Sigma_r}$ are supported only on the compact set
$$\Sigma_{r_0}^{(T_1/2,T_2)}\cup \ol{\Sigma}_{r_0}^{(T_1/2,T_2)},$$
i.e. $v = 0$, $n_{\Sigma_r}\cdot v = 0$ on $\Sigma_{r_0}$ for all $|t| \leq T_1'/2$ and for all $|t| \geq T_2'$. We apply the nondegenerate energy estimate of Proposition~\ref{prop:backwardesthorizon} to $v$ on the domain $\mc{R}_{\tau,T_1}^{(r_0)}$ near $\mc{C}^+$ so that for $\tau < T_1$, $v$ satisfies the estimate
\begin{multline}\label{eq:backwardhorizon1}
\int_{C_\tau^c} \>^*J^N[v] + \int_{\mc{C}_{\tau,T_1}^+} \>^*J^N[v] \leq e^{c(T_1-\tau)}\int_{C_{T_1}^c} \>^*J^N[v] + \int_{\Sigma_{r_0}^{(\tau,\tau_1)}}J^N[v]\cdot n_{\Sigma_r}\de \mu_{\ol{g}_r}\\*
+c\int_{\tau}^{T_1}e^{c(s-\tau)}\Big(\int_{\Sigma_{r_0}^{(s,T_1)}}J^N[v]\cdot n_{\Sigma_r}\de \mu_{\ol{g}_r}\Big)\de s.
\end{multline}
Since $v = \psi_{T_2}$ on $C_{T_1}^c$, we have
$$\int_{C_{T_1}^c}\>^*J^N[v] =\int_{C_{T_1}^c}\>^*J^N[\psi_{T_2}].$$
We then apply Proposition~\ref{prop:backwardesthorizon} again, this time to $\psi_{T_2}$ on the domain $\mc{R}_{T_1,T_2}^{(r_0)}$ so that
\begin{multline}\label{eq:backwardhorizon2}
\int_{C_{T_1}^c}\>^*J^N[\psi_{T_2}]+\int_{\mc{C}_{T_1,T_2}^+}\>^*J^N[\psi_{T_2}] \leq  \int_{\Sigma_{r_0}^{(T_1,T_2)}}J^N[\psi_{T_2}]\cdot n_{\Sigma_r}\de \mu_{\ol{g}_r}\\*
+c\int_{T_1}^{T_2}e^{c(s-T_1)}\Big(\int_{\Sigma_{r_0}^{(s,T_2)}}J^N[\psi_{T_2}]\cdot n_{\Sigma_r}\de \mu_{\ol{g}_r}\Big)\de s,
\end{multline}
where we used that $\psi_{T_2} = 0$ on $C_{T_2}^c$, and so the corresponding energy flux through $C_{T_2}^c$ vanishes. Also, note that the energy flux of $\psi_{T_1} = 0$ through $\mc{C}_{T_1}^+$ vanishes, and therefore
\begin{equation}\label{eq:backwardhorizon2.5}
\int_{\mc{C}_{T_1,T_2}^+}\>^*J^N[\psi_{T_2}] = \int_{\mc{C}_{T_1,T_2}^+}\>^*J^N[v]=\int_{\mc{C}_{T_1}^+}\>^*J^N[v].
\end{equation}
Because $M$ and $N$ are both strictly timelike on $\Sigma_{r_0}$, it follows that $J^N\cdot n \sim J^M\cdot n$ on $\Sigma_{r_0}$, and so
$$\int_{\Sigma_{r_0}^{(T_1,T_2)}} J^N\cdot n_{\Sigma_r}\de\mu_{\ol{g}_r} \sim_{\Lambda,m,r_0}\int_{\Sigma_{r_0}^{(T_1,T_2)}} J^{M}\cdot n_{\Sigma_r}\de\mu_{\ol{g}_r}.$$
Combining this with \eqref{eq:backwardhorizon1}, \eqref{eq:backwardhorizon2} and \eqref{eq:backwardhorizon2.5}, we get
\begin{multline}
\int_{C_\tau^c} \>^*J^N[v] + \int_{\mc{C}_{\tau}^+} \>^*J^N[v] \leq \|v\|_{M,\Sigma_{r_0}^{(\tau,T_1)}}^2 + c\int_{\tau}^{T_1}e^{c(s-\tau)}\|v\|_{M,\Sigma_{r_0}^{(s,T_1)}}^2\de s\\*
+e^{c(T_1-\tau)}\Big(\|\psi_{T_2}\|_{M,\Sigma_{r_0}^{(T_1,T_2)}}^2 + c\int_{T_1}^{T_2}e^{c(s-T_1)}\|\psi_{T_2}\|_{M,\Sigma_{r_0}^{(s,T_2)}}^2\de s\Big).\label{eq:backwardhorizon3}
\end{multline}
We will appeal to Proposition~\ref{prop:localexpdecay} to show that the right hand side of \eqref{eq:backwardhorizon3} vanishes as $T_1,T_2 \ra \infty$. Since $v=0$ and $n_{\Sigma_{r_0}}\cdot v = 0$ on $\Sigma_{r_0}$ for all $t \leq T_1'/2$, along with the fact that $\Sigma_{r_0}^{(\tau,T_1)} \subset \Sigma_{r_0}^\tau$, the energy flux of $v$ through $\Sigma_{r_0}^{(\tau,T_1)}$ can be estimated like
$$\|v\|_{M,\Sigma_{r_0}^{(\tau,T_1)}}^2 \lesa \begin{cases}\|\psi\|_{M,\Sigma_{r_0}^{T_1/2}}^2 + \|\psi\|_{L^2(\Sigma_{r_0}^{T_1/2})}^2,\quad &\tau \leq T_1/2.\\*
\|\psi\|_{M,\Sigma_{r_0}^{\tau}}^2 + \|\psi\|_{L^2(\Sigma_{r_0}^{\tau})}^2, \quad &T_1/2 \leq \tau \leq T_1.
\end{cases}$$
Similarly we can bound
$$\|\psi_{T_2}\|_{M,\Sigma_{r_0}^{(\tau,T_2)}}^2 \lesa \|\psi\|_{M,\Sigma_{r_0}^\tau}^2 + \|\psi\|_{L^2(\Sigma_{r_0}^\tau)}^2.$$
We now use the result from Proposition~\ref{prop:localexpdecay}, which implies the exponential decay of the energy and $L^2$-norm of $\psi$ along $\Sigma_{r_0}$. This implies the bounds
$$\|v\|_{M,\Sigma_{r_0}^{(\tau,T_1)}}^2 \lesa \begin{cases}e^{-\beta T_1},\quad&\tau \leq T_1/2.\\*
e^{-2\beta \tau},\quad&T_1/2 \leq \tau \leq T_1.
\end{cases}.$$
$$\|\psi_{T_2}\|_{M,\Sigma_{r_0}^{(\tau,T_2)}}^2 \lesa e^{-2\beta \tau}.$$
By these estimates, it follows from \eqref{eq:backwardhorizon3} that
\begin{multline}
\int_{C_\tau^c} \>^*J^N[v] +\int_{\mc{C}_{\tau}^+} \>^*J^N[v] \lesa e^{-\beta T_1} + c\int_\tau^{T_1/2} e^{c(s-\tau)}e^{-\beta T_1} \de s + c\int_{T_1/2}^{T_1}e^{c(s-\tau)}e^{-2\beta s}\de s\\*
+ e^{c(T_1-\tau)}\Big( e^{-2\beta T_1} + c\int_{T_1}^{T_2}e^{c(s-T_1)}e^{-2\beta s} \de s\Big).\label{eq:backwardhorizon4}
\end{multline}
If we assume $\beta > c/2$, then we have
\begin{equation}\label{eq:betarestriction}
\int_{T_1/2}^{T_1} e^{c(s-\tau)}e^{-2\beta s}\de s = -\frac{e^{-c\tau}}{2\beta -c}\Big(e^{-(2\beta-c)T_1}-e^{-(2\beta-c)T_1/2}\Big) \leq \frac{e^{-c\tau}}{2\beta-c}e^{-\beta T_1/2}.
\end{equation}
Bounding the other terms similarly we find that
$$\int_{C_\tau^c}  \>^*J^N[v]+\int_{\mc{C}_{\tau}^+} \>^*J^N[v]\lesa (1+e^{-c\tau})e^{-\beta T_1/2}.$$
This holds for all $T_2 > T_1$,
Repeating the same argument near $\ol{\mc{C}}^+$, it follows that
$$\int_{C_\tau^c}  \>^*J^N[v]+\int_{\mc{C}_{\tau}^+} \>^*J^N[v] + \int_{\ol{C_\tau^c}} \>^*J^N[v] +\int_{\ol{\mc{C}}_{\tau}^+}\>^*J^N[v]\lesa (1+e^{-c\tau})e^{-\beta T_1/2}.$$
which vanishes as $T_1,T_2 \ra \infty$. Hence there exists a function $\psi$ on $\mc{R}^+\cup\mc{C}^+\cup\ol{\mc{C}}^+$ such that for each $\tau\in(-\infty,\infty)$, we have
\begin{gather*}\lim_{T \ra \infty}\|\psi - \psi_T\|_{N,C_\tau^c} = \lim_{T \ra \infty}\|\psi - \psi_T\|_{N,\ol{C_\tau^c}} = 0,\\
\lim_{T \ra \infty}\|\psi - \psi_T\|_{N,\mc{C}_{\tau}^+} = \lim_{T \ra \infty}\|\psi - \psi_T\|_{N,\ol{\mc{C}}_{\tau}^+} =0.
\end{gather*}
This gives uniform boundedness of the energy of $\psi$ outside a small neighbourhood of the bifurcation sphere $\mc{C}^+\cap \ol{\mc{C}}^+$; this is simply because the vectorfield $N$ blows up there. With any regular, strictly timelike vectorfield $N'$, one can show $\psi$ remains bounded at the bifurcation sphere through a standard Gr\"onwall-type inequality on a compact set. Without changing notation, we modify $N$ to be regular and timelike in a neighbourhood of $\mc{C}^+\cap\ol{\mc{C}}^+$, and set
$$E^N[\psi] = \int_{\mc{C}^+\cap\ol{\mc{C}}^+} \>^*J^N[\psi].$$
Then we have $E^N[\psi] \lesa \|\psi\|_{\pa_r,\Sigma_{r_0}}^2$, hence $\psi$ is uniformly bounded on $\mc{C}^+\cup\ol{\mc{C}}^+$. Moreover we obtain \eqref{eq:scatteringhorizonbound}, as by \eqref{eq:scatteringenergybound} we have $ \|\psi\|_{\pa_r,\Sigma_{r_0}}^2 \lesa \|\psi_0\|_{H^4(\RR\times\sph^2)}^2 + \|\psi_3\|_{H^2(\RR\times\sph^2)}^2$. This proves Theorem~\ref{thm:scatteringhorizon}.
\end{proof}


\section{Asymptotics of forward solutions to the wave equation on Schwarzschild-de Sitter}\label{sec:forwardasymptotics}
In this section we prove the main asymptotics result \textbf{Theorem~\ref{thm:forwardasymptotics}} for solutions to the linear wave equation on Schwarzschild-de Sitter. We emphasise the absence of an $O(r^{-1})$ term in this expansion, and additionally the absence of log terms. In particular we prove the existence of the limits \eqref{eq:asymptoticlimits} in Section~\ref{sec:forwardlimits}.

We establish these asymptotics by employing several higher-order energy estimates, which we also prove in this section. These estimates are obtained by commuting the wave equation with suitably chosen vectorfield commutators.  These commutators are not the Killing vectorfields $T,\Omega_i$ that commute with the wave operator $\square_g$, but are instead proportional to the radial coordinate vectorfield $\frac{\pa}{\pa r}$. The main energy estimate, found in\textbf{ Theorem~\ref{thm:forwardenergyintro}}, is proved in Section~\ref{sec:secondenergy}.

\subsection{Higher order energy estimates with vectorfield commutators}
In order to prove the subsequent energy estimates throughout this section, we will be returning to a common argument for generating energy estimates on $\mc{R}^+$. We apply the divergence theorem to the energy current $J^{\pa_r}[X\psi]$ on the spacetime region bound in the past by $\Sigma_{r_1}$, and to the future by $\Sigma_{r_2}$, where $r_2 > r_1 \gg \rc$, which yields the inequality
\begin{equation}\label{eq:divergencecoarea}
\int_{\Sigma_{r_2}} J^{\pa_r}[X\psi] \cdot n_{\Sigma_r} \de \mu_{\ol{g}_r} + \int_{r_1}^{r_2} \Big(\int_{\Sigma_r}\phi\nabla \cdot J^{\pa_r}[X\psi] \de \mu_{\ol{g}_r}\Big)\de r \leq \int_{\Sigma_{r_1}} J^{\pa_r}[X\psi] \cdot n_{\Sigma_r} \de \mu_{\ol{g}_r},
\end{equation}
where we applied the coarea formula to the bulk term. The strategy for proving the higher-order energy estimates in this section is to identify vectorfield commutators $X$ for which the divergence of $J^{\pa_r}[X\psi]$ possesses key sign and decay properties. Recall that by the product rule we have
\begin{equation}\label{eq:divergenceleibniz}
\nabla \cdot J^{\pa_r}[X\psi] = K^{\pa_r}[X\psi] + \pa_r(X\psi)\square_g(X\psi)
\end{equation}
It was computed in \cite{Sch15} that
\begin{equation}\label{eq:deftensorforward}
K^{\pa_r}[X\psi] = \Big(\frac{2\Lambda}{3}r - \frac{1}{r}+\frac{m}{r}\Big)(\pa_r(X\psi))^2 + \Big(\LT r - \frac{m}{r}\Big)\phi^4 (\pa_t(X\psi))^2.
\end{equation}
One can see from this that (to leading order in $r$), the energy current $K^{\pa_r}$ is nonnegative. Indeed this nonnegativity is what implies the global redshift estimate of \cite{Sch15}. For higher-order estimates, however, one must consider contributions from the second term $\pa_r(X\psi)\square_g(X\psi)$. We will restrict our attention to commutation vectorfields of the form $X = f(r)\pa_r$. We begin by deriving an identity for $\square_g (X\psi)$.
\begin{lemma}\label{lem:commutedwaveeq}
Let $f \in C^2(\rc,\infty)$ with $f(r) > 0$ for all $r > \rc$, and let $X = f(r)\pa_r$. Then
\begin{multline}\label{eq:commutedwaveeq}
\square_g (X\psi) = X(\square_g \psi) + \frac{f^2}{r^2}\pa_r\Big(\frac{r^2\phi^{-2}}{f^2}\Big)\pa_r(X\psi) + \frac{f^2}{r^2}\pa_r^2\Big(\frac{r^2\phi^{-2}}{f}\Big)\pa_r \psi\\*
+2f\phi^4\Big(\frac{1}{r}-\frac{3m}{r}\Big)\pa_t^2 \psi + \frac{2f}{r}\square_g \psi.
\end{multline}
\end{lemma}
\begin{proof}
By \eqref{eq:waveeqcoords}, we can write
\begin{equation*}
[\square_g, X] \psi =-\frac{1}{r^2}\pa_r\Big(r^2\phi^{-2}\pa_r(f\pa_r\psi)\Big) +f\pa_r\Big(\frac{1}{r^2}\pa_r(r^2\phi^{-2}\pa_r\psi)\Big) +[\ol{\Delta},X]\psi.
\end{equation*}
Computing the final term, we have
\begin{align*}
[\ol{\Delta},X]\psi &=\phi^2\pa_t^2(f\pa_r\psi) - f\pa_r(\phi^2\pa_t^2\psi) + \frac{1}{r^2}\Delta_{\sph^2}(f\pa_r \psi) - f\pa_r\Big(\frac{1}{r^2}\Delta_{\sph^2}\psi\Big)\\*
&= -f\pa_r(\phi^2)\pa_t^2 \psi - f\pa_r \Big(\frac{1}{r^2}\Big)\Delta_{\sph^2}\psi\\*
&= \frac{2f}{r}\ol\Delta \psi + 2\Big(\frac{1}{r} - \frac{3m}{r}\Big)\phi^4\pa_t^2 \psi,
\end{align*}
where we used in the last line the fact that
\begin{align*}
\pa_r(\phi^2) &= -\Big(\frac{2\Lambda}{3}r - \frac{2m}{r^2}\Big)\phi^4\\*
&= -\frac{2}{r}\phi^2 - 2\Big(\frac{1}{r}-\frac{3m}{r^3}\Big)\phi^4.
\end{align*}
Substituting \eqref{eq:waveeqcoords}, it follows that
\begin{multline*}
[\square_g ,X]\psi = -\frac{1}{r^2}\pa_r\Big(r^2\phi^{-2}\pa_r(f\pa_r\psi)\Big) +f\pa_r\Big(\frac{1}{r^2}\pa_r(r^2\phi^{-2}\pa_r\psi)\Big)\\*
+\frac{2f}{r^3}\pa_r\Big(r^2\phi^{-2}\pa_r\psi\Big)+ \frac{2f}{r}\square_g \psi +2\Big(\frac{1}{r} - \frac{3m}{r}\Big)\phi^4\pa_t^2 \psi
\end{multline*}
As for the radial derivatives, we have
\begin{multline*}
-\frac{1}{r^2}\pa_r\Big(r^2\phi^{-2}\pa_r(f\pa_r\psi)\Big) +f\pa_r\Big(\frac{1}{r^2}\pa_r(r^2\phi^{-2}\pa_r\psi)\Big) + \frac{2f}{r^3}\pa_r\Big(r^2\phi^{-2}\pa_r\psi\Big)\\*
=\Big[\frac{2\phi^{-2}}{r}f + \pa_r(\phi^{-2})f - 2\phi^{-2}\pa_r f\Big]\pa_r^2\psi\\*
+\Big[\pa_r^2(\phi^{-2})f+\frac{4\pa_r(\phi^{-2})}{r}f+\frac{2\phi^{-2}}{r^2}f\\*
\quad\quad\quad\quad-\pa_r(\phi^{-2})\pa_r f - \frac{2\phi^{-2}}{r}\pa_r f - \phi^{-2}\pa_r^2f\Big]\pa_r \psi.
\end{multline*}
By using that
$$\pa_r^2 \psi = \frac{1}{f}\pa_r(X\psi) - \frac{\pa_r f}{f}\pa_r \psi,$$
it follows that
\begin{multline*}
\square_g (f\pa_r\psi) =f\pa_r(\square_g \psi) + \frac{2f}{r}\square_g \psi +2\Big(\frac{1}{r} - \frac{3m}{r}\Big)\phi^4\pa_t^2 \psi\\*
+\Big[\frac{2\phi^{-2}}{r} + \pa_r(\phi^{-2}) - 2\phi^{-2}\frac{\pa_r f}{f}\Big]\pa_r(f\pa_r\psi)\\*
+\Big[\pa_r^2(\phi^{-2})f+\frac{2\pa_r(\phi^{-2})}{r}f+\frac{2\phi^{-2}}{r^2}f\\*
\quad\quad\quad-2\pa_r(\phi^{-2})\pa_r f- \frac{2\phi^{-2}}{r}\pa_r f - 2\phi^{-2}\frac{(\pa_rf)^2}{f}-\phi^{-2}\pa_r^2f\Big]\pa_r \psi.\\*
\end{multline*}
A brief computation shows that the coefficients of the radial derivatives are given by
\begin{equation*}
\frac{f^2}{r^2}\pa_r\Big(\frac{r^2\phi^{-2}}{f^2}\Big) = \Big[\frac{2\phi^{-2}}{r} + \pa_r(\phi^{-2}) - 2\phi^{-2}\frac{\pa_r f}{f}\Big].
\end{equation*}
and
\begin{multline*}
\frac{f^2}{r^2}\pa_r^2\Big(\frac{r^2\phi^{-2}}{f}\Big) = \Big[\pa_r^2(\phi^{-2})f+\frac{2\pa_r(\phi^{-2})}{r}f+\frac{2\phi^{-2}}{r^2}f-2\pa_r(\phi^{-2})\pa_r f\\*
- \frac{2\phi^{-2}}{r}\pa_r f - 2\phi^{-2}\frac{(\pa_rf)^2}{f}-\phi^{-2}\pa_r^2f\Big].
\end{multline*}
i.e. they match those in \eqref{eq:waveeqcoords}.
\end{proof}

We now focus on the $\pa_r \psi$ term in \eqref{eq:commutedwaveeq}. In the energy estimates that follow, we will see that it is crucial that this term vanishes at much higher order. There is a family of nonzero functions $f(r)$ for which the coefficient of the $\pa_r \psi$ term vanishes completely, given by
$$f(r) = \frac{r^2\phi^{-2}}{c_1 + c_2 r},$$
for parameters $(c_1,c_2) \in \RR^2\bs \{(0,0)\}$. Consequently we define the vector fields
$$X_s = r\phi^{-2}, \quad X_w = r^2\phi^{-2},$$
which correspond with choosing $(c_1,c_2) = (0,1)$ and $(c_1,c_2) = (1,0)$ respectively. Throughout this section, we will frequently be considering only the leading-order behaviour of terms that appear in the energy estimates, and so we remind the reader that $\phi \sim r^{-1}$ for large $r$, and therefore $X_s \sim r^3\pa_r$, $X_w \sim r^4\pa_r$. The subscripts of $X_s,X_w$ refer to the fact that $X_s$ will be used to prove a sharp energy estimate, while $X_w$ will be used to prove an auxillary estimate that yields weaker decay of certain second derivatives present in both energy fluxes, but is nevertheless necessary to prove the sharp estimate.

\subsection{First-order weighted energy estimates}
We now state a weighted energy estimate that captures the first-order behaviour of forward solutions to the wave equation.
\begin{theorem}\label{thm:hoestimate}
Fix $r_0 > \rc$, and let $\psi$ be a solution to \eqref{eq:waveequation} on $\mc{R}^+$. For $r_2 > r_1 \geq r_0$, the following higher-order estimate of $\psi$ holds:
\begin{equation}\label{eq:hoestimate}
\int_{\Sigma_{r_2}} J^{\pa_r}[X_s\psi] \cdot n_{\Sigma_r}\de \mu_{\ol{g}_r}\lesa_{\Lambda,m,r_0} \int_{\Sigma_{r_1}} J^{\pa_r}[X_s\psi] \cdot n_{\Sigma_r} \de \mu_{\ol{g}_r} +\int_{\Sigma_{r_1}} J^{\pa_r}[\psi] \cdot n_{\Sigma_r} \de \mu_{\ol{g}_r}.
\end{equation}
\end{theorem}
\begin{proof}
We will denote the energy flux of $\psi$ through the hypersurface $\Sigma_r$ by
$$\int_{\Sigma_{r}} J^{\pa_r}[\psi] \cdot n_{\Sigma_r} \de \mu_{\ol{g}_r} = E[\psi](r),$$
and so we rewrite \eqref{eq:divergencecoarea} as
\begin{equation}\label{eq:divergencecoarea2}
E[X_s \psi](r_2) + \int_{r_1}^{r_2}\Big(\int_{\Sigma_r}\phi\nabla \cdot J^{\pa_r}[X_s\psi] \de \mu_{\ol{g}_r} \Big)\de r\leq E[X_s \psi](r_1).
\end{equation}
for all $r_2 > r_1 > \rc$. We focus on the bulk term. With the choice $f(r) = r\phi^{-2}$, we compute some of the the terms that appear in the commutation identity \eqref{eq:commutedwaveeq}:
\begin{gather*}
\frac{f^2}{r^2}\pa_r\Big(\frac{r^2\phi^{-2}}{f^2}\Big) = -\Big(\frac{2\Lambda}{3}r - \frac{2m}{r^2}\Big).\\*
\frac{f^2}{r^2}\pa_r^2\Big(\frac{r^2\phi^{-2}}{f}\Big) = 0.\\*
2f\phi^4\Big(\frac{1}{r}-\frac{3m}{r^2}\Big) = 2\phi^2\Big(1-\frac{3m}{r}\Big).
\end{gather*}
It follows from Lemma~\ref{lem:commutedwaveeq} that
\begin{equation}\label{eq:commutedeqstrong}
\phi \>\pa_r (X_s \psi)\square_g (X_s\psi) = -\phi\Big(\frac{2\Lambda}{3}r - \frac{2m}{r^2}\Big)(\pa_r(X_s\psi))^2 + 2\phi^3\Big(1-\frac{3m}{r}\Big)\pa_t^2\psi \pa_r (X_s\psi).
\end{equation}
Moreover, from \eqref{eq:deftensorforward} we have
\begin{equation}\label{eq:commutedenergy}
 \phi \>K^{\pa_r}[X_s\psi] = \phi\Big(\frac{2\Lambda}{3}r - \frac{1}{r}+\frac{m}{r}\Big)(\pa_r(X_s\psi))^2 + \phi^5\Big(\LT r - \frac{m}{r}\Big) (\pa_t(X\psi))^2.
\end{equation}
We point out that to leading order in $r$, the $(\pa_r(X_s\psi))^2$ term on the right hand side of \eqref{eq:commutedeqstrong} and \eqref{eq:commutedenergy} cancel, and there is no term that depends on $\pa_r \psi$. Summing these yields $\phi \nabla \cdot J^{\pa_r}[\psi]$, so that
\begin{multline}\label{eq:commutedenergy2}
\phi \>\nabla \cdot J^{\pa_r}[X_s\psi] = -\phi\Big(\frac{1}{r}-\frac{3m}{r^2}\Big)(\pa_r(X_s\psi))^2 \\*
+ \phi^5\Big(\frac{1}{r}-\frac{3m}{r^2}\Big)(\pa_t(X_s\psi))^2+2\phi^3\Big(1-\frac{3m}{r}\Big)\pa_t^2\psi\pa_r(X_s\psi).
\end{multline}
We now wish to bound each term on the right hand side of \eqref{eq:commutedenergy2} by the energies $J^{\pa_r}[\psi]\cdot n$ and $J^{\pa_r}[X_s \psi]\cdot n$. For reference, the energy flux density through $\Sigma_r$ with respect to the vector field $\pa/\pa_r$ of an arbitrary function $\varphi$ is
\begin{equation*}
J^{\pa_r}[\varphi]\cdot n = \frac{1}{2}\Big(\phi^{-1}(\pa_r \varphi)^2 + \phi^3(\pa_t \varphi)^2 + \phi|\slashed\nabla \varphi|^2\Big).
\end{equation*}
Therefore we may bound simply
\begin{gather*}
\Big|\phi\Big(\frac{1}{r}-\frac{3m}{r^2}\Big)(\pa_r(X_s\psi))^2\Big| \leq O\big(\frac{1}{r^3}\big)J^{\pa_r}[X_s\psi]\cdot n.\\*
\Big|\phi^5\Big(\frac{1}{r}-\frac{3m}{r^2}\Big)(\pa_t(X_s\psi))^2\Big| \leq O\big(\frac{1}{r^3}\big)J^{\pa_r}[X_s\psi]\cdot n.
\end{gather*}
The final term in \eqref{eq:commutedenergy2} is more difficult to bound, as terms like $\pa_t^2\psi$ are not controlled pointwise by the energy densities $J^{\pa_r}[\psi]$, $J^{\pa_r}[X_s\psi]$. An application of Cauchy-Schwartz implies
$$\Big|2\phi^3\Big(1-\frac{3m}{r}\Big)\pa_t^2\psi\>\pa_r(X_s\psi)\Big| \leq O\Big(\frac{1}{r^2}\Big)J^{\pa_r}[X_s\psi]\cdot n + O\Big(\frac{1}{r^5}\Big)(\pa_t^2\psi)^2.$$
It follows that
$$\Big|\phi \nabla\cdot J^{\pa_r}[X_s\psi]\Big| \leq O\big(\frac{1}{r^2}\big)J^{\pa_r}[X_s\psi]\cdot n_{\Sigma_r} + O\big(\frac{1}{r^5}\big)(\pa_t^2\psi)^2,$$
and so by \eqref{eq:divergencecoarea2} we have
\begin{multline}\label{eq:hoestgronwall}
E[X_s\psi](r_2) \leq E[X_s\psi](r_1) + \int_{r_1}^{r_2} O\big(\frac{1}{r^2}\big)E[X_s\psi](r)\de r \\*+\int_{r_1}^{r_2} \Big(O\big(\frac{1}{r^5}\big)\int_{\Sigma_r}(\pa_t^2\psi)^2\de \mu_{\ol{g}_r}\Big)\de r.
\end{multline}
To control the integral of $(\pa_t^2\psi)^2$, we employ the elliptic estimate of Corollary \ref{cor:ellipticestSdS} from Appendix~\ref{sec:elliptic}, which implies
\begin{equation}\label{eq:ellipticenergyprf}
O\big(\frac{1}{r^5}\big)\int_{\Sigma_r}(\pa_t^2\psi)^2\de \mu_{\ol{g}_r} \leq O\big(\frac{1}{r}\big)\int_{\Sigma_r}(\ol{\Delta}\psi)^2 \de \mu_{\ol{g}_r} + O\big(\frac{1}{r^5}\big)\int_{\Sigma_r}\sum_{i=1}^3(\pa_i\psi)^2 \de\mu_{\ol{g}_r}.
\end{equation}
Since $\psi$ satisfies the wave equation, we have by \eqref{eq:waveeqcoords} that
$$\ol{\Delta}\psi = \square_g \psi + \frac{1}{r^2}\pa_r(r^2\phi^{-2}\pa_r\psi) = \frac{1}{r^2}\pa_r(X_w \psi),$$
and therefore
\begin{equation}\label{eq:laplacianest}
\int_{\Sigma_r}(\ol{\Delta}\psi)^2 \de \mu_{\ol{g}_r} = \frac{1}{r^4}\int_{\Sigma_r}(\pa_r(X_w\psi))^2 \de \mu_{\ol{g}_r} \leq O\Big(\frac{1}{r^5}\Big)E[X_w \psi](r).
\end{equation}
Note that we still cannot bound this solely by the higher-order energy flux $E[X_s\psi]$. Moreover, if we were to naively bound it by both $E[X_s\psi]$ and the standard energy flux $E[\psi]$ via the product rule, the resulting terms would possess weights in $r$ that are too large to close the argument through a Gr\"onwall-type inequality. Instead, this is where we employ a standalone, auxillary energy estimate that arises from the alternative choice of vectorfield commutator $X_w$. We claim that for any solution $\psi$ to the wave equation on $\mc{R}^+$, the following estimate holds for all $r_2 > r_1 \geq r_0$:
\begin{equation}\label{eq:hoestimateweak}
\frac{1}{r_2^4}\int_{\Sigma_{r_2}} J^{\pa_r}[X_w\psi] \cdot n_{\Sigma_r} \de \mu_{\ol{g}_r} \lesa_{\Lambda,m,r_0} \frac{1}{r_1^4}\int_{\Sigma_{r_1}} J^{\pa_r}[X_w\psi] \cdot n_{\Sigma_r} \de \mu_{\ol{g}_r} +\int_{\Sigma_{r_1}} J^{\pa_r}[\psi] \cdot n_{\Sigma_r} \de \mu_{\ol{g}_r}.
\end{equation}
We will delay the proof of this estimate until later in this section. Combining the elliptic estimate \eqref{eq:ellipticenergyprf}, \eqref{eq:laplacianest}, as well as the auxillary energy estimate \eqref{eq:hoestimateweak}, we bound
\begin{align*}
O\big(\frac{1}{r^5}\big)\int_{\Sigma_r}(\pa_t^2\psi)^2\de \mu_{\ol{g}_r} &\leq O\big(\frac{1}{r^6}\big)E[X_w](r) + O\big(\frac{1}{r^2}\big)E[\psi](r)\\*
&= O\big(\frac{1}{r^2}\big)\Big[\frac{1}{r^4}E[X_w\psi](r) + E[\psi](r)\Big]\\*
&\lesa O\big(\frac{1}{r^2}\big)\Big[E[X_w\psi](r_1)+E[\psi](r_1)\Big],
\end{align*}
and so by \eqref{eq:hoestgronwall} we have
\begin{multline*}
E[X_s\psi](r_2) \lesa E[X_s\psi](r_1) + \int_{r_1}^{r_2}O\big(\frac{1}{r^2}\big)E[X_s\psi](r)\de r\\*
+\Big(\frac{1}{r_1^4}E[X_w\psi] (r_1)+E[\psi](r_1)\Big)\int_{r_1}^{r_2}\frac{1}{r^2}\de r.
\end{multline*}
As $r^{-2}$ is integrable, it follows that
\begin{equation*}
E[X_s\psi](r_2) \lesa E[X_s\psi](r_1) + E[X_w\psi](r_1) + E[\psi](r_1) +  \int_{r_1}^{r_2}O\big(\frac{1}{r^2}\big)E[X_s\psi](r)\de r.
\end{equation*}
An application of the standard Gr\"onwall inequality yields
\begin{align*}
E[X_s\psi](r_2) &\lesa_{r_0} \Big[E[X_s\psi](r_1)+E[X_w\psi](r_1)+E[\psi](r_1)\Big]\exp\Big[C\int_{r_1}^{r_2}\frac{1}{r^2}\de r\Big]\\*
&\lesa E[X_s\psi](r_1)+E[X_w\psi](r_1)+E[\psi](r_1)
\end{align*}
for all $r_2 \geq r_1$.
By the product rule we may simply bound the energy fluxes through $\Sigma_{r_1}$ like
$$E[X_w\psi](r_1) \lesa_{r_0} E[X_s\psi](r_1) + E[\psi](r_1),$$
which implies \eqref{eq:hoestimate}.
\end{proof}

\begin{remark}
To highlight the associated weights in $r$ for terms that appear in the energy estimate \eqref{eq:hoestimate}, we note that for large $r$ we have
\begin{equation}\label{eq:hoenergyexplicit}
\int_{\Sigma_r}J^{\pa_r}[X_s\psi]\cdot n \>\de \mu_{\ol{g}_r} \sim \int_{\RR\times\sph^2} r^4(\pa_r(r\phi^{-2}\pa_r \psi))^2 + r^6(\pa_t\pa_r\psi)^2 + r^8|\slashed{\nabla}\pa_r\psi|^2 \de t \de\mu_{\mr\gamma}.
\end{equation}
As was shown in \cite{Sch15}, solutions to the wave equation on the expanding region of Schwarzschild-de Sitter are bounded with respect to the energy \eqref{eq:naturalenergy}. The boundedness of the weighted higher-order energy already eliminates the possibility of an $O(r^{-1})$-weighted term being present in the asymptotic expansion of solutions, as we have
\begin{align*}
\int_{\Sigma_r}J^{\pa_r}\Big[X_s\Big(\psi_0 + \frac{\psi_1}{r}\Big)\Big]\cdot n \>\de \mu_{\ol{g}_r} &= \int_{\Sigma_r}J^{\pa_r}\Big[X_s\Big(\frac{\psi_1}{r}\Big)\Big]\cdot n \>\de \mu_{\ol{g}_r}\\*
&\sim \int_{\RR\times\sph^2} r^6 \psi_1^2  + r^2|\tilde{\nabla} \psi_1|^2 \de t \de \mu_{\mr{\gamma}},
\end{align*}
Therefore if a solution $\psi$ was  like $\psi \sim \psi_0 + \psi_1/r$ for large $r$, this higher-order energy would not be bounded.
\end{remark}
\begin{remark}
A related energy estimate was 
\end{remark}
\begin{proof}[Proof of the auxillary energy estimate \eqref{eq:hoestimateweak}]
The proof of \eqref{eq:hoestimateweak} is very similar to that of the sharp estimate in Theorem~\ref{thm:hoestimate}. We begin by noting by Lemma~\ref{lem:commutedwaveeq}
\begin{equation}\label{eq:commutedeqweak}
\phi \>\pa_r(X_w \psi) \square_g (X_w\psi) = -\phi\>\Big(\frac{4\Lambda}{3}r - \frac{2}{r} + \frac{2m}{r^2}\Big)(\pa_r(X_w\psi))^2 + 2\phi^3(r-3m)\pa_t^2\psi\pa_r(X_w\psi),
\end{equation}
and summing this with $\phi K^{\pa_r}[X_w\psi]$ gives by \eqref{eq:deftensorforward}:
\begin{multline}\label{eq:commutedenergyweak}
\phi\> \nabla \cdot J^{\pa_r}[X_w\psi] = -\phi\Big(\frac{2\Lambda}{3}r -\frac{1}{r} +\frac{m}{r^2}\Big)(\pa_r(X_w\psi))^2 + \phi^2\Big(\LT r - \frac{m}{r}\Big)(\pa_t(X_w\psi))^2\\*+ 2\phi^3(r-3m)\pa_t^2\psi \>\pa_r(X_w\psi).
\end{multline}
In contrast to \eqref{eq:commutedenergy2} for the commutator $X_s$, the $(\pa_r(X_w\psi))^2$ term in \eqref{eq:commutedenergyweak} does not vanish to leading order, we will soon see the negative sign of this term leads to a loss of decay in the resulting energy estimate. For now, we may bound each term like
\begin{gather*}
\Big|\phi\Big(\frac{2\Lambda}{3}r - \frac{1}{r}+\frac{m}{r^2}\Big)(\pa_r(X_w\psi))^2\Big| \leq \frac{4}{r} J^{\pa_r}[X_w\psi]\cdot n + O\big(\frac{1}{r^3}\big)J^{\pa_r}[X_w\psi]\cdot n.\\*
\Big|\phi^5\Big(\frac{1}{r}-\frac{3m}{r^2}\Big)(\pa_t(X_w\psi))^2\Big| \leq O\big(\frac{1}{r^3}\big)J^{\pa_r}[X_w\psi]\cdot n.\\*
\Big|2\phi^3(r-3m)\pa_t^2\psi\>\pa_r(X_w\psi)\Big| \leq O\Big(\frac{1}{r^2}\Big)J^{\pa_r}[X_w\psi]\cdot n + O\Big(\frac{1}{r^3}\Big)(\pa_t^2\psi)^2.
\end{gather*}
It follows that
$$\Big|\phi\nabla \cdot J^{\pa_r}[X_w\psi]\Big| \leq \frac{4}{r}J^{\pa_r}[X_w\psi]\cdot n + O\big(\frac{1}{r^2}\big)J^{\pa_r}[X_w\psi]\cdot n + O\big(\frac{1}{r^3}\big)(\pa_t^2\psi)^2,$$
and so by \eqref{eq:divergencecoarea2} we have
\begin{multline*}
E[X_w\psi](r_2) \leq E[X_w\psi](r_1) + \int_{r_1}^{r_2}\Big[\frac{4}{r}+O\big(\frac{1}{r^2}\big)\Big]E[X_w\psi](r)\de r\\*
+ \int_{r_1}^{r_2}O\big(\frac{1}{r^3}\big)\int_{\Sigma_r}(\pa_t^2\psi)^2 \de \mu_{\ol{g}_r}\de r.
\end{multline*}
An application of the elliptic estimate from Corollary \ref{cor:ellipticestSdS} along with \eqref{eq:laplacianest} yields
\begin{equation*}
O\big(\frac{1}{r^3}\big)\int_{\Sigma_r}(\pa_t^2\psi)^2 \de \mu_{\ol{g}_r} \lesa O\big(\frac{1}{r^4}\big)E[X_w\psi](r) + O\big(1)E[\psi](r),
\end{equation*}
and therefore there exists a constant $C > 0$ such that
$$E[X_w\psi](r_2)\leq E[X_w\psi](r_1)+C\int_{r_1}^{r_2}E[\psi](r)\de r + \int_{r_1}^{r_2}\beta(r)E[X_w\psi](r)\de r,$$
where 
$$\beta(r) = \frac{4}{r} + \frac{C}{r^2}$$
for some large constant $C$. We then use the fact from \cite{Sch15} that the natural energy $E[\psi]$ is monotonically decreasing in $r$ to bound
$$\int_{r_1}^{r_2}E[\psi](r)\de r \leq r_2 E[\psi](r_1).$$
If we set $\alpha(r) = E[X_w\psi](r_1) + C r E[\psi](r_1)$, then we have for all $r_2 \geq r_1$ that
$$E[X_w\psi](r_2) \leq \alpha(r_2) + \int_{r_1}^{r_2} \beta(r) E[X_w\psi](r) \de r,$$
which is a Gr\"onwall-type inequality, and implies
$$E[X_w\psi](r_2) \leq \alpha(r_2) +\int_{r_1}^{r_2} \alpha(s)\beta(s)\exp\Big[\int_s^{r_2}\beta(r)\de r\Big] \de s.$$
Finally we bound
\begin{equation*}
\exp\Big[\int_s^{r_2}\frac{4}{r} + \frac{C}{r^2}\Big]\de r =\exp\Big[4\log(r_2)-4\log(s)\Big]\exp\Big[\int_s^{r_2}\frac{C}{r^2}\de r\Big] \lesa \frac{r_2^4}{s^4},
\end{equation*}
which implies
\begin{align*}
\int_{r_1}^{r_2} \alpha(s)\beta(s)\exp\Big[\int_s^{r_2}\beta(r)\de r\Big] \de s &\lesa r_2^4\int_{r_1}^{r_2} \frac{E[X_w\psi](r_1)}{s^5} + \frac{E[\psi](r_1)}{s^4}\de s\\*
&\lesa r_2^4 \Big[E[X_w\psi](r_1) + E[\psi](r_1)\Big].
\end{align*}
This implies \eqref{eq:hoestimateweak} after dividing through by $r_2^4$.
\end{proof}
\begin{remark}
The auxillary energy estimate \eqref{eq:hoestimateweak} is weaker than that given in Theorem~\ref{thm:hoestimate} in the following way. Writing the commuted energy from \eqref{eq:hoestimateweak} more explicitly, we have
$$\int_{\Sigma_r}J^{\pa_r}[X_w\psi] \cdot n_{\Sigma_r}\de \mu_{\ol{g}_r} \sim \int_{\RR\times\sph^2} r^4 (\pa_r (r^2\phi^{-2}\pa_r\psi))^2 + r^8 |\tilde{\nabla}\pa_r\psi|^2\>\de t \>\de \mu_{\mr{\gamma}}, $$
and so the auxillary estimate implies that the quantity
$$\frac{1}{r^4}\int_{\Sigma_r}J^{\pa_r}[X_w\psi] \cdot n_{\Sigma_r}\de \mu_{\ol{g}_r} \sim \int_{\RR\times\sph^2} (\pa_r (r^2\phi^{-2}\pa_r\psi))^2 + r^4 |\tilde{\nabla}\pa_r\psi|^2 \>\de t \>\de \mu_{\mr{\gamma}}$$
is bounded. On the other hand, it follows from Theorem~\ref{thm:hoestimate} that
$$\int_{\Sigma_r}J^{\pa_r}[X_s\psi] \cdot n_{\Sigma_r}\de \mu_{\ol{g}_r} \sim \int_{\RR\times\sph^2} (\pa_r (r \phi^{-2}\pa_r\psi))^2 + r^6 |\tilde{\nabla}\pa_r\psi|^2 \>\de t \>\de \mu_{\mr{\gamma}}$$
is bounded, which yields stronger decay of $\pa_r\psi$ with respect to the homogeneous Sobolev norm $\dot{H}^1(\RR\times\sph^2)$.
\end{remark}

\subsection{Second-order energy estimate}\label{sec:secondenergy}
We now prove the second-order estimate of Theorem~\ref{thm:forwardenergyintro}. This estimate follows from a second commutation with the vectorfield of the form $Y = r\phi^{-1}\pa_r$, after already commuting once with $X_s$. The commutator $Y$ once again creates cancellations in various terms that are produced when computing $\square_g (Y(X_s\psi))$.
\begin{proof}[Proof of Theorem~\ref{thm:forwardenergyintro}]
We begin with the inequality
\begin{equation}\label{eq:hhorderdivergenceid}
E[Y(X_s\psi)](r_2) + \int_{r_1}^{r_2}\Big(\int_{\Sigma_r}\phi \nabla \cdot J^{\pa_r}[Y(X_s\psi)]\de \mu_{\ol{g}_r}\Big)\de r \leq E[Y(X_s\psi)](r_1).
\end{equation}
By Lemma~\ref{lem:commutedwaveeq}, for some scalar function $\varphi$ we have
$$\square_g (Y\varphi) = \phi^{-2}\pa_r^2(r\phi^{-1})\pa_r \varphi + 2\phi^3\Big(1-\frac{3m}{r}\Big)\pa_t^2\varphi + 2\phi^{-1}\square_g \varphi + Y(\square_g \varphi) .$$
Setting $\varphi = X_s\psi$, it follows that
\begin{multline}\label{eq:hhordercom}
\square_g (Y(X_s\psi)) = \phi^{-2}\pa_r^2(r\phi^{-1})(X_s\psi) + 2\phi^3\Big(1-\frac{3m}{r}\Big)\pa_t^2(X_s\psi) \\*+2\phi^{-1}\Big[-\Big(\frac{2\Lambda}{3}r - \frac{2m}{r^2}\Big)\pa_r(X_s\psi) + 2\phi^2\Big(1-\frac{3m}{r}\Big)\pa_t^2\psi\Big]\\*
+ Y\Big[-\Big(\frac{2\Lambda}{3}r - \frac{2m}{r^2}\Big)\pa_r(X_s\psi) + 2\phi^2\Big(1-\frac{3m}{r}\Big)\pa_t^2\psi\Big],
\end{multline}
where we used the commutation identity \eqref{eq:commutedeqstrong} for the commutator $X_s$. We expand the final term giving
\begin{multline*}
Y\Big[-\Big(\frac{2\Lambda}{3}r - \frac{2m}{r^2}\Big)\pa_r(X_s\psi) + 2\phi^2\Big(1-\frac{3m}{r}\Big)\pa_t^2\psi\Big]\\*
= -\phi^{-1}\Big(\frac{2\Lambda}{3}r^2 - \frac{2m}{r}\Big)\pa_r^2(X_s\psi) - \phi^{-1}\Big(\frac{2\Lambda}{3}r - \frac{4m}{r^2}\Big)\pa_r(X_s\psi)\\*
+2\phi^3\Big(1-\frac{3m}{r}\Big)\pa_t^2(X_s\psi) +  2r\phi^{-1}\pa_r\Big[\phi^2\Big(1-\frac{3m}{r}\Big)\Big]\pa_t^2\psi.
\end{multline*}
Since $Y = r\phi^{-1}\pa_r$, we have
$$\pa_r^2(X_s\psi) = \frac{\phi}{r}\pa_r(Y(X_s\psi)) - \frac{\phi}{r}\pa_r(r\phi^{-1})\pa_r(X_s\psi).$$
It follows from \eqref{eq:hhordercom} that
\begin{multline}\label{eq:hhordercom2}
\square_g(Y(X_s\psi)) =-\Big(\frac{2\Lambda}{3}r - \frac{2m}{r^2}\Big)\pa_r(Y(X_s\psi))+4\phi^3\Big(1-\frac{3m}{r}\Big)\pa_t^2(X_s\psi)\\* 
+\Big[2r\phi^{-1}\pa_r\Big[\phi^2\Big(1-\frac{3m}{r}\Big)\Big]+2\phi\Big(1-\frac{3m}{r}\Big)\Big] \pa_t^2\psi\\*
+\Big[\phi^{-2}\pa_r^2(r\phi^{-1})+\pa_r(r\phi^{-1})\Big(\frac{2\Lambda}{3}r-\frac{2m}{r^2}\Big)-\phi^{-1}\Big(\frac{6\Lambda}{3}r-\frac{8m}{r^2}\Big)\Big]\pa_r(X_s\psi).
\end{multline}
We now observe that the coefficient of the $\pa_r(X_s\psi)$ in \eqref{eq:hhordercom2} vanishes to leading order, as
\begin{gather*}
\phi^{-2}\pa_r^2(r\phi^{-1})= \frac{2}{r\phi^3} + O(1).\\*
\pa_r(r\phi^{-1})\Big(\frac{2\Lambda}{3}r-\frac{2m}{r^2}\Big) = \frac{4}{r\phi^3}+O(1).\\*
-\phi^{-1}\Big(\frac{6\Lambda}{3}r-\frac{8m}{r^2}\Big) = -\frac{6}{r\phi^3}+O(1).
\end{gather*}
Therefore the coefficient of the $\pa_r(X_s\psi)$ term is of order $O(1)$. This is also true for the coefficient of the $\pa_t^2\psi$ term, as
\begin{gather*}
2r\phi^{-1}\pa_r\Big[\phi^2\Big(1-\frac{3m}{r}\Big)\Big] = -2\phi + O\big(\frac{1}{r^2}\big)\\*
2\phi\Big(1-\frac{3m}{r}\Big) = 2\phi + O\big(\frac{1}{r^2}\big),
\end{gather*}
and so the coefficient of the $\pa_t^2\psi$ term is of order $O(r^{-2})$. If we keep track of the leading-order behaviour of the coefficients of the other terms as well, we get
\begin{multline}\label{eq:hhordercom3}
\phi\pa_r(Y(X_s\psi))\square_g(Y(X_s\psi)) = \\*-\phi\Big(\frac{2\Lambda}{3}r - \frac{2m}{r^2}\Big)[\pa_r(Y(X_s\psi))]^2 +O\big(\frac{1}{r}\big)\pa_r(Y(X_s\psi))\pa_r(X_s\psi)\\*
+O\big(\frac{1}{r^4}\big)\pa_r(Y(X_s\psi))\pa_t^2(X_s\psi) + O\big(\frac{1}{r^3}\Big)\pa_r(Y(X_s\psi))\pa_t^2\psi.
\end{multline}
We restate that
$$\phi K^{\pa_r}[Y(X_s\psi)] = \phi\Big(\frac{2\Lambda}{3}r - \frac{1}{r} + \frac{m}{r}\Big)[\pa_r(Y(X_s\psi))]^2 + \phi^5\Big(\frac{1}{r}-\frac{3m}{r^2}\Big)[\pa_t(Y(X_s\psi))]^2,$$
and point out that the coefficients of the $[\pa_r(Y(X_s\psi))]^2$ terms found above and in \eqref{eq:hhordercom3} cancel to leading order. This means that essentially all terms now become lower-order error terms that we may control with lower-order energy estimates, elliptic estimates and a Gr\"onwall-type inequality, similar to in Theorem~\ref{thm:hoestimate}. Thus we have
\begin{multline}\label{eq:hhorder4}
\phi \nabla\cdot J^{\pa_r}[Y(X_s\psi)] = O\big(\frac{1}{r^2}\big)[\pa_r(Y(X_s\psi))]^2 + O\big(\frac{1}{r^6}\big)[\pa_t(Y(X_s\psi))]^2\\*
+ O\big(\frac{1}{r}\big)\pa_r(Y(X_s\psi))\pa_r(X_s\psi))\\*
+O\big(\frac{1}{r^4}\big)\pa_r(Y(X_s\psi))\pa_t^2(X_s\psi) + O\big(\frac{1}{r^3}\big)\pa_r(Y(X_s\psi))\pa_t^2\psi.
\end{multline}
We now bound all terms on the right hand side of \eqref{eq:hhorder4} pointwise by energy energy currents, if possible. We have
\begin{gather*}
\frac{1}{r^2}[\pa_r(Y(X_s\psi))]^2 \leq O\big(\frac{1}{r^3}\big)J^{\pa_r}[Y(X_s\psi)]\cdot n.\\*
\frac{1}{r^6}[\pa_t(Y(X_s\psi))]^2 \leq O\big(\frac{1}{r^3}\big)J^{\pa_r}[Y(X_s\psi)]\cdot n.\\*
\frac{1}{r}\Big|\pa_r(Y(X_s\psi))\pa_r(X_s\psi))\Big| \leq O\big(\frac{1}{r^2}\big)J^{\pa_r}[Y(X_s\psi)]\cdot n + O\big(\frac{1}{r^2}\big)J^{\pa_r}[X_s\psi]\cdot n.\\*
\frac{1}{r^4}\Big|\pa_r(Y(X_s\psi))\pa_t^2(X_s\psi)\Big| \leq O\big(\frac{1}{r^2}\big)J^{\pa_r}[Y(X_s\psi)]\cdot n + O\big(\frac{1}{r^7}\big)(\pa_t^2(X_s\psi))^2.\\*
\frac{1}{r^3}\Big|\pa_r(Y(X_s\psi))\pa_t^2\psi\Big| \leq O\big(\frac{1}{r^2}\big)J^{\pa_r}[Y(X_s\psi)]\cdot n + O\big(\frac{1}{r^5}\big)(\pa_t^2\psi)^2.
\end{gather*}
It follows from \eqref{eq:hhorderdivergenceid} and \eqref{eq:hhorder4} that

\begin{multline}\label{eq:hhorder5}
E[Y(X_s\psi)](r_2) \leq E[Y(X_s\psi)](r_1) + \int_{r_1}^{r_2}O\big(\frac{1}{r^2}\big)E[Y(X_s\psi)](r)\de r\\*
+\int_{r_1}^{r_2}O\big(\frac{1}{r^2}\big)E[X_s\psi](r)\Big]\de r
+\int_{r_1}^{r_2}O\big(\frac{1}{r^7}\big)\Big[\int_{\Sigma_r}(\pa_t^2(X_s\psi))^2\de \mu_{\ol{g}_r}\Big] \de r\\*
+\int_{r_1}^{r_2}O\big(\frac{1}{r^5}\big)\Big[\int_{\Sigma_r}(\pa_t^2\psi)^2\de \mu_{\ol{g}_r} \Big]\de r.
\end{multline}
The final term can be controlled by $O(1/r^2)[E[X_s\psi](r) + E[\psi](r)]$ via an elliptic estimate like in Theorem~\ref{thm:hoestimate}. For the second-to-last term, another elliptic estimate gives
$$\frac{1}{r^7}\int_{\Sigma_r} (\pa_t^2(X_s\psi))^2\de \mu_{\ol{g}_r} \leq O\big(\frac{1}{r^3}\big)\int_{\Sigma_r}(\ol{\Delta}(X_s\psi))^2 \de \mu_{\ol{g}_r} + O\big(\frac{1}{r^7}\big)\sum_{i=1}^3\int_{\Sigma_r}|\pa_i(X_s\psi)|^2 \de \mu_{\ol{g}_r}.$$
The final term can be controlled by $O(r^{-4})E[X_s\psi](r)$. Moreover, observe that
\begin{multline*}
\ol{\Delta}(X_s\psi) = \square_g (X_s\psi) + \frac{1}{r^2}\pa_r(r^2\phi^{-2}\pa_r(X_s\psi))\\*
= -\Big(\frac{2\Lambda}{3}r - \frac{2m}{r^2}\Big)\pa_r(X_s\psi) + 2\phi^2\Big(1-\frac{3m}{r}\Big)\pa_t^2\psi \\*+ \frac{1}{r\phi}\pa_r(Y(X_s\psi)) + \frac{1}{r\phi}\pa_r(r\phi^{-1})\pa_r(X_s\psi),
\end{multline*}
and therefore
\begin{equation*}
\frac{1}{r^3}(\ol{\Delta}(X_s\psi))^2\leq O\big(\frac{1}{r^2}\big)J^{\pa_r}[X_s\psi]\cdot n + O\big(\frac{1}{r^7}\big)(\pa_t^2\psi)^2+O\big(\frac{1}{r^4}\big)J^{\pa_r}[Y(X_s\psi)]\cdot n.
\end{equation*}
We estimate each of these terms as before, giving
\begin{equation*}
\frac{1}{r^3}\int_{\Sigma_r}(\ol{\Delta}(X_s\psi))^2 \de \mu_{\ol{g}_r} \leq O\big(\frac{1}{r^2}\big)\Big\{E[Y(X_s\psi)](r)+E[X_s\psi](r)+E[\psi](r)\Big\}.
\end{equation*}
Therefore
\begin{multline}\label{eq:hhorder6}
E[Y(X_s\psi)](r_2) \leq E[Y(X_s\psi)](r_1)+ \int_{r_1}^{r_2}O\big(\frac{1}{r^2}\big)E[Y(X_s\psi)](r)\de r\\*
+\int_{r_1}^{r_2}O\big(\frac{1}{r^2}\big)\Big\{E[X_s\psi](r)+E[\psi](r)\Big\}\de r.
\end{multline}
An application of Theorem~\ref{thm:hoestimate} implies
$$\int_{r_1}^{r_2}\frac{1}{r^2}\Big\{E[X_s\psi](r)+E[\psi](r)\Big\}\de r \lesa E[X_s\psi](r_1) + E[\psi](r_1).$$
The result now follows by a Gr\"onwall inequality, and the fact that the factor $r^{-2}$ is integrable.
\end{proof}
\subsection{Pointwise decay estimates}

For sufficiently regular solutions to $\square_g \psi = 0$, the energy estimates from Theorem~\ref{thm:hoestimate} implies pointwise decay of higher derivatives of $\psi$. We return to the argument of Section~\ref{sec:pointwisescattering} that takes advantage of the Killing vectorfields of Schwarzschild-de Sitter, as well as the Sobolev estimate from Appendix~\ref{sec:sobolev}. To this end, recall the notation that given a multiindex $I = (I_1,I_2,I_3,I_4)$, $Z^I$ denotes the differential operator
$$Z^I = \Omega_1^{I_1}\Omega_2^{I_2}\Omega_3^{I_3}T^{I_4},$$
where $T,\Omega_i:i=1,2,3$ are Killing vector fields that span the tangent space of the $\Sigma_r$ hypersurfaces.
\begin{corollary}
Let $\psi$ be a solution to the linear wave equation $\square_g \psi = 0$, with finite higher-order energy on an initial hypersurface $\Sigma_{r_0}$,
$$D[\psi](r_0) = \sum_{|I|\leq 2}\Big\{E[Y(X_s(Z^I\psi))](r_0) + E[X_s(Z^I\psi)](r_0)+E[Z^I\psi](r_0)\Big\} < \infty.$$
Then we have the following pointwise bounds for large $r$:
\begin{gather}
\sup_{\Sigma_r}|\pa_r\psi| \lesa \frac{1}{r^3}(D[\psi](r_0))^{1/2}.\label{eq:higherorderpw1}\\*
\sup_{\Sigma_r}|\tilde{\nabla}(X_s\psi)| \lesa (D[\psi](r_0))^{1/2}.\label{eq:higherorderpw2}\\*
\sup_{\Sigma_r}|\pa_r(X_s\psi)| \lesa \frac{1}{r^2}(D[\psi](r_0))^{1/2}.\label{eq:higherorderpw3}\\*
\sup_{\Sigma_r}|\tilde{\nabla}(Y(X_s\psi))|\lesa (D[\psi](r_0))^{1/2}.\label{eq:higherorderpw4}
\end{gather}
\end{corollary}
\begin{proof}
Using the Sobolev estimate from Proposition~\ref{SobEst}, we have
\begin{equation*}
\sup_{\Sigma_r}|X_s\psi| \lesa \frac{1}{r^{3/2}}\sum_{|I|\leq 2}\|X_s(Z^I\psi)\|_{L^2(\Sigma_r)},
\end{equation*}
keeping in mind that the differential operators $Z^I$ commute with $X_s$. An application of the fundamental theorem of calculus on $[r_0,r]$ then implies
$$\frac{1}{r^{3/2}}\|X_s(Z^I\psi)\|_{L^2(\Sigma_r)} \lesa \|X_s(Z^I\psi)\|_{L^2(\Sigma_{r_0})} + \int_{r_0}^{r}\frac{1}{r^{3/2}}\|\pa_r (X_s(Z^I\psi))\|_{L^2(\Sigma_r)}\de r.$$
We control the last term with the higher-order energy estimate from Theorem~\ref{eq:hoestimate}. Since $Z^I$ is composed of Killing vector fields, $Z^I\psi$ satisfies the homogeneous wave equation, and so by \eqref{eq:hoestimate} we have
\begin{align*}\frac{1}{r^{3/2}}\|\pa_r (X_s(Z^I\psi))\|_{L^2(\Sigma_r)} &\lesa \frac{1}{r^2}\Big(E[X_s(Z^I\psi)](r)\Big)^{1/2}\\*
&\lesa \frac{1}{r^2}\Big(E[X_s(Z^I\psi)](r_0)+E[Z^I\psi](r_0)\Big)^{1/2}\\*
&\lesa \frac{1}{r^2}\Big(D[\psi](r_0)\Big)^{1/2}.
\end{align*}
Therefore
\begin{align*}
\sup_{\Sigma_r}|X_s\psi|  &\lesa \sum_{|I|\leq 2}\Big\{\|X_s(Z^I\psi)\|_{L^2(\Sigma_{r_0})} + \int_{r_0}^r \frac{1}{r^2}\Big(D[\psi](r_0)\Big)^{1/2}\de r\Big\}\\*
&\lesa \Big(D[\psi](r_0)\Big)^{1/2}.
\end{align*}
Since $X_s \sim r^3 \pa_r$, we divide the above inequality through by $r^3$ and obtain \eqref{eq:higherorderpw1}. As for \eqref{eq:higherorderpw2}, we use the Sobolev estimate again and control the resulting $L^2$-norm by the higher-order energy, so that
\begin{align*}
\sup_{\Sigma_r}|\tilde{\nabla} (X_s\psi)|^2 &\lesa \frac{1}{r^3}\sum_{|I|\leq 2}\|\tilde{\nabla}(X_s\psi)\|_{L^2(\Sigma_r)}^2
\\*&\lesa \sum_{|I|\leq 2} E[X_s(Z^I\psi)](r),
\end{align*}
and finally appeal to the higher-order estimate from Theorem~\ref{thm:hoestimate}. The decay estimates \eqref{eq:higherorderpw3} and \eqref{eq:higherorderpw4} follow in the same manner from the second-order estimate of Theorem~\ref{thm:forwardenergyintro}.
\end{proof}

\subsection{Concluding the proof of Theorem~\texorpdfstring{\ref{thm:forwardasymptotics}}{1.8}}\label{sec:forwardlimits}
We conclude this section by proving the existence of the limits 
\begin{equation}\label{eq:forwardlimitssec3}\psi_0 = \lim_{r\ra\infty}\psi,\quad \psi_2 = -2\lim_{r\ra\infty}[r^3\pa_r\psi],\quad \psi_3 = 3\lim_{r\ra\infty}[r^2\pa_r(r^3\pa_r\psi)]
\end{equation}
with respect to the $\dot{H}^1$-norm on the future boundary $\Sigma^+$, which we recall is diffeomorphic to $\RR\times\sph^2$. This is equivalent to proving that the quantities $\psi,X_s\psi,Y(X_s\psi)$ all have a limit on $\Sigma^+$. The particular constants in front of the limits for $\psi_2$, $\psi_3$ are just normalisation factors that account for repeated differentiation in $r$. We note that the energy controls the $\dot{H}^1$-norm on the $\Sigma_r$ hypersurfaces, meaning there exists a constant $C(\Lambda,m,r_0) > 0$ independent of $r$ such that
\begin{gather*}
\|\psi(r)\|_{\dot{H}^1(\RR\times\sph^2)}^2 \leq C E[\psi](r),
\end{gather*}
for all $r \geq r_0$. To prove the existence of the limits \eqref{eq:forwardlimitssec3}, we appeal to a density argument. Let $(\varphi_{j})_{j\in \NN}$ be a sequence of smooth, compactly supported solutions to the linear wave equation that converge to $\psi$, so that for some $r_0 > \rc$ we have
$$\lim_{j \ra \infty}\Big[ \|\psi-\varphi_j\|_{\pa_r,\Sigma_{r_0}} + \|X_s(\psi-\varphi_j)\|_{\pa_r,\Sigma_{r_0}} + \|Y(X_s(\psi-\varphi_j))\|_{\pa_r,\Sigma_{r_0}} \Big] = 0.$$
This can be accomplished by taking a sequence of smooth compactly supported data on $\Sigma_{r_0}$ that converge to the data of $\psi$ on $\Sigma_{r_0}$, and then considering the corresponding solution to the wave equation. By commuting the energy estimate from Theorem~\ref{thm:forwardenergyintro} with Killing vectorfields, one can then show that for each $j$, the functions $\varphi_j, X_s\varphi_j, Y(X_s\psi)$ each have a pointwise limit on $\Sigma^+$. This pointwise limit extends to a limit with respect to the $\dot{H}^1$-norm via the dominated convergence theorem. Finally, it follows from the energy estimate of Theorem~\ref{thm:hoestimate} that $\lim_{r\ra\infty} \psi$, $\lim_{r\ra\infty} X_s\psi$, and $\lim_{r\ra\infty} Y(X_s\psi)$ all exist with respect to the homogeneous Sobolev norm $\dot{H}^1(\RR\times\sph^2)$.

We finally prove the existence of the limits \eqref{eq:asymptoticlimits} by showing that
\begin{gather}\lim_{r\ra\infty}\|r^2(\psi-\psi_0)-\psi_2\|_{\dot{H}^{1}(\RR\times\sph^2)} = 0.\label{eq:existenceasymptotic1}\\*
\lim_{r\ra\infty}\Big\|r^3\Big(\psi-\psi_0-\frac{\psi_2}{r^2}\Big) - \psi_3\Big\|_{\dot{H}^{1}(\RR\times\sph^2)} = 0.\label{eq:existenceasymptotic2}
\end{gather}
To prove \eqref{eq:existenceasymptotic1}, we know that $\|\psi - \psi_0 - r^{-2}\psi_r\|_{\dot{H}^1} \ra 0$ as $r\ra \infty$, and so by the fundamental theorem of calculus we have
$$\Big\|\psi - \psi_0 - \frac{1}{r^{2}}\psi_2\Big\|_{\dot{H}^1(\RR\times\sph^2)} \leq \int_r^{\infty}\Big\|\pa_r\psi+\frac{2}{r^{3}}\psi_2\Big\|_{\dot{H}^1(\RR\times\sph^2)}\de r,$$
noting that $\pa_r(\psi-\psi_0-r^{-2}\psi_2) = \pa_r\psi+2r^{-3}\psi_2$. We bound
\begin{align*}
\int_r^{\infty}\Big\|\pa_r\psi+\frac{2}{r^{3}}\psi_2\Big\|_{\dot{H}^1(\RR\times\sph^2)}\de r &\leq \int_{r}^{\infty}\frac{1}{r^3}\sup_{[r,\infty)}\|r^3\pa_r\psi + 2\psi_2\|_{\dot{H}^1(\RR\times\sph^2)}\de r\\*
&\leq \frac{1}{2r^2}\sup_{[r,\infty)}\|r^3\pa_r\psi + 2\psi_2\|_{\dot{H}^1(\RR\times\sph^2)}.
\end{align*}
This implies \eqref{eq:existenceasymptotic1}, as
$$\|r^2(\psi-\psi_0)-\psi_2\|_{\dot{H}^1(\RR\times\sph^2)} \leq \frac{1}{2}\sup_{[r,\infty)}\|r^3\pa_r\psi + 2\psi_2\|_{\dot{H}^1(\RR\times\sph^2)}.$$
The right hand side vanishes as $r \ra \infty$, hence \eqref{eq:existenceasymptotic1} holds. The limit \eqref{eq:existenceasymptotic2} follows in a similar manner.

\section{Scattering on general expanding spacetimes}\label{sec:perturbations}
We will now generalise the scattering result \textbf{Theorem~\ref{thm:scatteringwithouthorizon}} to the linear wave equation on a wide class of expanding spacetimes. No symmetries are required on the expanding region, and the geometry at infinity of these spacetimes can differ significantly from Schwarzschild-de Sitter. In particular the expanding region of Kerr-de Sitter is contained in the class of spacetimes we consider, so we prove existence and uniqueness of scattering solutions to the linear wave equation on Kerr-de Sitter.

\begin{remark}
Wave equations on another class of perturbations of Schwarzschild-de Sitter were also considered in \cite{Sch15}, and the redshift estimate \eqref{eq:redshift0} for the forward problem was shown to generalise to those perturbations. We note the perturbed spacetimes considered in \cite{Sch15} are in a more restricted class compared to those that we consider in this paper, as their metric components are assumed to converge at infinity to their counterparts on Schwarzschild-de Sitter.
\end{remark}

For a given spacetime, we will construct asymptotic solutions to the wave equation which are specialised to that spacetime. We note that for the class of perturbations of Schwarzschild-de Sitter we consider, logarithmic terms are present in the asymptotic solution, see Remark~\ref{rmk:logterms}. However we will show that for a restriction of the perturbed class that includes Kerr-de Sitter, logarithmic terms are not present.

We will then prove a weighted energy estimate suitable for the backward problem, extending Proposition~\ref{prop:backwardenergyestimate} to this large class of spacetimes. The existence of scattering solutions from Theorem~\ref{thm:pertscatteringinformal} then follows by repeating the argument from Section~\ref{sec:scatteringinterior}, but in a perturbed setting.

We define the following class of spacetimes:
\begin{definition}\label{def:perturbations}
Let $\Sigma^+$ denote a 3-dimensional Riemannian manifold, equipped with a Riemannian $C^3$-metric $h$, and fix $\Lambda > 0$. Let $M$ denote the differentiable manifold $M = (0,\rho_0)\times \Sigma^+$ for some $\rho_0 > 0$, and let $g$ be a Lorentzian $C^3$-metric defined on $M$. 

We say $(M,g) \in \mc{G}_\Lambda^1$ if the conformal metric $\tilde{g} = \rho^2 g$ admits a $C^3$-extension across $\{\rho = 0\} \times \Sigma^+$ of the form
\begin{equation}\label{eq:conformalmetricbad}
\tilde{g} = -\frac{3}{\Lambda} \de \rho^2 + h + O(\rho).
\end{equation}
Moreover, we say $(\mc{M},g) \in \mc{G}_{\Lambda}^2$ if the conformal metric $\tilde{g}$ is like
\begin{equation}\label{eq:conformalmetricgood}
\tilde{g} = -\frac{3}{\Lambda} \de \rho^2 + h + O(\rho^2).
\end{equation}
\end{definition}
\begin{remark}
This class of spacetimes is very similar those considered in \cite{Vas10}, where in particular smooth scattering solutions are constructed from infinity. The primary difference here is that we do not assume smoothness of the metric, but instead $C^3$-regularity. This class is also interesting because of its relation to the scattering problem for Einstein's vacuum equations with positive cosmological constant, see for example \cite{Hintz23}.
\end{remark}
\begin{remark}
A straightforward computation shows that de Sitter, Schwarzschild-de Sitter and Kerr-de Sitter spacetimes are all members of $\mc{G}_\Lambda^2$. For de Sitter and Schwarzschild-de Sitter, one can see this by taking the metric in standard coordinates on the expanding region (for example take \eqref{eq:metricstandardcoords} for Schwarzschild-de Sitter), and introducing the coordinate $\rho = r^{-1}$. For Kerr-de Sitter this can be seen by considering the metric in \emph{Boyer-Lindquist} coordinates on the expanding region, and again introducing the coordinate $\rho = r^{-1}$.
\end{remark}
It is simpler for us to do the computations in this section with respect to the conformal metric $\tg$ rather than $g$. For the rest of this section, we assume that unless explicitly stated otherwise, all quantities are defined in terms of $\tg$ and not the original metric.

First, we note that the homogeneous wave equation $\square_g \psi = 0$ is equivalent to the following wave equation on the conformal metric $\tg$:
\begin{equation}\label{eq:conformalwaveequationsimple}
\square_{\tg} \psi = \frac{2}{\rho}\tg^{\rho\nu}\pa_\nu \psi,
\end{equation}
where a subscript or superscript $\rho$ is to be treated as the component in terms of the $\rho$-coordinate and not an abstract index.
If we set $\phi = (-\tg^{\rho\rho})^{-1/2}$, then the future-pointing timelike unit normal $n$ of the level sets of constant $\rho$ is given by \begin{equation}\label{eq:normalvfperturbed}
n^\mu = \phi \tg^{\rho \mu}
\end{equation}
and so we may write
\begin{equation}\label{eq:conformalwaveequation}
\square_{\tg} \psi = \frac{2}{\rho\phi}n\psi.
\end{equation}
Note that for the conformal metric neither the unit normal $n$ nor the lapse $\phi$ possess weights in $\rho$, and both are uniformly continuously differentiable up to $\Sigma^+$.

\subsection{Constructing asymptotic solutions}
In this section we construct asymptotic solutions to the wave equation for metrics in $\mc{G}_{\Lambda}^1$, $\mc{G}_{\Lambda}^2$. While these constructions are similar to their counterparts on Schwarzschild-de Sitter, in the perturbed setting we include logarithmic terms in the asymptotic solution. The wave equation \eqref{eq:conformalwaveequation} can be written as
\begin{multline}\label{eq:conformalwavemetric}
\Big(\square_{\tg} - \frac{2}{\rho\phi}n\Big)\psi = \tg^{\rho\rho}\pa_\rho^2\psi + 2\tg^{\rho i}\pa_\rho \pa_i \psi + \tg^{ij}\pa_i\pa_j \psi\\*
- \tg^{\mu\nu}\Gamma_{\mu\nu}^\rho\pa_\rho \psi - \tg^{\mu\nu}\Gamma_{\mu\nu}^i \pa_i\psi - \frac{2}{\rho\phi}\tg^{\rho\rho}\pa_\rho \psi -\frac{2}{\rho\phi}\tg^{\rho i}\pa_i \psi.
\end{multline}
If the metric belongs to $\mc{G}_\Lambda^1$, then a short computation using \eqref{eq:conformalmetricbad} implies that the asymptotics of the various coefficients that appear above are
\begin{gather}
\tg^{\rho\rho} = -\frac{\Lambda}{3} + O(\rho), \quad \tg^{\rho i} = O(\rho), \quad \tg^{ij} = h^{ij} + O(\rho),\label{eq:wavetermsasympslow1}\\*
\tg^{\mu\nu}\Gamma_{\mu\nu}^\rho = O(1),\quad\tg^{\mu\nu}\Gamma_{\mu\nu}^i = f_{\Sigma^+}^i + O(\rho),\label{eq:wavetermsasympslow2}
\end{gather}
where the $f_{\Sigma^+}^i$ are $C^2$-functions on $\Sigma^+$. If the metric also belongs to $\mc{G}_\Lambda^2$, then by \eqref{eq:conformalmetricgood} the asymptotics are
\begin{gather}
\tg^{\rho\rho} = -\frac{\Lambda}{3} + O(\rho^2), \quad \tg^{\rho i} = O(\rho^2), \quad \tg^{ij} = h^{ij} + O(\rho^2).\label{eq:wavetermsasympfast1}\\*
\tg^{\mu\nu}\Gamma_{\mu\nu}^\rho = O(\rho),\quad\tg^{\mu\nu}\Gamma_{\mu\nu}^i = f_{\Sigma^+}^i+  O(\rho^2).\label{eq:wavetermsasympfast2}
\end{gather}
\begin{lemma}\label{lem:asympperturbed}
Fix a spacetime $(\mc{M},g) \in \mc{G}_{\Lambda}^1$, and let $\psi_0,\psi_{3}$ be sufficiently regular functions on $\Sigma^+$. Then define the asymptotic solution
$$\psi_{\text{asymp}} = \psi_0 + \rho^2\psi_2 + \rho^3\log \rho \>\psi_{3,1} + \rho^3\psi_3,$$
where $\psi_2$,$\psi_{3,1}$ are functions determined entirely by $\psi_0$ and derivatives thereof. Then there exists an integer $k \geq 0$ such that
\begin{equation}\label{eq:asympboundpw1}
\Big|\Big(\square_{\tg} - \frac{2}{\rho\phi}n\Big)(\psi_{\text{asymp}})\Big| \lesa_{\mc{M}} \rho^2 \log \rho \sum_{|\alpha|\leq k}|\nabla_{\Sigma^+}^\alpha \psi_0| + \rho^2\sum_{|\alpha|\leq k}|\nabla_{\Sigma^+}^\alpha \psi_3|,
\end{equation}
where $\nabla_{\Sigma^+}$ is the covariant derivative restricted to $\Sigma^+$. If additionally $(\mc{M},g) \in \mc{G}_{\Lambda}^2$, then the $\psi_{3,1}$ term vanishes identically, and we have
\begin{equation}\label{eq:asympboundpw2}\Big|\Big(\square_{\tg} - \frac{2}{\rho\phi}n\Big)(\psi_{\text{asymp}})\Big| \lesa \rho^2 \sum_{|\alpha|\leq k}|\nabla_{\Sigma^+}^\alpha \psi_0| + \rho^3\sum_{|\alpha|\leq k}|\nabla_{\Sigma^+}^\alpha \psi_3|
\end{equation}
\end{lemma}
\begin{proof}
We first assume just $(\mc{M},g) \in \mc{G}_{\Lambda}^1$. Then we compute term-by-term. By \eqref{eq:conformalwavemetric}, \eqref{eq:wavetermsasympslow1}, and \eqref{eq:wavetermsasympslow2} we have
\begin{gather*}
\Big(\square_{\tg} - \frac{2}{\rho\phi}n\Big)(\psi_0) = h^{ij}\pa_i\pa_j \psi_0 - f_{\Sigma^+}^i\pa_i \psi + O(\rho),\\* \quad \Big(\square_{\tg} - \frac{2}{\rho\phi}n\Big)(\rho^2\psi_2) = \frac{2\Lambda}{3}\psi_2 + O(\rho).
\end{gather*}
For the remaining terms, we note that in \eqref{eq:conformalwavemetric} the slowest-decaying terms are $\tg^{\rho\rho}\pa_\rho^2\psi$ and $\frac{2}{\rho\phi}\tg^{\rho\rho}$. But for $\rho^3\log \rho\> \psi_{3,1}$, $\rho^3\psi_3$ these terms cancel to leading order, giving
\begin{gather*}
\Big(\tg^{\rho\rho}\pa_\rho^2-\frac{2}{\rho\phi}\tg^{\rho\rho}\Big)(\rho^3 \log \rho \>\psi_{3,1}) = \Lambda\rho + O(\rho^2\log \rho),\\*
\Big(\tg^{\rho\rho}\pa_\rho^2-\frac{2}{\rho\phi}\tg^{\rho\rho}\Big)(\rho^3\psi_{3}) = O(\rho^2).
\end{gather*}
Hence we may freely choose $\psi_0$, $\psi_3$, and we set if we set the terms in 
$$\square_{\tg} \psi_{\text{asymp}} = \frac{2}{\rho\phi}n\psi_{\text{asymp}}$$ of order $O(1)$ and $O(\rho)$ equal to zero, then $\psi_2$, $\psi_{3,1}$ are determined entirely by $\psi_0$ and its derivatives, which implies \eqref{eq:asympboundpw1}.

Moreover if we assume $(\mc{M},g)\in\mc{G}_\Lambda^2$, then by \eqref{eq:conformalwavemetric}, \eqref{eq:wavetermsasympfast1}, and \eqref{eq:wavetermsasympfast2} we have instead for $\psi_0$, $\psi_2$,
\begin{gather*}
\Big(\square_{\tg} - \frac{2}{\rho\phi}n\Big)(\psi_0) = h^{ij}\pa_i\pa_j \psi_0 - f_{\Sigma^+}^i\pa_i \psi + O(\rho^2),\\* \quad \Big(\square_{\tg} - \frac{2}{\rho\phi}n\Big)(\rho^2\psi_2) = \frac{2\Lambda}{3}\psi_2 + O(\rho^2).
\end{gather*}
Hence setting the order $O(\rho)$ term in $\square_{\tg} \psi_{\text{asymp}} = \frac{2}{\rho\phi}n\psi_{\text{asymp}}$ to be identically zero implies that $\psi_{3,1}$ = 0, and so \eqref{eq:asympboundpw2} holds.
\end{proof}
\begin{remark}
We compare the Lemma above to the construction of asymptotic solutions on Schwarzschild-de Sitter in Section~\ref{sec:asympsol}. In the notation used in this section, Lemma~\ref{lem:recurrencerelations} and \eqref{eq:boxpsiasymp} the approximate solution on Schwarzschild-de Sitter is
$$\psi_{\text{asymp}} = \psi_0 + \rho^2\psi_2 + \rho^3\psi_3,$$
and $\psi_{\text{asymp}}$ obeys the bound
$$\Big|\Big(\square_{\tg} - \frac{2}{\rho\phi}n\Big)(\psi_{\text{asymp}})\Big| \lesa \rho^2 \sum_{|\alpha|\leq 4}|\nabla_{\Sigma^+}^\alpha \psi_0| + \rho^3\sum_{|\alpha|\leq 2}|\nabla_{\Sigma^+}^\alpha \psi_3|,$$
which is essentially \eqref{eq:asympboundpw2}. Also note that while logarithmic terms may be present in asymptotic solutions in the perturbed setting, the term $\rho \>\psi_1$ cannot be present, even for metrics in $\mc{G}_\Lambda^1$. This is because 
$$\Big(\square_{\tg} - \frac{2}{\rho\phi}n\Big)(\rho \>\psi_1) = O\big(\frac{1}{\rho}\Big),$$
so any asymptotic solution containing $\rho\psi_1$ would only satisfy \eqref{eq:conformalwaveequation} up to order $\rho^{-1}$.
\end{remark}
\begin{remark}
As previously stated, Kerr-de Sitter spacetimes are in the class $\mc{G}_{\Lambda}^2$. Thus by the previous Lemma, asymptotic solutions to the linear wave equation on Kerr-de Sitter are of the form
$$\psi_{\text{asymp}} = \psi_0 + \frac{\psi_2}{r^2} + \frac{\psi_3}{r^3},$$
where $\psi_2$ is determined by $\psi_0$.
\end{remark}

\subsection{Energy estimate for the backward problem on expanding spacetimes}
We now prove a weighted backward estimate in the perturbed setting, analogous to Proposition~\ref{prop:backwardenergyestimate}. We will apply this to solutions to the inhomogeneous wave equation
\begin{equation}\label{waveeqconformalinh}
\square_{\tg}\psi - \frac{2}{\rho\phi}n\psi = F.
\end{equation}
\begin{proposition}\label{prop:weightedenergyperturbed}
Let $M = \frac{1}{\rho^4}n$, and suppose $\psi$ satisfy the inhomogeneous wave equation \eqref{waveeqconformalinh}. Then there exists a constant $C > 0$ such that
$$K^M[\psi] \leq CJ^{M}[\psi]\cdot n.$$
on all of $\tilde{M}$. Furthermore, there exists a (different) constant $B > 0$ such that for all $\rho_1\rho_2\in(0,\rho_0)$ with $\rho_1 < \rho_2$, we have
\begin{equation}\label{eq:bwenergyperturbed}
\int_{\Sigma_{\rho_2}} J^{M}[\psi]\cdot n\>\de \mu_{\gamma_\rho}  \leq B\Big[\int_{\Sigma_{\rho_1}} J^{M}[\psi]\cdot n\>\de \mu_{\gamma_\rho} + \int_{\rho_1}^{\rho_2}\Big(\int_{\Sigma_{\rho}}\frac{1}{\rho^4}F^2\Big)\de \rho\Big],
\end{equation}
where $\Sigma_\rho$ denotes the level sets of $\rho$ in $\tilde{\mc{M}}$.
\end{proposition}
\begin{remark}
We point out that for $\rho_1 < \rho_2$, the hypersurface $\Sigma_{\rho_1}$ lies to the \emph{future} of $\Sigma_{\rho_2}$, and the future boundary $\Sigma^+$ corresponds to $\{\rho = 0\}$. So this is a backwards-in-time estimate just like Proposition~\ref{prop:backwardenergyestimate}.
\end{remark}
\begin{proof}
Recall the definition \eqref{eq:normalvfperturbed} of the timelike unit normal $n$. As $M$ is a multiple of $n$, by the product rule we have
$$\>^{(M)}\pi^{\mu\nu} = n^{(\mu}\tg^{\nu)\lambda} \pa_\lambda(\rho^{-4}) + \rho^{-2}\>^{(n)}\pi^{\mu\nu} = -\frac{4\rho^{-5}}{\phi} n^\mu n^\nu +\rho^{-4}\>^{(n)}\pi^{\mu\nu},$$
which implies
$$K^M[\psi] = -\frac{4}{\rho\phi} J^{M}[\psi]\cdot n + \frac{1}{\rho^4}K^n[\psi].$$
Writing the current explicitly we may bound
$$J^M[\psi]\cdot n = \frac{1}{2}(M\psi)(n\psi) + \frac{1}{2}\frac{\phi^2}{\rho^4}|\ol{\nabla}\psi|^2 \geq \frac{1}{2}(M\psi)(n\psi),$$
and so
$$K^M[\psi] \leq -\frac{2}{\rho\phi}(M\psi)(n\psi) + \frac{1}{\rho^4}K^n[\psi].$$
We then claim that the following bound holds on all of $\tilde{\mc{M}}$:
\begin{equation}\label{eq:redshiftcurrentperturbed}
\phi |K^{n}[\psi]| \leq CJ^{n}[\psi]\cdot n
\end{equation}
The pointwise version of this bound is always true, but the bound being uniform follows from the fact that the components of $\tg$ and the function $\phi$ are continuously differentiable, moreover their derivatives are uniformly bounded across $\{\rho = 0\}$.
It follows from \eqref{eq:redshiftcurrentperturbed} that for the divergence of $J^M[\psi]$ we have
\begin{align*}\phi \nabla \cdot J^M[\psi] &\leq -\frac{2}{\rho}(M\psi)(n\psi) + \frac{\phi}{\rho^4}K^n[\psi]+ \phi(M\psi)\Big(\frac{2}{n\phi}\psi + F\Big)\\*
&\leq \frac{\phi}{\rho^4}CJ^n[\psi]\cdot n + \phi(M\psi)F \sim CJ^M[\psi]\cdot n + (M\psi)F.
\end{align*}
We apply the Cauchy-Schwartz inequality to the $(M\psi)F$ term, so that
$$(M\psi)F \leq \frac{\rho^4}{2}(M\psi)^2 + \frac{1}{2\rho^4}F^2 \lesa CJ^M[\psi]\cdot n + \frac{1}{\rho^4}F^2$$
which yields the bound
\begin{equation}\label{eq:divergenceperturbedbound}
\phi \nabla \cdot J^M[\psi] \lesa CJ^M[\psi]\cdot n + \frac{1}{\rho^4}F^2.
\end{equation}
We apply the divergence theorem to the energy current $J^{M}[\psi]$ on the region bounded in the past by $\Sigma_{\rho_2}$ and to the future by $\Sigma_{\rho_1}$ for $0 < \rho_1 < \rho_2 < P$. We denote this region by $\mc{D}_{\rho_1,\rho_2}$. This gives
$$ \int_{\Sigma_{\rho_2}} J^{M}[\psi] \cdot n \>\de \mu_{\gamma_\rho}= \int_{\Sigma_{\rho_1}} J^{M}[\psi] \cdot n \>\de \mu_{\gamma_{\rho}} + \int_{\mc{D}_{\rho_1,\rho_2}}\nabla \cdot J^{M}[\psi] \>\de \mu_{\tg}=,$$
where $\gamma_{\rho}$ is the Riemannian metric on $\Sigma_\rho$ induced from the conformal metric $\tg$. The particular signs of each integral are due to the fact that $n \sim -\pa_\rho$, i.e. the normal vectorfield is pointing in the direction of \emph{decreasing} $\rho$. We assert that $|\tg|^{1/2} = \phi|\gamma_{\rho}|^{1/2}$, which implies that the volume form $\de \mu_{\tg}$ satisfies
$$\de \mu_{\tg} = \phi \>\de \mu_{\gamma_{\rho}} \wedge \de \rho,$$
and so by the coarea formula we have
\begin{equation}\label{eq:conformalcoarea}
\int_{\Sigma_{\rho_2}} J^{M}[\psi] \cdot n \>\de \mu_{\gamma_{\rho}} = \int_{\Sigma_{\rho_1}} J^{M}[\psi] \cdot n \>\de \mu_{\gamma_{\rho}} + \int_{\rho_1}^{\rho_2}\Big(\int_{\Sigma_{\rho}}\phi\nabla \cdot J^{M}[\psi] \>\de \mu_{\gamma_{\rho}}\Big)\de \rho.
\end{equation}
Inserting \eqref{eq:divergenceperturbedbound} then gives
\begin{multline*}
\int_{\Sigma_{\rho_2}} J^{M}[\psi] \cdot n \>\de \mu_{\gamma_{\rho}} \lesa \int_{\Sigma_{\rho_1}} J^{M}[\psi] \cdot n \>\de \mu_{\gamma_{\rho}} + \int_{\rho_1}^{\rho_2}\Big(\int_{\Sigma_\rho} \frac{1}{\rho^4}F^2 \>\de \mu_{\gamma_\rho}\Big) \>\de \rho\\*
+ C\int_{\rho_1}^{\rho_2}\Big(\int_{\Sigma_\rho} J^M[\psi]\cdot n \>\de \mu_{\gamma_\rho}\Big) \>\de \rho.
\end{multline*}
The result now follows through a Gr\"onwall-type inequality.
\end{proof}
\subsection{Existence and uniqueness of scattering solutions on expanding spacetimes}\label{sec:scatteringperturbed}
We now state the full version of Theorem~\ref{thm:pertscatteringinformal}, proving existence and uniqueness of scattering solutions to the linear wave equation in the class of expanding spacetimes $\mc{G}_\Lambda^1$. Recall from \eqref{eq:energynormnotation} the notation for the energy norm
$$\|\psi\|_{M,\Sigma_\rho} = \int_{\Sigma_\rho} J^M[\psi]\cdot n\> \de \mu_{\gamma_\rho}$$
\begin{theorem}\label{thm:pertscatteringformal}
Fix $\Lambda > 0$, and let $(\mc{M},g) \in \mc{G}_\Lambda^1$ be an expanding spacetime with future boundary $\Sigma^+$. Suppose $\psi_0,\psi_3\in H^k(\Sigma^+)$ for some $k$ sufficiently large. Then there exists a unique solution to the linear wave equation $\square_g \psi = 0$ such that for all $\rho \in (0,\rho_0)$, we have
\begin{equation}\label{eq:scatteringperturbedbound}
\|\psi\|_{M, \Sigma_\rho} \lesa_{\mc{M}} \|\psi_0\|_{H^k(\Sigma^+)} + \|\psi_3\|_{H^k(\Sigma^+)}.
\end{equation}
\end{theorem}
\begin{proof}
Given $\psi_0,\psi_3$, let $\psi_{\text{asymp}}$ be the asymptotic solution constructed in Lemma~\ref{lem:asympperturbed}. Let $\psi_{\text{rem}}^{(P)}$ be a solution to the inhomogeneous wave equation
$$\Big(\square_{\tg} - \frac{2}{\rho\phi}n\Big)\psi_{\text{rem}}^{(P)} = -\Big(\square_{\tg} - \frac{2}{\rho\phi}n\Big)\psi_{\text{asymp}},$$
with trivial data on $\Sigma_P$:
$$\psi_{\text{rem}}^{(P)}|_{\Sigma_P} = n \psi_{\text{rem}}^{(P)}|_{\Sigma_P} = 0.$$
Then taking the difference $v = \psi_{\text{rem}}^{(P_1)} - \psi_{\text{rem}}^{(P_2)}$ for $0 < P_1 < P_2$, one can repeat the proof of existence of scattering solutions from Section~\ref{sec:scatteringinterior}. In particular, we show that for all $\rho \in (0,\rho_0)$, $v \ra 0$ as $P_1,P_2 \ra 0$, where the limit of $v$ is taken with respect to the norm $\|\cdot\|_{n,\Sigma_\rho}$. Here we use the bound \eqref{eq:asympboundpw1} for the asymptotic solution, and the weighted estimate of Proposition~\ref{prop:weightedenergyperturbed}. The key inequality we reach is (with a constant independent of $P_1$, $P_2$):
$$\|v\|_{M,\Sigma_\rho}^2 \leq C\int_{P_1}^{P_2} \Big(\int_{\Sigma_\rho}\frac{1}{\rho^4}\Big[\Big(\square_{\tg} - \frac{2}{\rho\phi}n\Big)\psi_{\text{asymp}}\Big]^2\>\de \mu_{\gamma_{\rho}}\Big)\>\de \rho,$$
for all $\rho > P_2$. We compare this to the corresponding bound \eqref{eq:vest3} in Section~\ref{sec:scatteringinterior}. By the pointwise estimate \eqref{eq:asympboundpw1}, we have the $L^2$-bound
$$\int_{P_1}^{P_2}\int_{\Sigma_\rho}\frac{1}{\rho^4}\Big[\Big(\square_{\tg} - \frac{2}{\rho\phi}n\Big)\psi_{\text{asymp}}\Big]^2 \de \mu_{\gamma_\rho} \de \rho \lesa \int_{P_1}^{P_2} (\log \rho)^2(\|\psi_0\|_{H^k(\Sigma^+)}^2 + \|\psi_3\|_{H^k(\Sigma^+)}^2)\> \de\rho.$$
Since $(\log \rho)^2$ is integrable up to $\rho = 0$, the right hand side above vanishes as $P_1,P_2 \ra 0$, and so $\lim_{P_1,P_2\ra0}\|v\|_{M,\Sigma_\rho} = 0$. Hence there exists $\psi$ such that for each $\rho \in (0,\rho_0)$, we have
$$\lim_{P \ra \infty}\|\psi - \psi_{\text{asymp}} - \psi_{\text{rem}}^{(P)}\|_{M,\Sigma_\rho} = 0.$$
The bound \eqref{eq:scatteringperturbedbound} from bounding
\begin{equation*}\|\psi\|_{n,\Sigma_\rho} \leq \|\psi - \psi_{\text{asymp}} - \psi_{\text{rem}}^{(P)}\|_{M,\Sigma_\rho}  + \|\psi_{\text{asymp}}\|_{n,\Sigma_\rho} + \|\psi_{\text{rem}}^{(P)}\|_{n,\Sigma_\rho}.
\end{equation*}
By \eqref{eq:asympboundpw1} we estimate $\|\psi_{\text{asymp}}\|_{n,\Sigma_\rho} \lesa  \|\psi_0\|_{H^k(\Sigma^+)} + \|\psi_3\|_{H^k(\Sigma^+)}$. To control $\|\psi_{\text{rem}}^{(P)}\|_{n,\Sigma_\rho}$, we apply Proposition~\ref{prop:weightedenergyperturbed} to $\psi_{\text{rem}}^{(P)}$ on the spacetime slab between $\Sigma_P$ and $\Sigma_\rho$, giving
$$\|\psi_{\text{rem}}^{(P)}\|_{n,\Sigma_\rho}^2 \lesa \rho^2 (\|\psi_0\|_{H^k(\Sigma^+)}^2 + \|\psi_3\|_{H^k(\Sigma^+)}^2)\int_{P}^\rho (\log s)^2\>\de s \lesa  \|\psi_0\|_{H^k(\Sigma^+)}^2 + \|\psi_3\|_{H^k(\Sigma^+)}^2,$$
where we used the fact that that $\|\cdot\|_{n,\Sigma_\rho} \sim \rho^2 \|\cdot\|_{M,\Sigma_\rho}$. This implies
$$\|\psi\|_{n,\Sigma_\rho} \leq \|\psi - \psi_{\text{asymp}} - \psi_{\text{rem}}^{(P)}\|_{M,\Sigma_\rho}  + \|\psi_0\|_{H^k(\Sigma^+)} + \|\psi_3\|_{H^k(\Sigma^+)}.$$
Taking the limit $P \ra 0$, the result now follows.
\end{proof}

\appendix
\section{Elliptic Estimates}\label{sec:elliptic}
In this appendix we will state an elliptic estimate which we use to prove the higher-order energy estimates in Section~\ref{sec:forwardasymptotics}. Given a 3-dimensional Riemannian manifold $(\Sigma,\ol{g})$, we denote the metric connection by $\ol{\nabla}$, and the Laplacian $\ol{\nabla}^i \ol{\nabla}_i = \ol{\Delta}$. We also write
$$|\ol{\nabla}^2\psi|^2 = \ol{g}^{ij}\ol{g}^{kl}\ol\nabla_i\ol\nabla_k \psi\ol\nabla_j\ol\nabla_l \psi.$$
We recall the following result from Section~4 of \cite{CK93}; see Remark 4 therein.
\begin{proposition}[Elliptic estimate from \cite{CK93}]
Let $\psi$ be a scalar function on the 3-dimensional Riemannian manifold $(\Sigma,\ol{g})$. Then there exists some positive constant $C$ such that
\begin{equation}\label{eq:ellipticestck93}
\int_{\Sigma} |\ol{\nabla}^2\psi|^2\de \mu_{\ol{g}} \leq C\Big(\int_{\Sigma}(\ol{\Delta}\psi)^2\de \mu_{\ol{g}} + \int_{\Sigma}|\emph{Ric}[\ol{g}]^{ij}\ol{\nabla}_i \psi \ol{\nabla}_j \psi|\de \mu_{\ol{g}}\Big).
\end{equation}
\end{proposition}
As a consequence, if we consider the spacelike hypersurfaces $\Sigma_r$ on Schwarzschild-de Sitter, the Ricci curvature of the induced metric $\ol{g}_r$ obeys the bound
$$|\text{Ric}[\ol{g}_r]^{ij}| \sim \frac{1}{r^4}.$$
for sufficiently large $r$. Moreover, one may bound
\begin{align}
|\ol{\nabla}^2\psi|^2 &\lesa \sum_{i,j=1}^3 \ol{g}_r^{ii}\ol{g}_r^{jj}(\pa_i\pa_j\psi)^2 + \sum_{i=1}^3 (\ol{g}_r^{ii})^2(\pa_i\psi)^2\nonumber\\*
&\sim \frac{1}{r^4}\sum_{i,j=1}^3 (\pa_i\pa_j\psi)^2 + \frac{1}{r^4}\sum_{i=1}^3 (\pa_i\psi)^2.\label{eq:ellipticest2}
\end{align}
We note that the final term in \eqref{eq:ellipticest2} can be controlled by the integrand of the final term in \eqref{eq:ellipticestck93}, which allows us to absorb this term into the right hand side of \eqref{eq:ellipticestck93}. Thus Proposition~\eqref{eq:ellipticestck93} implies the following Corollary:
\begin{corollary}\label{cor:ellipticestSdS}
Let $(\mc{M}_{\Lambda,m},g)$ be a subextremal member of the Schwarzschild-de Sitter family of spacetimes, and fix $r_0 \geq \rc$, and let $\psi$ be a scalar field on the expanding region $\mc{R}^+$. Then for all $r \geq r_0$, we have the following elliptic estimate:
\begin{equation}
\sum_{i,j=1}^3\int_{\Sigma_r} (\pa_i\pa_j\psi)^2\de \mu_{\ol{g}_r} \lesa_{\Lambda,m,r_0} C\Big(r^4\int_{\Sigma_r}(\ol{\Delta}\psi)^2\de \mu_{\ol{g}_r} + \sum_{i=1}^3\int_{\Sigma_r}(\pa_i\psi)^2\de \mu_{\ol{g}_r}\Big).
\end{equation}
\end{corollary}

\section{Sobolev Estimates}\label{sec:sobolev}
In this appendix we prove a Sobolev estimates that we use to derive pointwise estimates of solutions to the wave equation on Schwarzschild-de Sitter.
\begin{proposition}[Sobolev Estimate for Schwarzschild-de Sitter]\label{SobEst}
Let $(\mc{M}_{\Lambda,m},g)$ be a subextremal member of the Schwarzschild-de Sitter family of spacetimes, and fix $r_0 \geq \rc$. Recall the family of Killing vectors $T$, $\Omega_i:i=1,2,3$ that span the level sets $\Sigma_r$. Then there exists a constant $C = C(\Lambda,m) > 0$ such that for all $r \geq r_0$, we have
\begin{equation}\label{SobolevExact}
\sup_{\Sigma_r}|\varphi|^2 \leq \frac{C}{r^3}\int_{\Sigma_r}\Big\{\varphi^2 + (T\varphi)^2 + \sum_{i=1}^3|\Omega_i \varphi|^2 + \sum_{i=1}^3 |\Omega_i T \varphi|^2 + \sum_{i,j=1}^3|\Omega_i\Omega_j \psi|^2\Big\}\de \mu_{\ol{g}_r}.
\end{equation}
\begin{proof}
It was shown in \cite{Sch22} for a class of spacetimes that includes Schwarzschild-de Sitter, we have for any tensor field $\theta$ that
$$\sup_t \Big(\int_{S_t} r^4|\theta|^4 \de \mu_{\slashed{g}}\Big)^{1/4} \lesa_{\Lambda,m,r_0} \Big(\int_{\Sigma_r} |\theta|^2 + r^2|\overline{\nabla}\theta|^2\de\mu_{\ol{g}_r}\Big)^{1/2},$$
where $S_t$ is the sphere on $\Sigma_r$ on which the coordinate $t$ is constant, and $\de \mu_{\slashed{g}}$ is the induced volume form on $S_t$.
Moreover, in \cite{Sch19} it was shown that 
$$\sup_{S_t}|\varphi|^2 \lesa_{\Lambda,m,r_0} \frac{1}{r}\Big(\int_{S_t}\varphi^4 + r^4|\slashed\nabla \varphi|^4\de \mu_{\slashed{g}}\Big)^{1/2}.$$
Combining these two estimates gives
\begin{align*}
\sup_{\Sigma_r}|\varphi|^2 &\lesa \frac{1}{r}\sup_{t}\Big(\int_{S_t}\varphi^4 + r^4|\slashed\nabla \varphi|^4\de \mu_{\slashed{g}}\Big)^{1/2}\\*
&\lesa \frac{1}{r^3}\int_{\Sigma_r}\varphi^2 + r^2|\ol{\nabla}\varphi|^2 + r^4|\ol{\nabla}\slashed{\nabla}\varphi|^2 \de \mu_{\ol{g}_r}.
\end{align*}
We can bound the covariant derivative restricted to the sphere and the cylinder with spheres of radius $r$ like
\begin{align*}
r^2|\ol{\nabla}\varphi|^2 &\lesa_{\Lambda,m,r_0} |T\varphi|^2 + \sum_{i=1}^3 |\Omega_i \varphi|^2\\*
r^2|\slashed{\nabla}\varphi|^2 &\lesa_{\Lambda,m,r_0} \sum_{i=1}^3 |\Omega_i \varphi|^2,
\end{align*}
and therefore \eqref{SobolevExact} holds.
\end{proof}
\end{proposition}

\bibliographystyle{amsplain}
\bibliography{linearwavesonsds}

\providecommand{\bysame}{\leavevmode\hbox to3em{\hrulefill}\thinspace}
\providecommand{\MR}{\relax\ifhmode\unskip\space\fi MR }
\providecommand{\MRhref}[2]{%
  \href{http://www.ams.org/mathscinet-getitem?mr=#1}{#2}
}
\providecommand{\href}[2]{#2}
\begin{thebibliography}{10}

\bibitem{BW14}
Dean Baskin and Fang Wang, Comm. Math. Phys. \textbf{331} (2014), 477--506.

\bibitem{BH08}
Jean-Fran{\c c}ois Bony and Dietrich H{\"a}fner, \emph{Decay and non-decay of the local energy for the wave equation on the de {S}itter–{S}chwarzschild metric}, Commun. Math. Phys. \textbf{282} (2007), no.~3, 697--719.

\bibitem{CK93}
Demetrios Christodoulou and Sergiu Klainerman, \emph{The global nonlinear stability of the {M}inkowski space}, Princeton Mathematical Series, vol.~41, Princeton University Press, Princeton, NJ, 1993.

\bibitem{Cic23}
Serban Cicortas, \emph{Scattering for the wave equation on de {S}itter space in all even spatial dimensions}, ar{X}iv:2309.07342 (2023).

\bibitem{DHR13}
Mihalis Dafermos, Gustav Holzegel, and Igor Rodnianski, \emph{A scattering theory construction of dynamical vacuum black holes}, ar{X}iv:1306.5364 (2013).

\bibitem{DR07}
Mihalis Dafermos and Igor Rodnianski, \emph{The wave equation on {S}chwarzschild-de {S}itter spacetimes}, ar{X}iv:0709.2766 (2007).

\bibitem{DR09}
\bysame, \emph{The red-shift effect and radiation decay on black hole spacetimes}, Comm. Pure. Appl. Math. \textbf{7} (2009), 859--919.

\bibitem{DR13}
\bysame, \emph{Lectures on black holes and linear waves}, Evolution equations, Clay Mathematics Proceedings, vol.~17, 2013, (also arXiv:0811.0354), pp.~97--205.

\bibitem{DRSR18}
Mihalis Dafermos, Igor Rodnianski, and Yakov Shlapentokh-Rothman, \emph{A scattering theory for the wave equation on {K}err black hole exteriors}, Ann. Sci. {\'E}c Norm. Sup{\'e}r. \textbf{51} (2018), no.~2, 371–486.

\bibitem{Dy11}
Semyon Dyatlov, \emph{Quasi-normal modes and exponential energy decay for the {K}err-de {S}itter black hole}, Comm. Math. Phys. \textbf{306} (2011), no.~1, 119--163.

\bibitem{Fang23}
Allen~Juntao Fang, \emph{Nonlinear stability of the slowly-rotating {Kerr}-de {S}itter family}, ar{Xiv}:2112.07183 (2021).

\bibitem{FS20}
Grigorios Fournodavlos and Jan Sbierski, \emph{Generic blow-up results for the wave equation in the interior of a {S}chwarzschild black hole}, Arch. Ration. Mech. and Anal. \textbf{235} (2020), no.~2, 927–971.

\bibitem{Fr1}
F.~G. Friedlander, \emph{On the radiation field of pulse solutions of the wave equation}, Proc. Roy. Soc. London Ser. A \textbf{269} (1962), 53--65.

\bibitem{He21}
Lili He, \emph{Scattering from infinity of the {M}axwell {K}lein {G}ordon equations in {L}orenz gauge}, Commun. Math. Phys. \textbf{386} (2021), 1747–1801.

\bibitem{Hintz23}
Peter Hintz, \emph{Asymptotically de {S}itter metrics from scattering data in all dimensions}, Phil. Trans. Roy. Soc. Lond. A \textbf{382} (2024), no.~2267.

\bibitem{HV18}
Peter Hintz and Andr{\'a}s Vasy, \emph{The global non-linear stability of the {K}err-de {S}itter family of black holes}, Acta Math. \textbf{220} (2018), no.~1, 1--206.

\bibitem{HV20}
\bysame, \emph{Stability of {M}inkowski space and polyhomogeneity of the metric}, Ann. PDE \textbf{6} (2020), no.~2.

\bibitem{CNO19}
Pedro F.C.~Oliveira Jo{\~a}o~Costa, Jos{\'e}~Nat{\'a}rio, \emph{Decay of solutions of the wave equation in expanding cosmological spacetimes}, Journal of Hyperbolic Differential Equations \textbf{16} (2019), no.~01, 35--38.

\bibitem{KSR19}
Christoph Kehle and Yakov Shlapentokh-Rothman, \emph{A scattering theory for linear waves on the interior of {R}eissner–{N}ordstr{\"o}m black holes}, Ann. Henri Poincar{\'e} \textbf{20} (2019), no.~5, 1583--1650.

\bibitem{Kot18}
Friedrich Kottler, \emph{{\"U}ber die physikalischen {G}rundlagen der {E}insteinschen {G}ravitationstheorie}, Ann. Phys. \textbf{56} (1918), 401--462.

\bibitem{LR77}
Kayll Lake and R.C. Roeder, \emph{Effects of a nonvanishing cosmological constant on the spherically symmetric vacuum manifold}, Phys. Rev. D \textbf{15} (1977), no.~12, 3513--3519.

\bibitem{Lind17}
Hans Lindblad, \emph{On the asymptotic behavior of solutions to {E}instein's vacuum equations in wave coordinates}, Comm. Math. Phys. \textbf{353} (2017), 135--184.

\bibitem{LS23}
Hans Lindblad and Volker Schlue, \emph{Scattering from infinity for semilinear wave equations satisfying the null condition or the weak null condition}, Journal of Hyperbolic Differential Equations \textbf{20} (2023), no.~01, 155--218.

\bibitem{Ma23}
Georgios Mavrogiannis, \emph{{M}orawetz estimates without relative degeneration and exponential decay on {S}chwarzschild–de {S}itter spacetimes}, Ann. Henri Poincar{\'e'} \textbf{24} (2023), 3113–3152.

\bibitem{NR23}
Jose Nat\'ario and Flavio Rossetti, \emph{Explicit formulas and decay rates for the solution of the wave equation in cosmological spacetimes}, J.Math.Phys. \textbf{64} (2023), no.~3.

\bibitem{Sb15}
Jan Sbierski, \emph{Characterisation of the energy of {G}aussian beams on {L}orentzian manifolds with applications to black hole spacetimes}, Analysis {\&} PDE \textbf{8.6} (2015), 1379--1420.

\bibitem{Sch15}
Volker Schlue, \emph{Global results for linear waves on expanding {K}err and {S}chwarschild de {S}itter cosmologies}, Commun. Math. Phys. \textbf{334} (2015), 977--1023.

\bibitem{Sch19}
\bysame, \emph{Optical functions in de {S}itter}, J. Math. Phys. \textbf{62} (2021), no.~8.

\bibitem{Sch22}
\bysame, \emph{Decay of the {W}eyl curvature in expanding black hole cosmologies}, Ann. PDE \textbf{8} (2022), no.~9.

\bibitem{Vas10}
Andr{\'a}s Vasy, \emph{The wave equation on asymptotically de {S}itter-like spaces}, Adv. Math. \textbf{223} (2010), no.~1, 49--97.

\bibitem{Wey19}
Hermann Weyl, \emph{{\"U}ber die statischen kugelsymmetrischen {L}{\"o}sungen von {E}insteins kosmologischen {G}ravitationsgleichungen}, Phys. Z. \textbf{20} (1919), 31--34.

\bibitem{Yu22}
Dongxiao Yu, \emph{Nontrivial global solutions to some quasilinear wave equations in three space dimensions}, ar{X}iv:2204.12870 (2022).

\end{thebibliography}

\end{document}